\documentclass[english]{article}

\usepackage{amsthm,amsmath, amssymb}
\usepackage{url} 
\usepackage[authoryear]{natbib}

\usepackage[colorlinks=true,linkcolor=blue,citecolor=blue,urlcolor=blue,breaklinks,
bookmarks=false]{hyperref}
\usepackage{babel}
\usepackage{caption}
\usepackage{subcaption}
\usepackage{graphicx} 
\usepackage{color}
\usepackage{framed}
\usepackage{lscape}
\usepackage{rotating}
\usepackage{algorithm}
\usepackage[noend]{algpseudocode}
\usepackage{stefan_tex}
\usepackage{grffile}
\usepackage{multirow}

\newcommand{\ep}{\varepsilon}
\newcommand{\fg}{\mathcal{F}_\gamma}
\newcommand{\dfg}{\delta\mathcal{F}_\gamma}
\newcommand{\df}{\delta\mathcal{F}}
\newcommand{\I}{\mathcal{I}}
\newcommand{\la}{\langle}
\newcommand{\ra}{\rangle}

\newcommand{\ld}{\lim_{\ep\downarrow 0}}
\newcommand{\FW}{0.5\textwidth}
\newcommand{\TRA}{0}
\newcommand{\TRB}{0}
\newcommand{\TRC}{0}
\newcommand{\TRD}{0}

\makeatletter
\def\@footnotecolor{red}
\define@key{Hyp}{footnotecolor}{%
 \HyColor@HyperrefColor{#1}\@footnotecolor%
}
\def\@footnotemark{%
    \leavevmode
    \ifhmode\edef\@x@sf{\the\spacefactor}\nobreak\fi
    \stepcounter{Hfootnote}%
    \global\let\Hy@saved@currentHref\@currentHref
    \hyper@makecurrent{Hfootnote}%
    \global\let\Hy@footnote@currentHref\@currentHref
    \global\let\@currentHref\Hy@saved@currentHref
    \hyper@linkstart{footnote}{\Hy@footnote@currentHref}%
    \@makefnmark
    \hyper@linkend
    \ifhmode\spacefactor\@x@sf\fi
    \relax
  }%
\makeatother
\hypersetup{footnotecolor=blue}

\newtheorem{theorem}{Theorem}[section]

\newtheorem{lemma}[theorem]{Lemma}

\newtheorem{corollary}[theorem]{Corollary}
\newtheorem{proposition}[theorem]{Proposition}

\usepackage{geometry} 

\makeatletter
\def\BState{\State\hskip-\ALG@thistlm}
\makeatother

\makeatletter
\def\blfootnote{\gdef\@thefnmark{}\@footnotetext}
\makeatother

\begin{document}
\title{Sharp detection in PCA under correlations: all eigenvalues matter}

\author{Edgar Dobriban\footnote{E-mail: \texttt{dobriban@stanford.edu}. Supported in part by NSF grants DMS-1418362 and DMS-1407813, and by an HHMI International Student Research Fellowship.}}
\date{Stanford University}

\maketitle

\begin{abstract}
Principal component analysis (PCA) is a widely used method for dimension reduction.  In high dimensional data, the ``signal'' eigenvalues corresponding to weak principal components (PCs) do not necessarily separate from the bulk of the ``noise'' eigenvalues. Therefore, popular tests based on the largest eigenvalue have little power to detect weak PCs. In the special case of the spiked model, certain tests asymptotically equivalent to linear spectral statistics (LSS)---averaging effects over \emph{all} eigenvalues---were recently shown to achieve some power.  

We consider a nonparametric, non-Gaussian generalization of the spiked model to the setting of Marchenko and Pastur (1967). This allows a general bulk of the noise eigenvalues for flexible data modelling, accomodating correlated variables even under the null hypothesis of no significant PCs. 

We develop new tests based on LSS to detect weak PCs in this model. We show using the CLT for LSS that the \emph{optimal LSS} satisfy a Fredholm integral equation of the first kind. We develop algorithms to solve it, building on our recent method for computing the limit empirical spectrum. In contrast to the standard spiked model, we find that under ``widely spread'' null eigenvalue distributions, the new tests have a lot of power.
\end{abstract}


\section{Introduction}
\label{sec:intro}

Introduced by Pearson and Hotelling in the early 1900's, Principal Component Analysis (PCA) is a widely used statistical method for dimension reduction.  Inference in PCA is classically based on the asymptotic distribution of the top sample eigenvalues of the covariance matrix, which are consistent estimators of the top population eigenvalues under low-dimensional asymptotics---i.e., when the sample size grows while the dimension is fixed \citep{anderson1963asymptotic,anderson1958introduction}.

In contrast, in high dimensions---when the dimension is proportional to the sample size---the behavior of the eigenvalues is different. Below a critical value of the top eigenvalue in the population, the top sample eigenvalue has the same behavior as if there were only null eigenvalues, see e.g., \cite{baik2005phase,benaych2011eigenvalues} for results in this direction, and \cite{hachem2015survey} for a survey. In particular, the top eigenvalue does not separate from the bulk of the noise eigenvalues. Tests based on the top eigenvalue alone---despite their optimality in low dimensions---have small power to detect weak PCs in high dimensions. 

This raises several broad questions. Can we detect weak PCs in high-dimensional data even when the optimal low-dimensional tests fail? What statistical models are helpful to understand the problem? Can we find the optimal tests, perhaps restricted to certain classes? Can we characterize their performance?

To gain a deeper understanding of the problem, it is helpful to leverage results from random matrix theory, where the eigenvalues of large sample covariance matrices have been studied for nearly 50 years \citep{marchenko1967distribution}. There has been a lot of work on general nonparametric ensembles, where the unobserved population covariance matrix can be nearly arbitrary \citep[see e.g.,][for a reference]{bai2009spectral}. 

Despite this work, our current methods for detecting weak PCs are limited to a small number of covariance matrix models solved explicitly. These all center on the special case of the ``spiked model'', where the covariance matrix is a low rank perturbation of the identity \citep{johnstone2001distribution}.
For instance, \cite{onatski2013asymptotic,onatski2014signal} recently showed that in Gaussian spiked models, likelihood ratio tests have some power. 

Is it possible to detect weak PCs under the general covariance matrix models of \cite{marchenko1967distribution}? If so, what are the suitable methods, and what is their performance? This question is relevant for many applications, where the spiked model is not always a good description of empirical data (see Section  \ref{empirical_evidence} for a short review). The new methods are practically relevant, because tests assuming identity covariance---or ``sphericity''---may lose type I error control and lead to false discoveries in general models.

Working with the nonparametric Marchenko-Pastur models, however, poses several challenges. First, these models are characterized only implicitly by certain difficult fixed-point equations. While the theoretical existence of these equations---and of the associated ensembles---has been known for a long time, a reliable numerical approach has only recently been developed \citep{dobriban2015efficient}. This has enabled us to compute eigenvalue densities for examples never done before.  We will use here the same computational framework.

A second key challenge is that the pre-existing theoretical approach does not generalize directly. \cite{onatski2013asymptotic,onatski2014signal} work with the likelihoods of the eigenvalues in Gaussian spiked models---but in our non-Gaussian case these likelihoods do not exist. Even in the Gaussian case, the eigenvalue densities for general covariance matrices are much harder to work with than in the identity case \cite[e.g.,][]{muirhead2009aspects}. Therefore, a new theoretical approach is needed. 

In this paper we show how to detect weak PCs in certain nonparametric spiked models that generalize the standard one to the setting of  \cite{marchenko1967distribution}. We overcome the computational challenges by using the recently developed method and framework of \cite{dobriban2015efficient}. We overcome the theoretical challenges by directly working with a broad class of trace-like functionals of the covariance matrix, linear spectral statistics. Gaussian LR tests are a special case.

As a consequence of our results, quite generally \emph{all eigenvalues matter} to achieve sharp detection of weak PCs in high-dimensional data. We will see that tests based on top eigenvalue have little power, while our novel tests can have substantial power, especially when the null distribution of eigenvalues is ``widely spread''. This finding is in contrast to the low-dimensional case discussed above, as well as to the high-dimensional case with \emph{strong} PCs. In the latter, the top eigenvalues are not consistent estimates of their population counterparts, but they separate from the noise eigenvalues, and so can be detected with full power \citep[e.g.,][etc]{baik2005phase,paul2007asymptotics}. Thus, our results identify a special but broad regime where optimal inference must be based on all eigenvalues. 

\subsection{Our contributions}

To describe our results more concretely, suppose we have an $n \times p$ data matrix $X_{n\times p}$, with $n$ rows sampled from a $p$-dimensional population. The samples are allowed to have a general covariance structure, and have the distribution $\smash{X_i = \Sigma_p^{1/2}\ep_i}$ for white noise $\ep_i$ with iid real standardized entries. In the special case of the spiked model  \citep{johnstone2001distribution}, the null hypothesis is  that the  covariance matrix is spherical, $\Sigma_p = I_p$. This is a model for isotropic data varying equally in each spatial direction. The alternative hypothesis of interest in this case is that $\smash{\Sigma_p = I_p+ \sum_{j=1}^{k} h_j v_j v_j^{\top}}$, for orthonormal directions $v_j$ and scalars $h_j$. This allows for a greater variability in the directions of $v_j$. The problem is to test if there are any directions of variation with $h_j>0$.

We will study these questions under high-dimensional asymptotics, taking $n,p\to \infty$ such that $p/n\to\gamma>0$. In the standard spiked model, the top eigenvalue $\lambda_1$ of the sample covariance matrix $\smash{\hSigma = n^{-1} X_{n\times p}^\top X_{n\times p}}$ undergoes a phase transition. If $h_1>\sqrt{\gamma}$, $\lambda_1$ is asymptotically separated from the bulk of the noise eigenvalues---i.e., the other eigenvalues of $\smash{\hSigma}$---and detection is possible with full power. However, if $0\le h_1<\sqrt{\gamma}$, the top eigenvalue does not separate from the bulk \citep[e.g.,][etc]{baik2005phase,baik2006eigenvalues,paul2007asymptotics}. Therefore, tests based on it have trivial power.

\cite{onatski2013asymptotic,onatski2014signal} have recently discovered that despite the non-separation, weak PCs can be detected with nontrivial power by suitable likelihood ratio (LR) tests. One of their key observations is that the LR tests in Gaussian models are asymptotically equivalent to certain specific \emph{linear spectral statistics} or LSS. More generally, LSS are defined for all suitably smooth functions $\varphi$ as  $\smash{\tr(\varphi(\hSigma))}$$ = \sum_i \varphi(\lambda_i)$, where $\smash{\hSigma}$ is the sample covariance matrix and $\lambda_i$ are its eigenvalues. Notably, LSS aggregate effects over \emph{all} eigenvalues, unlike top eigenvalue based tests. 

Given this background, we can now state our contributions. 

\begin{enumerate}
\item We consider a hypothesis testing formulation for PCA in a  \emph{nonparametric} spiked model. This is a natural generalization of the standard spiked model of \cite{johnstone2001distribution} to the setting of \cite{marchenko1967distribution}. Our model allows for general distributions of PC variances---equivalently, of eigenvalues---under the null and alternative. In particular, the measured variables can be correlated even under the null. We model the distribution $H_p$ of eigenvalues as a mixture $(1-hp^{-1})H+h\,p^{-1}G_0$ of null eigenvalues $H$ and spikes $G_0$. The problem is to test for the presence of spikes. 

Motivated by the optimality of LSS in the standard Gaussian spiked model, we \emph{directly optimize} over LSS using the seminal CLT of \cite{bai2004clt}. This bypasses the difficulty that the density of eigenvalues is not available. 
We give an integral equation for the optimal LSS (Theorem \ref{thm:lss_part2}), and describe the maximum power (Theorem \ref{pow_lss}).  We show that the power is unity precisely if the equation is \emph{not} solvable. 

We show in simulations that there is a large power for spikes below the phase transition when the null $H$ is widely ``spread out'' (Sec. \ref{numer_res}). This is in contrast to the standard spiked model, where the power below the phase transition is small \citep{onatski2013asymptotic,onatski2014signal}. The larger power in our case is encouraging.

\item As an innovation in the proofs, we find the weak derivative of the Marchenko-Pastur forward map of the eigenvalues (Theorem \ref{weak_der}). This key new object allows us to compare the \emph{difference} in the distribution of the LSS under the null and alternative. 

The weak derivative proves to be a fruitful object of study, with interesting statistical consequences. By studying its structure---i.e., density and point masses---in Proposition \ref{prop_df}, we conclude that the asymptotic power of the optimal LSS is unity for spikes above the \emph{known} phase transition in existing spiked models \citep{baik2005phase,benaych2011eigenvalues,bai2012sample} (Theorem \ref{full_pow}). 
Finally, we also explain how the weak derivative sheds new light on the phase transition phenomenon. 

\item  We extend the whole framework to allow for an unknown scale factor of the PC variances. This development mirrors the extension from tests of identity---$\Sigma=I_p$---to tests of sphericity---$\Sigma=\sigma^2I_p$ for some unknown $\sigma^2$ in classical multivariate statistics \citep[e.g.][Ch. 10]{anderson1958introduction}. It allows flexibility, as only the general ``shape'' of the null must be specified, and not the scale.  

To allow for the unknown scale factor, we introduce and study the scale-invariant linear \emph{standardized} spectral statistics $\smash{\tr(\varphi(\hSigma/\hsigma^2))}$ = $\smash{\sum_{i=1}^{p} \varphi(\lambda_i/\hsigma^2)}$, where $\hsigma^2  = p^{-1}\tr\hSigma$. After establishing a CLT for them, we obtain results parallel to those for LSS. The results have some interesting consequences---for instance the classical LRT for sphericity behaves like the one for identity, despite their seemingly  different form. 

\item In addition to finding the optimal tests among LSS, we take a broader perspective that underscores their ubiquity in multivariate analysis.  We study both classical and new tests of sphericity---LR tests, the popular tests of  \cite{john1971optimal, ledoit2002some} and the new tests of \cite{fisher2010new,choi2015regularized}---and show that they are all asymptotically equivalent to certain LSS in our nonparametric models. 

While tests of sphericity were not classically developed for PCA, our analysis shows that they do in fact have some power to detect PCs in high-dimensional spiked models. More broadly, these results complement our main optimality theorems, arguing that LSS are a helpful unifying notion in multivariate analysis in high dimensions.  

\item  We develop an efficient algorithm for our method (Sec. \ref{Implementation}), based on the computational framework of \cite{dobriban2015efficient}, and on methods for solving linear integral equations.
Software implementing our methods and for reproducing our computational results is available at \url{github.com/dobriban}.  We also give some empirical motivation by reviewing literature from genomics and finance, and by an empirical data example (Sec. \ref{empirical_evidence}). 

\end{enumerate}

\subsection{Related work}
In addition to the already mentioned work, there are many interesting results on PCA in high dimensions. For general reviews on this and related topics in random matrix theory, we refer to \cite{johnstone2007high,couillet2011random, paul2014random,yao2015large}. There are at least two broad lines of work on testing in high-dimensional PCA connected to our results. 
The first links to tests of sphericity against low-rank alternatives and to spiked models \citep{johnstone2001distribution,onatski2013asymptotic,onatski2014signal, wang2013sphericity, wang2014note,choi2015regularized,dharmawansa2014local,johnstone2015testing}. The second generally studies strong PCs, allowing for correlated residuals \citep[e.g.,][]{bai2002determining,bai2008large,onatski2009testing,ahn2013eigenvalue}. These and other results are reviewed in Sections \ref{rel_work} and \ref{lss_examples}  after  our main results. 

\section{Sharp detection in PCA}
\label{olss}

We now set the stage to present our results. Suppose we observe an $n \times p$ data matrix $X_{n \times p}$, where $n$ is the sample size and $p$ is the dimensionality. If the samples are drawn independently from a population with covariance matrix $\Sigma_{p}$, then one can model $\smash{X_{n \times p} = Z_{n \times p} \Sigma_{p}^{1/2}}$, where the $n \times p$ matrix $Z_{n\times p}$ has iid standardized entries, and $\Sigma_{p}$ is a $p \times p$ deterministic positive semi-definite population covariance matrix. Let $H_p$ be the spectral distribution of $\Sigma_{p}$, i.e., the discrete uniform distribution on its eigenvalues $l_i$, sorted so that $l_1\ge l_2 \ge \ldots \ge l_p$.  Its cumulative distribution function is defined as $\smash{H_p(x) = p^{-1}\sum_{i=1}^{p}I(l_i \le x)}$. In our context, $l_i$ are the population variances of the principal components.

 The null hypothesis of sphericity $\Sigma_p = \sigma^2 I_p$ is equivalent to $H_p = \delta_{\sigma^2}$, for an unknown $\sigma^2>0$, where $\delta_c$ is the point mass at $c$. The alternative hypothesis in the spiked model $\smash{\Sigma_p = \sigma^2 I_p +}$ $\smash{ \sum_{j=1}^{k} h_j v_j v_j^{\top}}$, for orthonormal $v_j$, is equivalent to $H_p = (1-k/p)\delta_{\sigma^2}+p^{-1}\sum_{j=1}^{k} \delta_{\sigma^2+h_j}$. This expresses the sphericity null and spiked alternative in terms of the spectral distribution of $\Sigma_p$. The test of identity $H_0: \Sigma_p = I_p$ against low rank alternatives can handled similarly. 

We consider a more general \emph{nonparametric spiked model}. Let $\smash{H = d^{-1}\sum_{i=1}^{d}\delta_{t_i}}$ and $\smash{G_j =}$ $\smash{ h^{-1}\sum_{i=1}^{h}\delta_{s_i^j}}$, $j=0,1$ be fixed probability distributions on $[0,\infty)$. Under the null, we take the eigenvalues to be $t_1,t_2,\ldots,t_d$ each with multiplicity $m$, and $s_1^0, s_2^0,\ldots,s_h^0$. Under the alternative, the eigenvalues are $t_i$ with the same multiplicity, and $s_1^1, s_2^1,\ldots,s_h^1$. Therefore, the total number of eigenvalues is $p = dm + h$, and $h$ of them differ between the null and alternative. Without loss of generality, we can take $p\to\infty$ along such a subsequence (as $m\to\infty$).
 
We can write this sequence of null hypotheses $H_{p,0}$ and alternatives $H_{p,1}$ as
\begin{align}
\label{null}
H_{p,0}:&\, H_p = (1-hp^{-1})H+h\,p^{-1}G_0,\\
\label{altve}
H_{p,1}:&\, H_p = (1-hp^{-1})H+h\,p^{-1}G_1.
\end{align}

Taking $H=\delta_1$, $G_0=\delta_1$, and $G_1 =h^{-1}\sum_{j=1}^{h} \delta_{1+h_j}$, the above null generalizes the hypothesis of identity $H_p = \delta_1$ against spiked alternatives. We will first focus on the identity test, and then extend the whole methodology to testing sphericity in Section \ref{sphericity}. Similar---but different---spiked models have appeared in \cite{nadler2008finite,benaych2011eigenvalues,bai2012sample}.

An analogy to classical asymptotic statistics helps explain the scaling of the problem. In classical statistics, fixed-dimensional distributions $P_\theta$ are tested against sequences $\smash{P_{\theta + hN^{-1/2}}}$ based on $N$ iid observations \citep{van1998asymptotic, lehmann2005testing}. These local alternatives are scaled at the $\sqrt{N}$-level. In our setting the dimension $p$ will grow proportionally to $n$, creating $N=np$ effective sources of randomness. Therefore, heuristically the right rate for local alternatives is $\sqrt{np}\sim p$. Furthermore, building on this analogy, we will call $h$ the \emph{local parameter}.

While in some cases a null hypothesis for the eigenvalues may be known from prior work, in many cases the null is not known, and must be estimated. The solution for known $H_p$ is an important step toward the setting of unknown $H_p$. We will discuss this in Section \ref{unknown_null}. 

We will construct tests based on linear spectral statistics (LSS) $\smash{T_p(\varphi) = \tr(\varphi(\hSigma))}$ = $\smash{\sum_{i=1}^{p} \varphi(\lambda_i)}$, where $\smash{\hSigma = n^{-1} X_{n \times p}^\top X_{n \times p}}$ is the sample covariance matrix, and $\lambda_i$ are its eigenvalues. We will find the \emph{optimal LSS} for the hypothesis testing problem \eqref{null} vs \eqref{altve}, when the sample size $n$ and dimension $p$ grow such that $\gamma_p = p/n \to \gamma >0$. 
In fact we will assume that $\gamma_p=\gamma$, which imposes the extra condition that  $\gamma$ must be rational and $n = (dm+h)/\gamma$ must belong to the integers for infinitely many $m\in \mathbb{N}$. However, this is not a limitation, because in practice we always have \emph{finite} $n,p$, and we can set $\gamma:=p/n$ to use our methods.

In this model, the Marchenko-Pastur \emph{forward map}---or simply \emph{Marchenko-Pastur map}---describes the  spectral distribution $F_p$ of $\smash{\hSigma}$. If the entries of $Z_{n\times p}$ come from an infinite array  of iif variables with mean zero and variance 1, and $H_p \Rightarrow H$ weakly, then with probability 1, $F_p \Rightarrow \fg(H)$ for a probability measure $ \fg(H)$ \citep{marchenko1967distribution,bai2009spectral}. We will assume $H\neq\delta_0$. 
An example of this model is the autoregressive covariance matrix of order 1, where the entries of $\Sigma_p$ are $\Sigma_{p}[i,j] = \rho^{|i-j|}$, $\rho \in (0,1)$; for other examples, see for instance \cite{dobriban2015high}.

The Marchenko-Pastur map $\fg$ has a smoothing effect: for any $H$, $ \fg(H)$ has a continuous density for all $x\neq 0$, and also for $x=0$ if $\gamma<1$ \citep{silverstein1995analysis}. If $\gamma>1$, the so-called companion empirical spectral distribution (ESD) $\smash{\underline{F}}$, defined by $\smash{\underline{F} = \gamma \fg(H) + (1-\gamma) I_{[0,  \infty)}}$ has a density at zero; we will find it convenient to work with this distribution. The companion ESD is the limit of the spectral distribution of the matrix $\smash{\underline \hSigma = n^{-1} X_{n \times p} X_{n \times p}^\top}$. 

The asymptotic distribution of the LSS is also known for smooth functions.  Let $\I = [a,b]$ be a compact interval whose interior includes $[\liminf l_p(\Sigma_p) I(\gamma\in(0,1)) (1-\sqrt{\gamma})^2, \limsup l_1(\Sigma_p) (1+\sqrt{\gamma})^2]$ for both null and alternative $\Sigma_p$ sequences, where we assume $l_1(\Sigma_p)$ is uniformly bounded above. This interval includes the support of the limiting ESD $\mathcal F_\gamma(H)$ \citep{bai2009spectral}.
Let $\mathcal{H}(\I)$ be the set of complex analytic functions on some open domain of $\mathbb{C}$ containing $\I$, and let $f \in \mathcal{H}(\I)$. Suppose that the iid real standardized random variables $Z_{n\times p}[i,j]=Z[i,j]$ come from an infinite array, with $\EE{Z[i,j]^4}=3$. 

The CLT for linear spectral statistics of \cite{bai2004clt} implies that the centered test statistics converge weakly: $\smash{T_p(\varphi) - p\int \varphi(x) d \mathcal{F}_{\gamma_p}(H_p)}$ $\smash{ \Rightarrow \mathcal{N}(m_\varphi,\sigma_{\varphi}^2)}$ under the null and alternative, for a certain mean $m_\varphi$ and variance $\sigma_{\varphi}^2$. The limit parameters depend on $H$, $G_i$ and $\gamma$. We focus on variables whose fourth moment matches the Gaussian distribution, but a similar approach should work for more generally, using the CLT of \cite{zheng2015substitution}.

Recall that the Stieltjes transform of a signed measure $\mu$ on $[0,\infty)$ is defined as the map $m: \mathbb{C}\setminus[0,\infty) \to \mathbb{C}$, $m(z) = \int (x-z)^{-1} d\mu(x)$. Let $v(z) = v_\gamma(z;H)$ be the Stieltjes transform of the companion ESD $\underline F$. The limit $v(x) = \lim_{z\to x} v(z)$ exists for all $x \in \mathbb{R}\setminus\{0\}$ \citep{silverstein1995analysis}. 
We will also need the kernel (well-defined a.s. with respect to Lebesgue measure on $\mathbb{R}$)
\begin{equation}
\label{kernel}
k(x,y) = k_\gamma(x,y;H)  =  \frac{1}{2\pi^2}\log\left( 1 + 4\frac{\Im (v(x)) \, \Im (v(y))}{|v(x) - v(y)|^2}\right). 
\end{equation}

Note that $k\neq0$ only within the support of $\fg(H)$. Since $k$ is a logarithmically weakly singular kernel \citep[][p. 564]{bai2004clt}, it induces a compact linear integral operator $K = K_{\gamma,H}$ as a map $K: L^2[\mathcal I] \to L^2[\mathcal I]$ in the usual way:  $K(\varphi)(x) = \int_{\I} k(x,y) \varphi(y) dy$ \citep[see][p. 29 and 62, for this property]{kress2013linear}. We write $\text{Im}(K) = \{Kl: l\in L^2(\I)\}$ for the image of the linear operator $K$, and $\la \cdot,\cdot \ra $ for the inner product on $L^2[\mathcal I]$. The generalized inverse $K^+$ of $K$ is the linear operator which assigns to each $\Delta \in \textnormal{Im}(K)$ the minimum norm solution to the equation $Kl =\Delta$ \citep[see e.g.,][p. 115]{groetsch1977generalized}.

\subsection{Main results}
\label{main_res}

In the above model, the optimal LSS depends on the \emph{weak derivative} $\dfg$ of the Marchenko-Pastur map. For two probability measures $H,G$ we define this as the signed measure arising in the weak limit 
\begin{equation}
\label{dfg}
\dfg(H,G) = \lim_{\ep\to 0} \frac{\fg((1-\ep)H+\ep G)-\fg(H)}{\ep}
\end{equation}
We will show in Theorem \ref{weak_der} that the limit is well defined. 
To find the optimal LSS we will first give an asymptotically equivalent normal test for \emph{fixed} LSS. 

\begin{theorem}[Asymptotically Equivalent Normal Test]
\label{thm:lss}
Consider the problem of testing for weak PCs in the nonparametric spiked model \eqref{null} vs \eqref{altve}. For each $\varphi\in \mathcal{H}(\I)$, there is a sequence of constants $c_p$ such that under the null $H_{p,0}$, one has $T_p(\varphi) -c_p \Rightarrow \mathcal{N}(0,\sigma_{\varphi}^2)$, while under the alternative $H_{p,1}$, one has $T_p(\varphi) -c_p \Rightarrow \mathcal{N}(\mu_{\varphi},\sigma_{\varphi}^2)$.

The mean and variance are 
\begin{align}
\label{mu_f}
\mu_{\varphi} & =  - h\int_{\I} \varphi'(x) \Delta(x) dx \,\,\textnormal{ and } \\
\label{sigma_f}
\sigma^2_{\varphi} & =\int_{\I}\int_{\I} \varphi'(x)\varphi'(y) k(x,y) dx\, dy.
\end{align}

Here $\Delta$ denotes the difference between the distribution functions of the weak derivatives $\dfg(H,G_i)$, and $k$ denotes the kernel defined in \eqref{kernel}.
\end{theorem}

The proofs of the results in this section are outlined in Section \ref{mainproofs}. Therefore, using the linear spectral statistic $T_p(\varphi)$ is asymptotically equivalent to a hypothesis test of a distribution $\mathcal{N}(0,\sigma_{\varphi}^2)$ against $\mathcal{N}(\mu_{\varphi},\sigma_{\varphi}^2)$. The next step is to optimize over LSS $\varphi$. In analogy to the asymptotic theory of optimal testing in iid models, we will call $\theta(\varphi)=\mu_{\varphi}/\sigma_{\varphi}$ the \emph{efficacy} of a test sequence $T_p(\varphi)$ \citep[][p. 536]{lehmann2005testing}. If $\sigma_{\varphi}=0$ while $\mu_{\varphi}\neq0$, we define $\theta(\varphi)=+\infty$, because the efficacy in distinguishing $\mathcal{N}(0,\sigma_{\varphi}^2)$ from $\mathcal{N}(\mu_{\varphi},\sigma_{\varphi}^2)$ is infinite. Similarly, if $\sigma_{\varphi}=0$ while $\mu_{\varphi}=0$, define $\theta(\varphi)=0$. With these definitions, one does not have to worry about dividing by 0. 

 We will maximize the efficacy over certain function classes $\mathcal{X}$: 
\begin{equation}
\label{opt_lss_x}
\sup_{\varphi\in\mathcal{X}} \frac{\mu_{\varphi}}{\sigma_{\varphi}}.
\end{equation}
The value of the optimization problem will be called the \emph{efficacy over $\mathcal{X}$}, and will be denoted $\theta^*(\mathcal{X})$. A function $\varphi\in\mathcal{X}$ achieving this value will be called an \emph{optimal LSS} over $\mathcal{X}$.  Due to the quadratic nature of the the objective, it will be easier first to optimize over the space $\mathcal{W}(\I)=\{\varphi:\I\to\mathbb{R}$ $: \varphi'(x) \text{ exists}$ $ \text{for almost every } x\in\I; \text{ and } \varphi' \in L^2[\mathcal I]\}$, using Hilbert space techniques.  

\begin{theorem}[Optimal Linear Spectral Statistics over $\mathcal{W}(\I)$]
\label{thm:lss_part2}
Consider the optimization of the efficacy over $\mathcal{W}(\I)$. The following dichotomy arises:

\begin{enumerate}
\item If $\Delta\in\textnormal{Im}(K)$, then the efficacy over $\mathcal{W}(\I)$ equals $h\cdot\la \Delta,K^{+}\Delta \ra^{1/2}$ $ < \infty$. The optimal linear spectral statistics over $\mathcal{W}(\I)$ are given by a Fredholm integral equation of the first kind for their derivatives: 
\begin{equation}
\label{opt_lss}
K(\varphi') = -\eta \Delta,
\end{equation}
where $\eta>0$ is any constant. 

\item On the other hand, if $\Delta \notin \textnormal{Im}(K)$, then the efficacy over $\mathcal{W}(\I)$ equals $+\infty$. The optimal LSS are all functions $\varphi\in\mathcal{W}(\I)$ with $K(\varphi')=0$ and $\la \Delta, \varphi' \ra<0$.
\end{enumerate}

\end{theorem}

This gives an equation for the optimal LSS, which we call the  \emph{optimal LSS equation}. Since the equation does not depend on $h$, the optimal LSS is \emph{uniformly optimal} against all $h>0$.  If the equation is not solvable in $L^2(\I)$, we will construct a sequence of functions $\varphi_n \in \mathcal{W}(\I)$ with efficacies $\theta(\varphi_n)\to\infty$, concluding that the supremum of asymptotic power over $\mathcal{W}(\I)$ is unity. 

We now return to smooth LSS. While the solution of the optimal LSS may not be an analytic function, we will show that analytic functions in $\mathcal{H}(\I)$ have the same maximum power as functions in $\mathcal{W}(\I)$. Denoting the centered test statistics $\smash{\tilde T_p(\varphi) = T_p(\varphi)-}$ $\smash{p\int \varphi(x) d \mathcal{F}_{\gamma_p}(H_p)}-m_\varphi$, we consider two-sided testing procedures that reject $H_{p,0}$ if $\smash{\tilde T_p(\varphi) \notin [t^-_{\varphi},t^+_{\varphi}]}$ for some constants $t^-_{\varphi}<t^+_{\varphi}$. Our goal is to optimize over smooth functions $\varphi\in \mathcal{H}(\I)$ and the critical values $\smash{t^-_{\varphi}<t^+_{\varphi}}$. The maximal asymptotic power is defined as
$$
\beta = \sup_{\varphi\in \mathcal{H}(\I),\, t^-_{\varphi}<t^+_{\varphi}} \,\, \lim_{p\to\infty} \mathbb{P}_{H_{p,1}}\left(\tilde T_p(\varphi) \notin [t^-_{\varphi},t^+_{\varphi}]\right).
$$

We find an expression for the power, depending on the null, the spikes, and the local parameter.  
\begin{theorem}[Asymptotic power] 
\label{pow_lss}
Among tests based on linear spectral statistics $T_p(\varphi)$ for $\varphi\in \mathcal{H}(\I)$ with asymptotic level $\alpha\in(0,1)$,
the maximal asymptotic power is 
$$
\beta=
\left\{
	\begin{array}{ll}
		\Phi\left(z_{\alpha}+h\, \la \Delta, K^{+}\Delta\ra^{1/2} \right)  & \mbox{\, if \, }  \Delta\in\textnormal{Im}(K), \\
		1 & \mbox{\, if \, }\Delta\notin\textnormal{Im}(K).
	\end{array}
\right.
$$
Here $\Delta=\dfg(H,G_1)-\dfg(H,G_0)$ is the difference of the weak derivatives,
while $K$ is the compact operator induced by the kernel \eqref{kernel}, and $K^+$ is the pseudoinverse of $K$.
\end{theorem}

This shows that there are two possibilities, depending on the relation between the null and the alternative. If $\Delta\in\textnormal{Im}(K)$, the asymptotic power depends on the norm of $\Delta$ via $\la \Delta, K^{+}\Delta\ra^{1/2} $. This is reasonable, as a ``larger'' derivative $\Delta$ perturbs the null more, and should be easier to detect. A larger local parameter $h>0$ also leads to more power, as there are more spikes. 

The second case, $\Delta\notin\textnormal{Im}(K)$, can occur---for instance---if the alternative sample spikes separate from the bulk. In certain spiked models, the existence of a threshold beyond which the top eigenvalue separates from the bulk---a phase transition phenomenon---was established for complex-valued Gaussian white noise in \cite{baik2005phase}, and for correlated noise in \cite{benaych2011eigenvalues, bai2012sample} \cite[see also][Chapter 11]{yao2015large}. While the models differ slightly between the authors, the location of the phase transition is the same. 

For large spikes we will show in Section \ref{sec:full_pow} that the weak derivative $\dfg$ has mass outside of the support $S$ of $\fg(H)$. Hence the distribution function $\Delta$ is not in the image of $K$, which is supported on $S$. In conclusion, there is full power above the phase transition (Section \ref{sec:full_pow}). 

Intuitively, $\Delta\in\textnormal{Im}(K)$ should correspond to the spikes being below the phase transition. Indeed, in this case $L$ is supported within $S$. However, it is not clear that $L$ actually belongs to the image of the compact operator $K$. Showing this would require a more detailed, and perhaps challenging, operator-analytic study of $K$. We leave this interesting work for future research.

\subsection{Examples of optimal LSS; Numerical results}
\label{ex_plots}

\subsubsection{Standard spiked model}

We take a detour to illustrate the optimal LSS in two simple cases. First, in the ``standard spiked model'' introduced in \cite{johnstone2001distribution}, the null is specified by $H=\delta_1$ and $G_0=\delta_1$, while the alternative has $G_1 = \delta_t$. We take the aspect ratio $\gamma=1/2$. The well known BBP phase transition \citep{baik2005phase} states that for a ``subcritical'' spike $t$ below the ``phase transition'' (PT) threshold $1+\sqrt{\gamma} \approx 1.7$, the corresponding ``sample spike'' moves to the top of the bulk spectrum. For a ``supercritical'' spike $t$ above the PT threshold, the sample spike moves to a value $z(t) = t[1 + \gamma/(t-1)]$ above the bulk edge. 

In a Gaussian model, \cite{onatski2013asymptotic}  (OMH) showed that the LR test has nontrivial power below the PT. Moreover, the LR test asymptotically equivalent to the LSS with $f(x) = -\log(z(t)-x)$, which we call the ``OMH LSS''. It is also known that above the PT the Tracy-Widom test based on the top eigenvalue has asymptotically full power.

\begin{figure}
\centering
\begin{minipage}{.5\textwidth}
  \centering
  \includegraphics[scale=0.33]
  {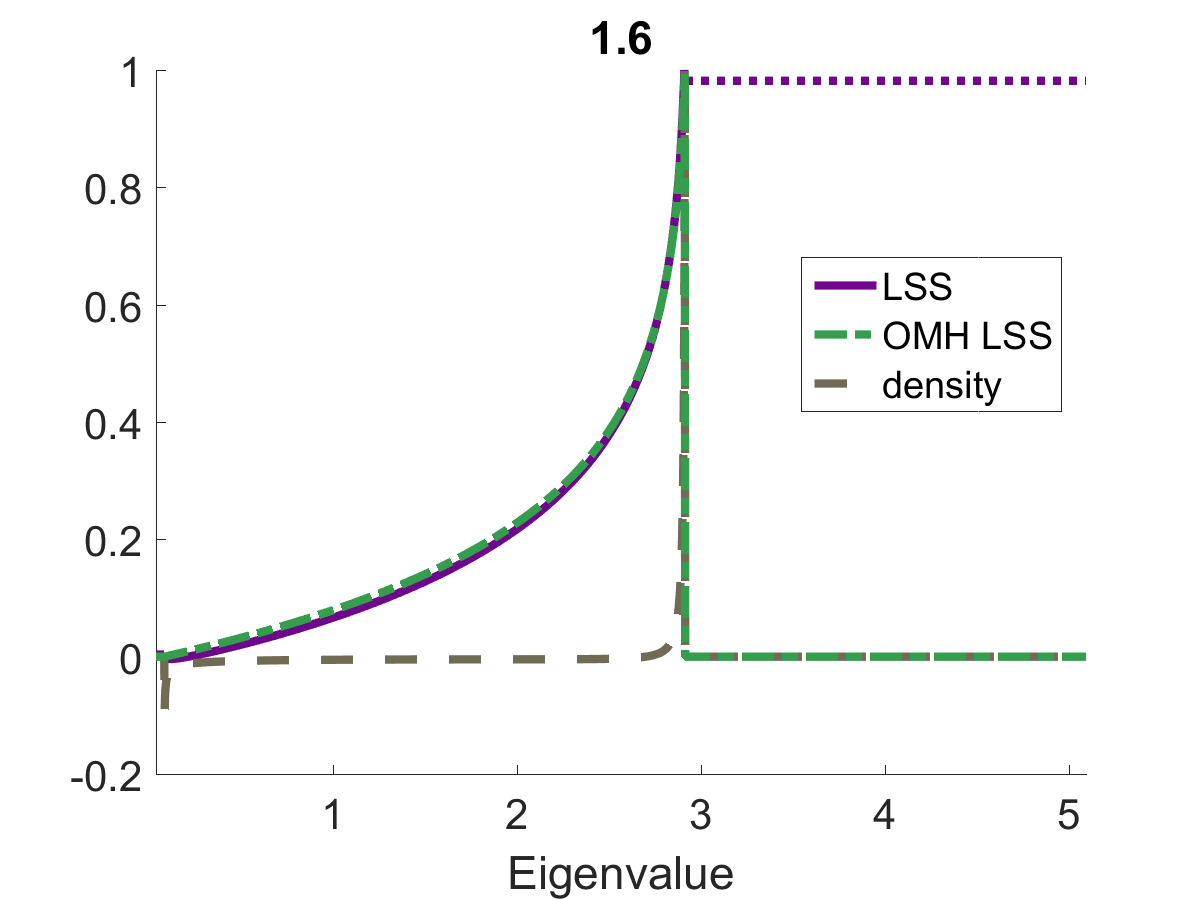}

\end{minipage}%
\begin{minipage}{.5\textwidth}
  \centering
  \includegraphics[scale=0.33]{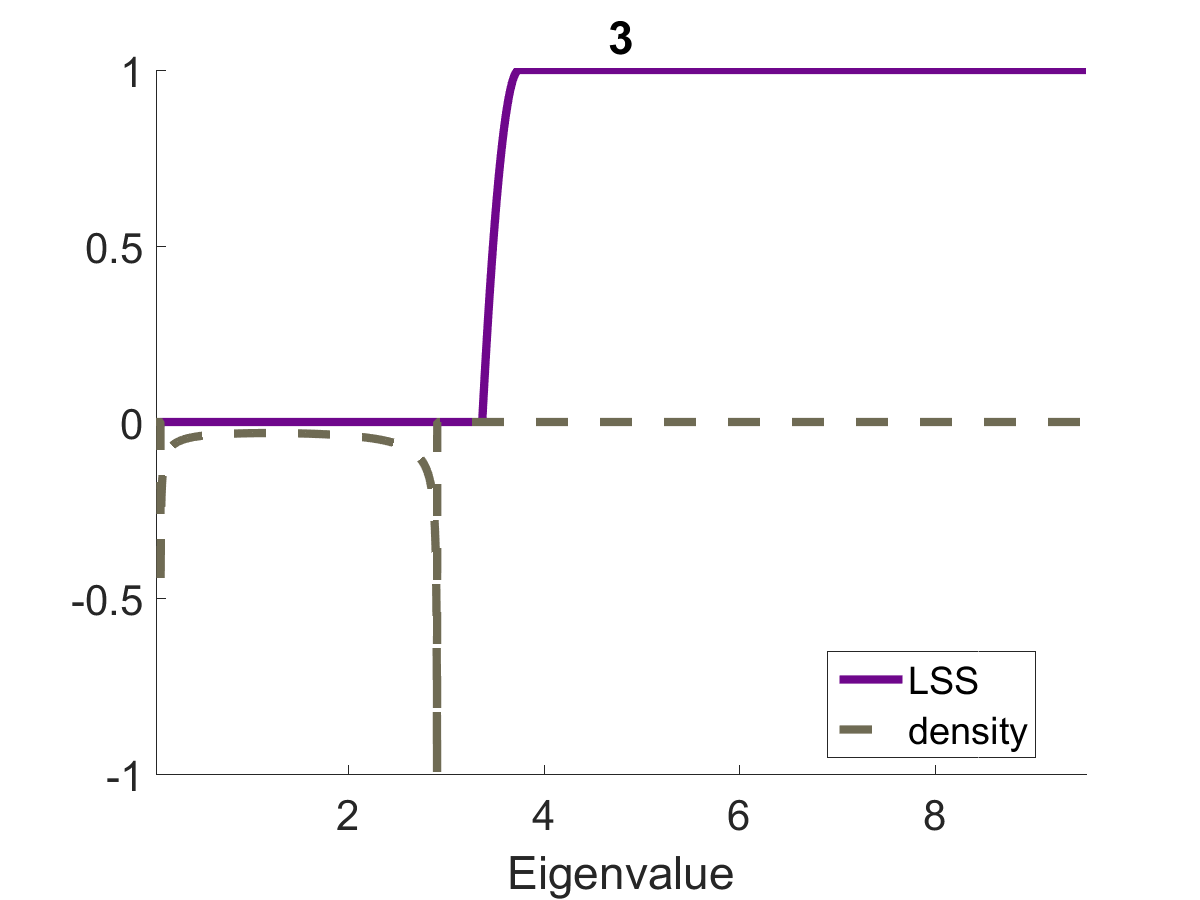}
\end{minipage}
\caption{Optimal LSS and density of $\dfg(H,G_1)$ with $H=\delta_1$, $G_0 = \delta_1$, $G_1 = \delta_t$, $\gamma=1/2$. On the left the spike $t=1.6$ is below the phase transition, while on the right $t=3$ is above the phase transition. On the left figure, the LSS equivalent to the LRT from OMH is also plotted, and agrees with our LSS. 
}
\label{fig1}
\end{figure}

With these preparations, we show the density of the weak derivative $\dfg(H,G_1)$, the pointwise values of our optimal LSS, and the OMH LSS (Fig.  \ref{fig1}).  They are normalized to have maximum absolute value equal to unity. On the left plot, the spike $t=1.6$ is below the PT, while on the right $t=3$ is above the PT. 

We observe the following:
\begin{enumerate}
\item \emph{The density of $\dfg(H,G_1)$}: The density of the weak derivative exists within the support of the Marchenko-Pastur bulk $[(1-\sqrt{\gamma})^2,(1+\sqrt{\gamma})^2]$. In the subcritical case, we will show later that $\dfg(H,G_1)$ is supported on the same set as the bulk (see Proposition \ref{prop_df}). Furthermore we see that it has a positive singularity at the right edge, and a negative singularity at the left edge. This shows that the perturbation by the spike $t$ affects the \emph{whole} bulk, and the effect is strongest at the two edges. Since $\smash{t>1}$ and the sample spike moves to the right edge, it makes sense that the perturbation ``moves mass'' from towards the right edge. No mass is moved outside the bulk, consistent with the classical spiked model \citep{baik2005phase}. 

In the supercritical case, we will show later in Proposition \ref{prop_df} that $\dfg(H,G_1)$ has a point mass at $z(t)$. Now the density is negative throughout the bulk, showing that the perturbation moves mass away. 

\item \emph{The LSS}: In the subcritical case, our optimal LSS agrees with the Onatski-Moreira-Hallin LSS  \citep{onatski2013asymptotic} within numerical precision. This confirms that we recover their methods as a special case. It is reassuring that we match the state of the art method for this special case, given that our approach is very different. 

Our theory only specifies the optimal LSS within the support of the Marchenko-Pastur map---and we extend it as a constant to the complement, see Section \ref{comput}. This is illustrated by the dotted line.

For a supercritical spike there is more latitude in the choice of the optimal LSS. Here we set it equal to 0 on the support of the bulk and equal to unity at and above the location of the sample spike $z(t)$, interpolating by an Epanechnikov kernel (see Section \ref{comput}).

\end{enumerate}

\subsubsection{Nonparametric spiked model}

Next we consider an example where the null hypothesis is a non-identity distribution for the population PC variances. We let  $H=2^{-1}(\delta_1+\delta_3)$, and $G_0=H$, corresponding to a mixture of two distinct PC variances. In this background noise, we want to test for the presence of a PC with magnitude $t$, corresponding to $G_1 = \delta_t$. 

We show the density of $\dfg(H,G_1)$, and the optimal LSS for $\gamma=1/2$ (Fig.  \ref{fig2}) and $\gamma=1/10$ (Fig.  \ref{fig3}). We consider two values for $t$, 0.8 and 3.6, both of which turn out to be subcritical. 

\begin{figure}
\centering
\begin{minipage}{.5\textwidth}
  \centering
  \includegraphics[scale=0.33]{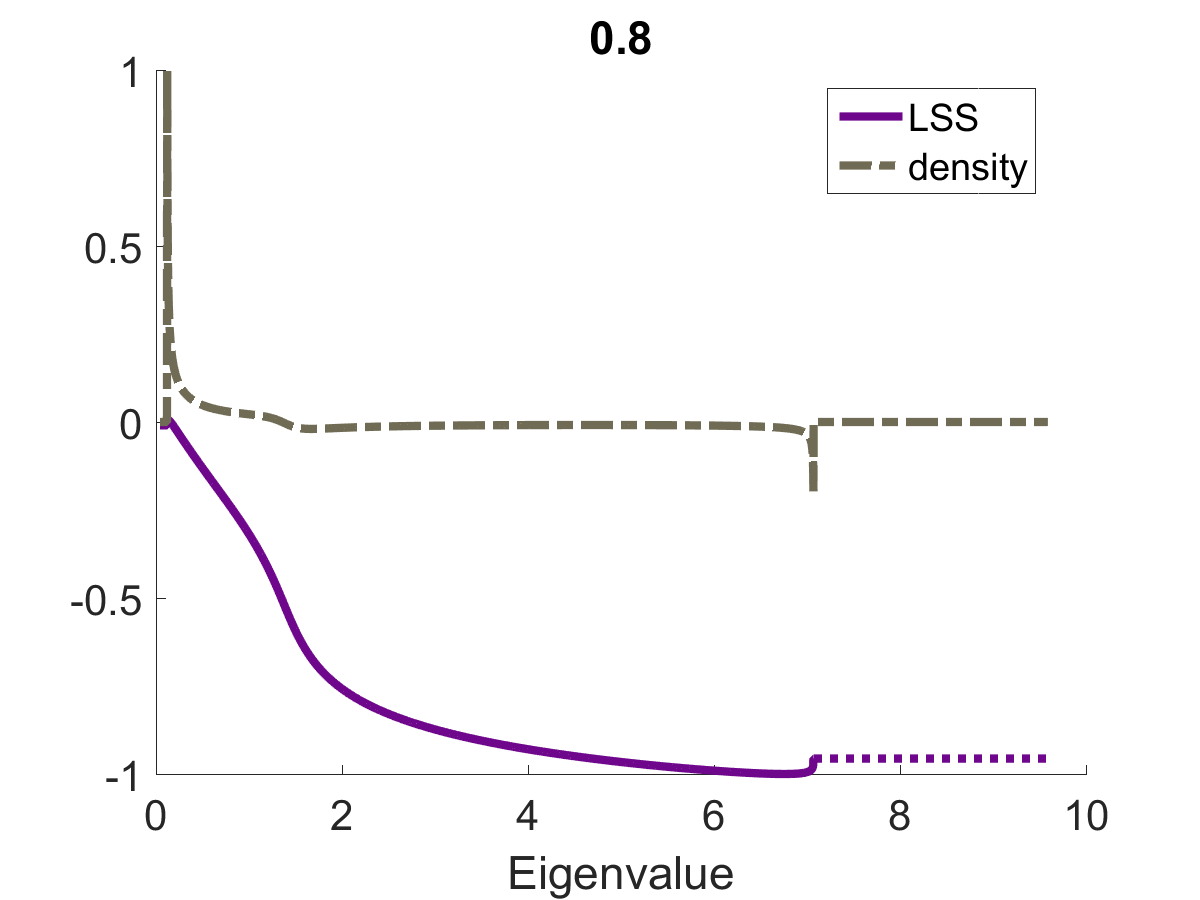}
\end{minipage}%
\begin{minipage}{.5\textwidth}
  \centering
  \includegraphics[scale=0.33]{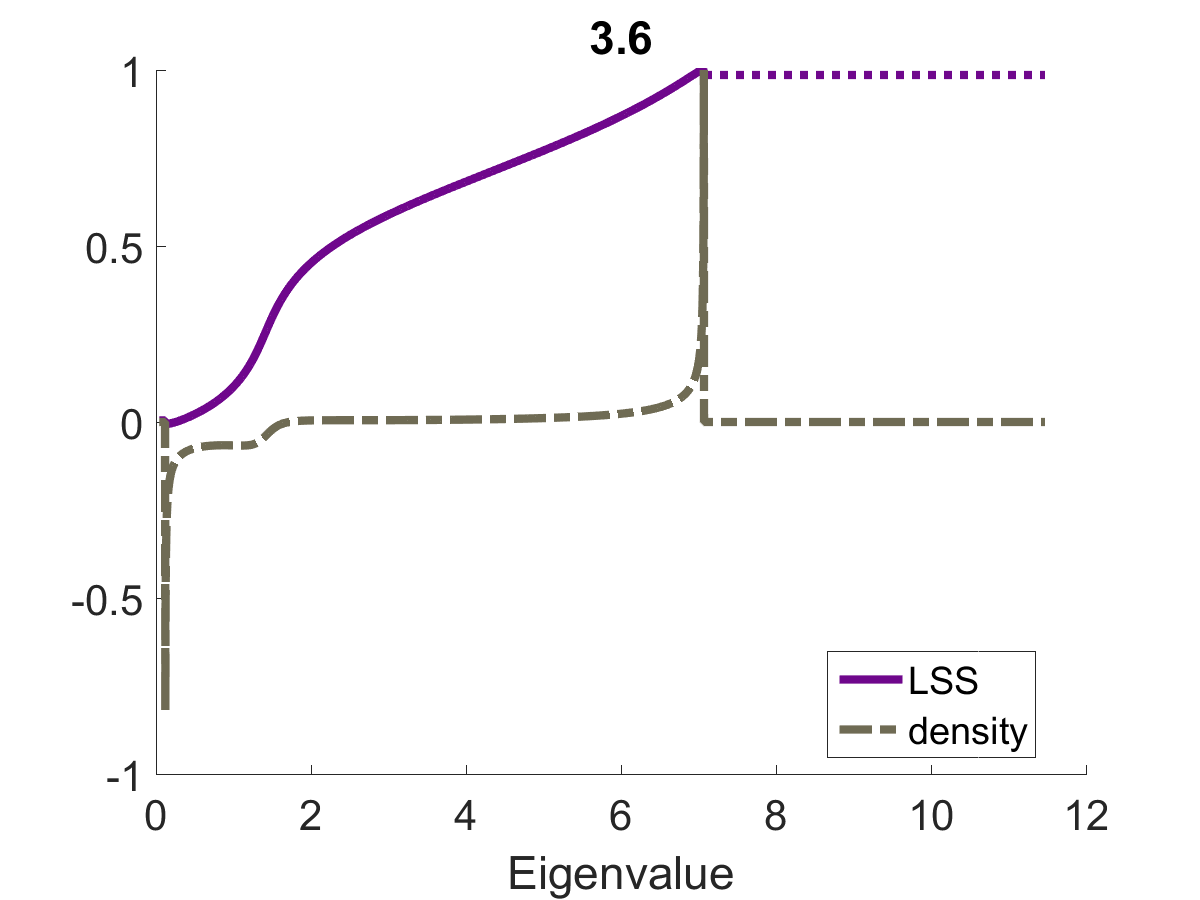}
\end{minipage}
\caption{Density of $\dfg(H,G_1)$ and optimal LSS with $H=2^{-1}(\delta_1+\delta_3)$, $G_1 = \delta_t$, $\gamma=1/2$. On the left plot, the spike $t=0.8$; while on the right plot $t=3.6$; both are subcritical. }
\label{fig2}
\end{figure}

\begin{figure}
\centering
\begin{minipage}{.5\textwidth}
  \centering
  \includegraphics[scale=0.33]
  {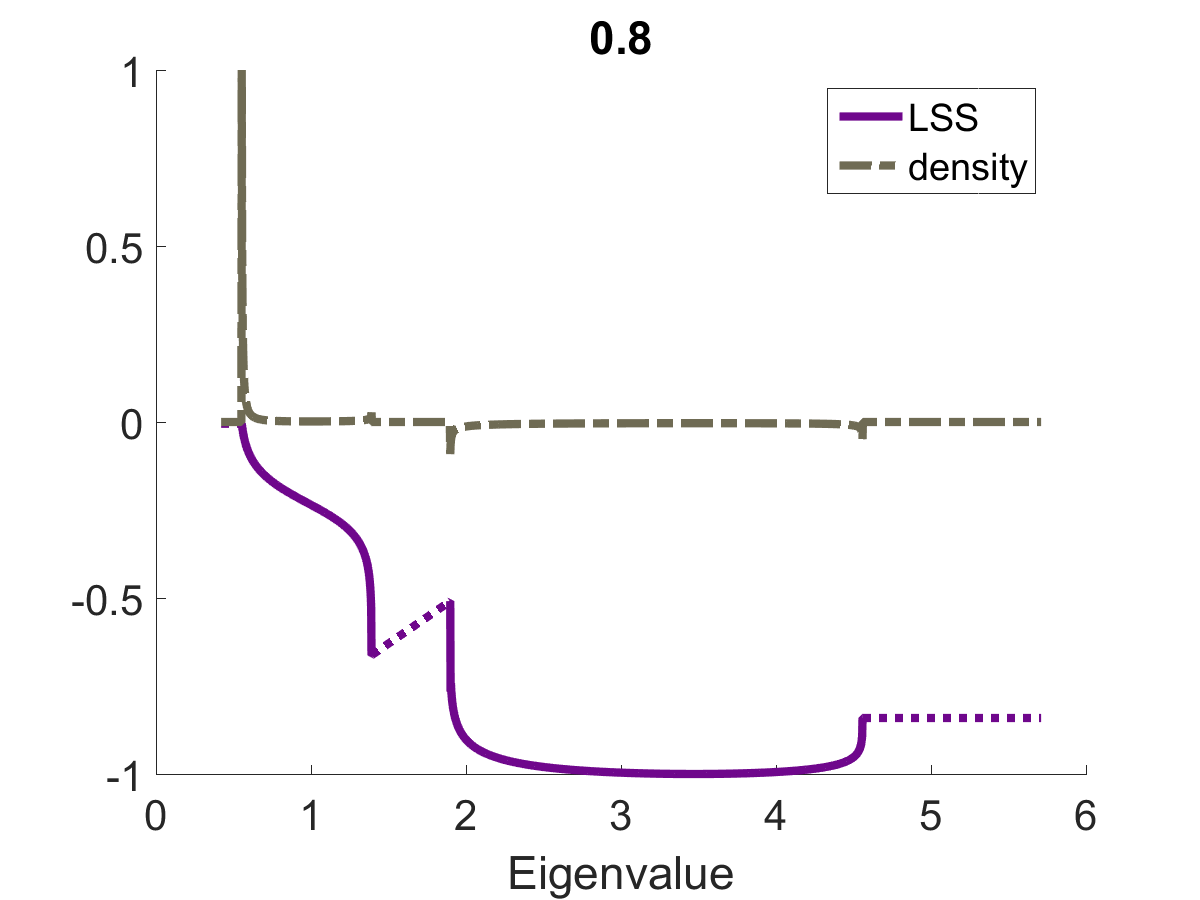}
\end{minipage}%
\begin{minipage}{.5\textwidth}
  \centering
  \includegraphics[scale=0.33]{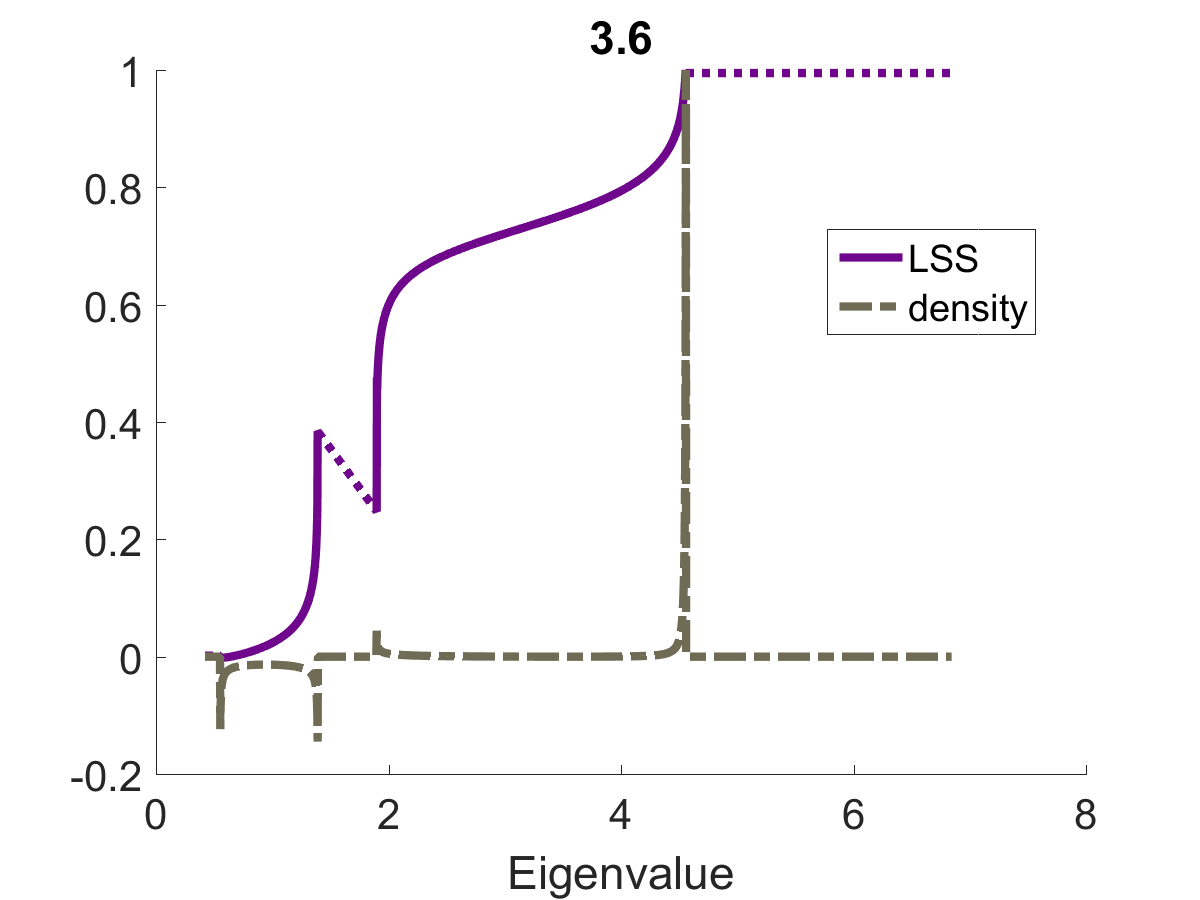}
\end{minipage}
\caption{The same plot as Figure \ref{fig2}, except with $\gamma=1/10$. }
\label{fig3}
\end{figure}

We observe the following:
\begin{enumerate}
\item  \emph{The density of $\dfg$}: For $\gamma = 1/10$, the bulk of sample eigenvalues has two components; for $\gamma = 1/2$, it has only one. This affects both the weak derivative and the optimal LSS. For $\gamma=1/2$, the singularities of $\dfg$ are similar to the standard case. For $\gamma=1/10$, the spike seems to perturb positively the component of the bulk containing it, and perturb negatively the other component.

\item \emph{The LSS}: The optimal LSS are highly nonlinear, and differ a great deal between the four settings ($\gamma \in \{1/10,1/2\}$, $t \in \{0.8,3.6\}$). Note that our theorem only specifies the LSS within the support of the bulk $S$. We extend them by linear interpolation outside, see Section \ref{comput}; this is indicated by the dotted lines.  

In general the optimal LSS are ``large'' where the density of $\dfg$ is positive. However, they have nontrivial shapes; in particular, they showing sharp ``peaks'' at the edges. This shows that the test statistics have qualitatively novel properties. They do not look like the---typically polynomial---LSS equivalent to existing tests of sphericity, see Sec. \ref{lss_examples}.
\end{enumerate}

\subsection{Simulation results}
\label{numer_res}

\begin{figure}[p]
\centering
\begin{tabular}{ccc}

\includegraphics[width=\FW, trim = \TRA mm \TRB mm \TRC mm \TRD mm, clip = TRUE]{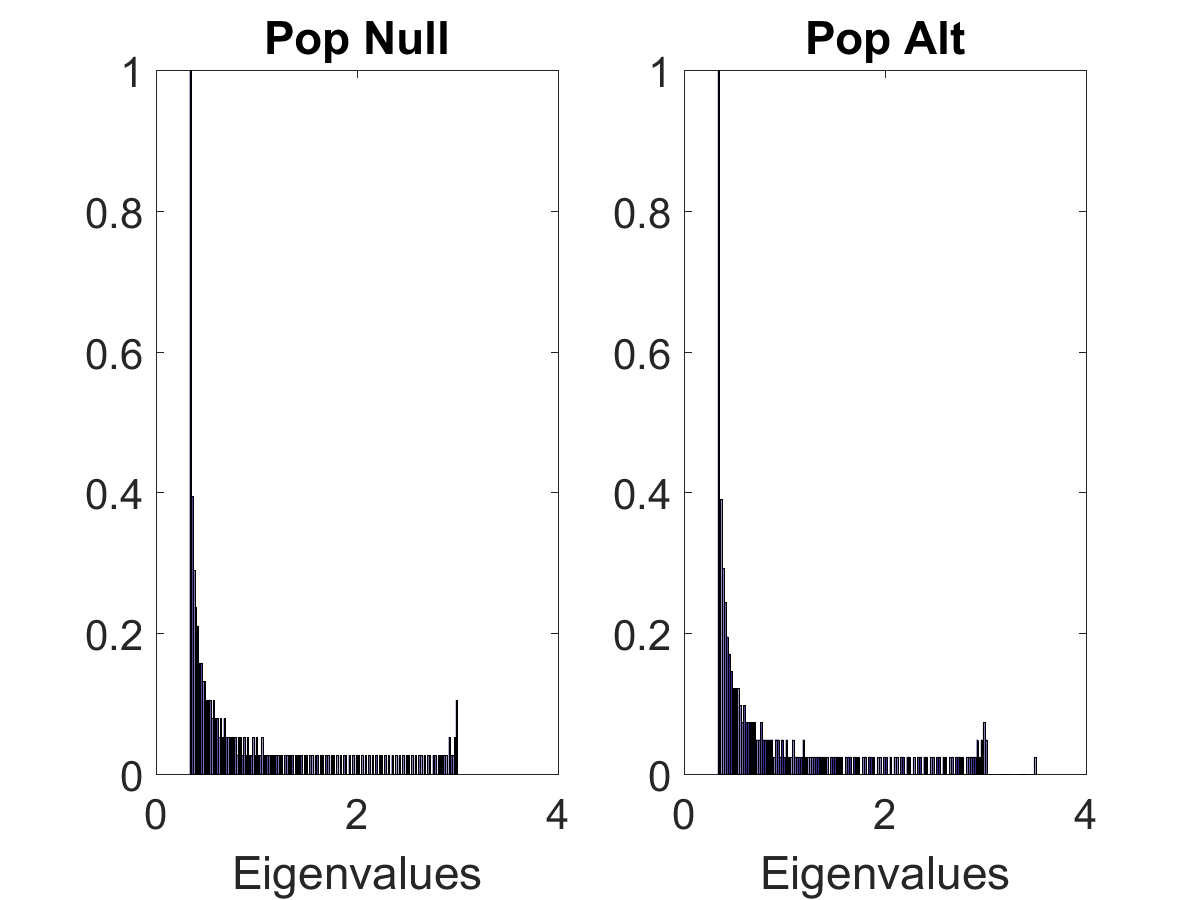} &
\includegraphics[width=\FW, trim = \TRA mm \TRB mm \TRC mm \TRD mm, clip = TRUE]{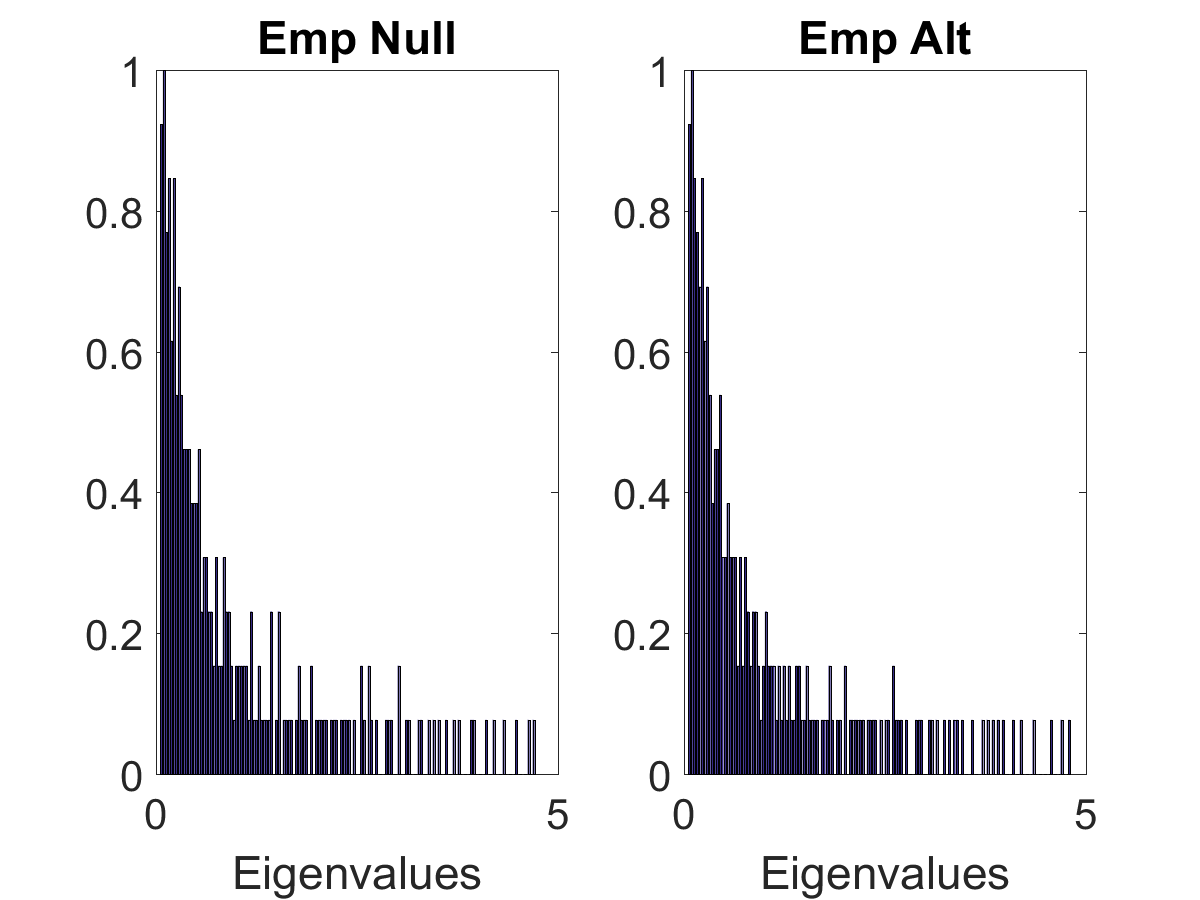} \\
\includegraphics[width=\FW, trim = \TRA mm \TRB mm \TRC mm \TRD mm, clip = TRUE]{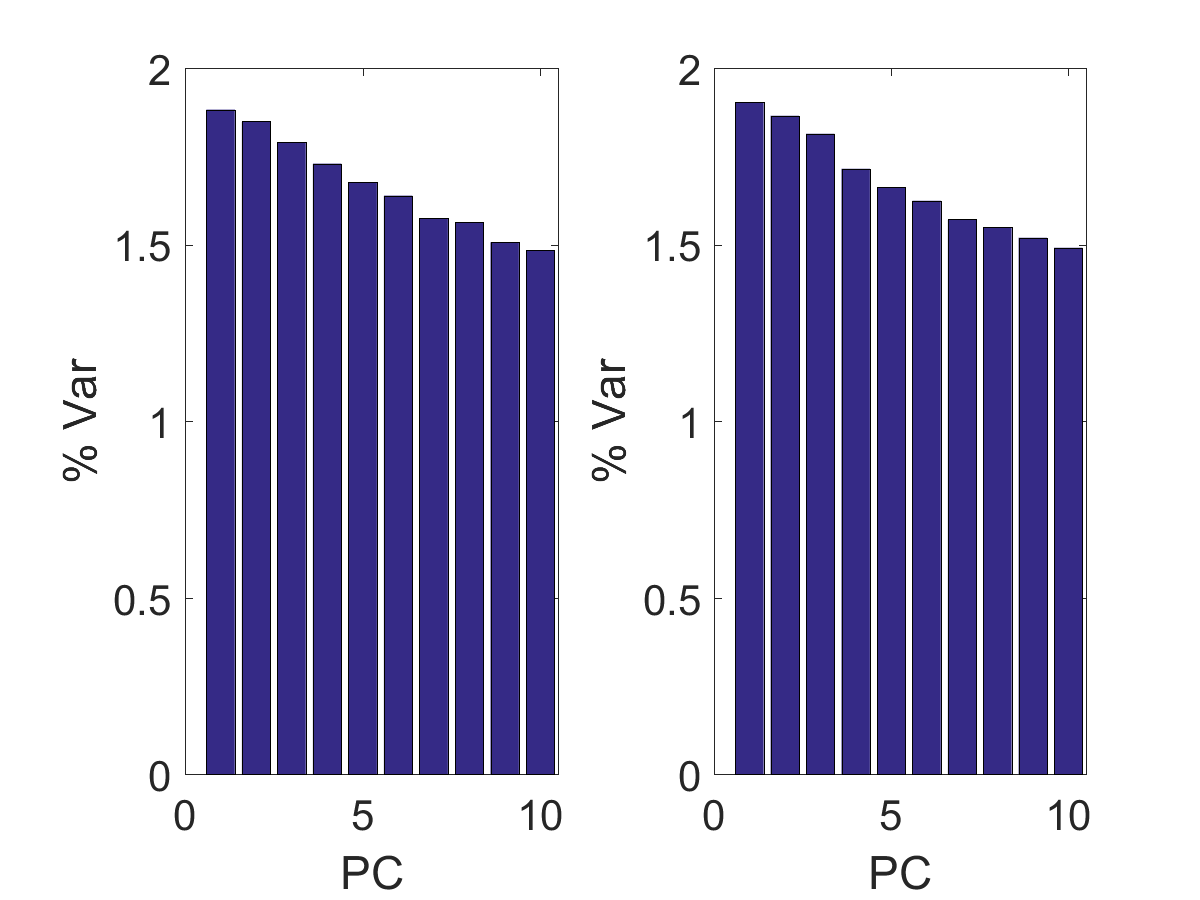}  &
\includegraphics[width=\FW, trim = \TRA mm \TRB mm \TRC mm \TRD mm, clip = TRUE]{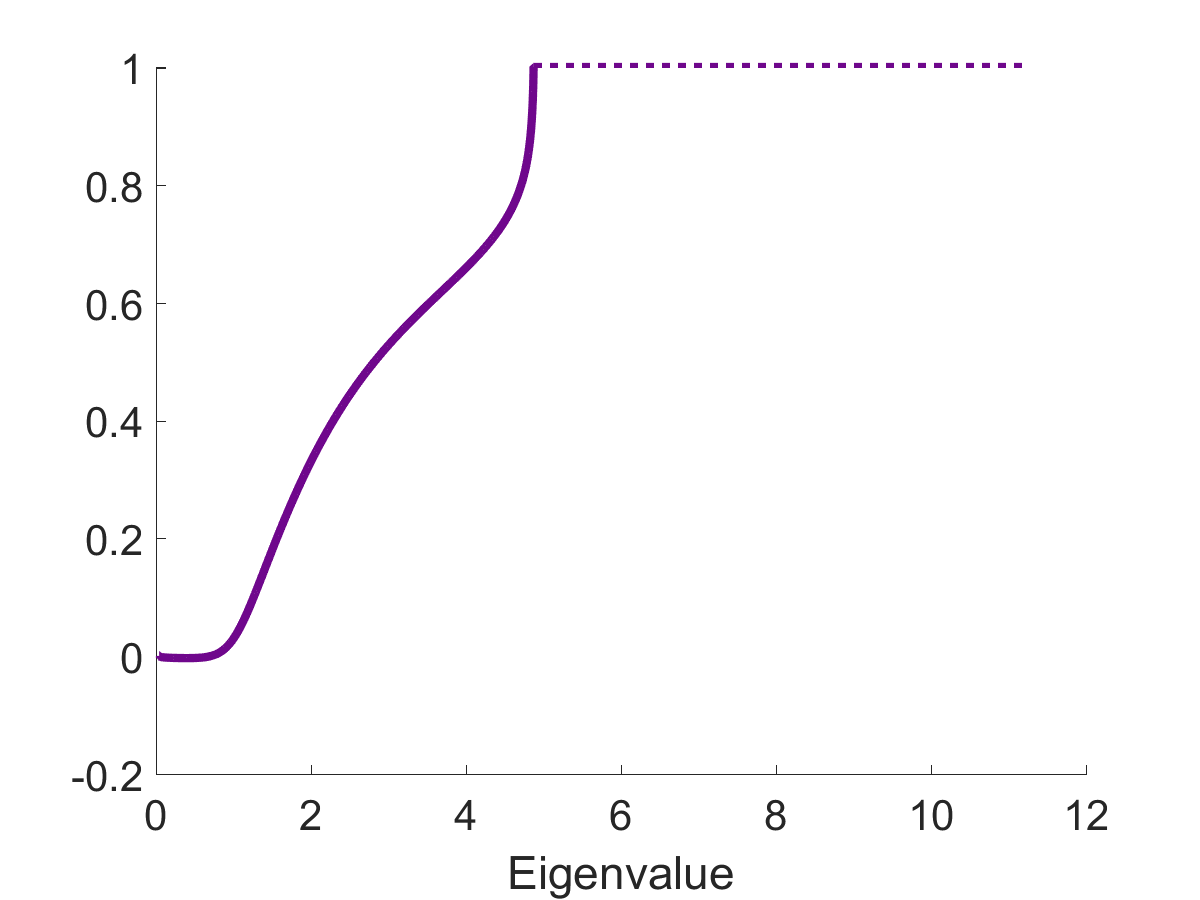}   \\

\includegraphics[width=\FW, trim = \TRA mm \TRB mm \TRC mm \TRD mm, clip = TRUE]{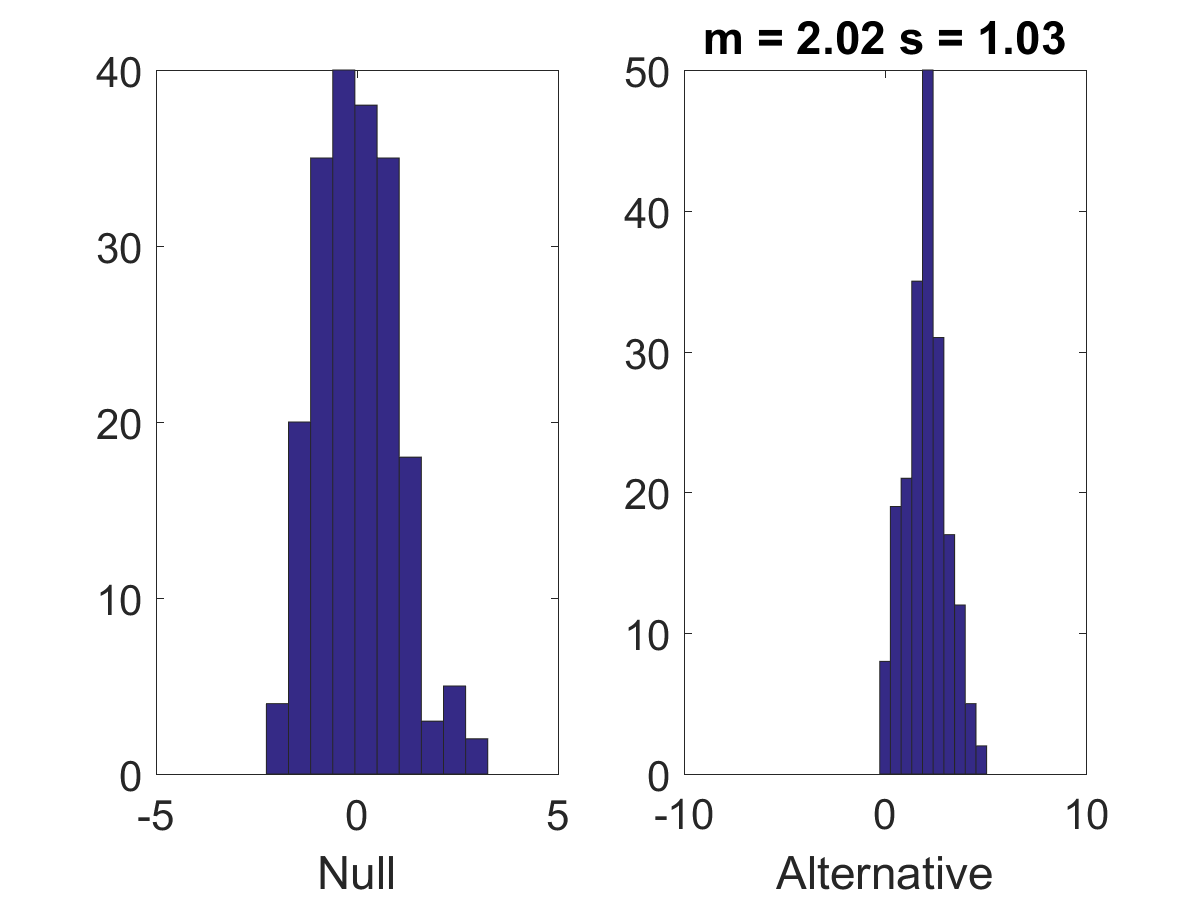} &
\includegraphics[width=\FW, trim = \TRA mm \TRB mm \TRC mm \TRD mm, clip = TRUE]{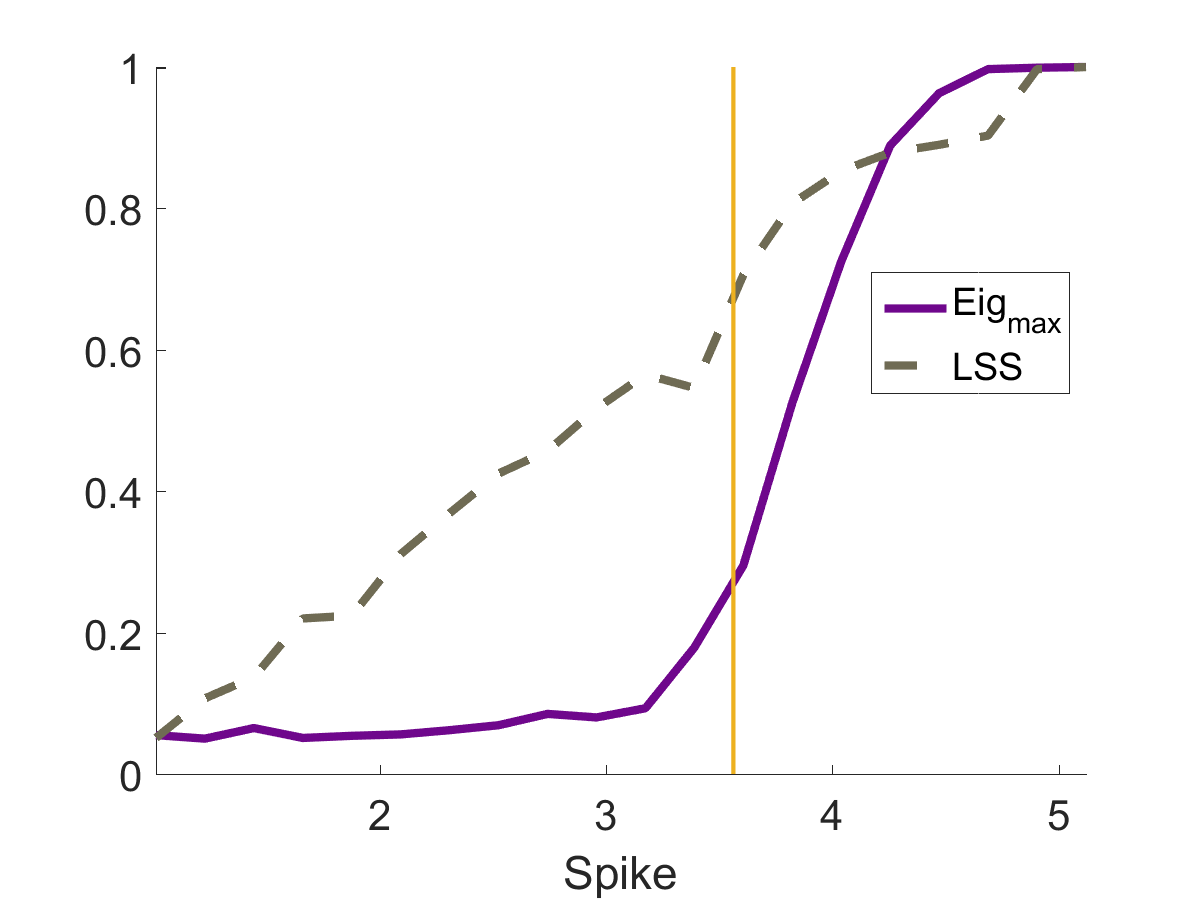}  \\
\end{tabular}
\caption{Simulation results. \emph{Top row, left}: Histogram of population eigenvalues under null and alternative. \emph{Top row, right}: Histogram of sample eigenvalues under null and alternative, for one MC instance. \emph{Middle row, left}: Scree plot of top 10 sample eigenvalues under null and alternative, for one MC instance. \emph{Middle row, right}: pointwise plot of optimal LSS. \emph{Bottom row, left}: Histogram LSS under null and alternative, over 200 MC instances. \emph{Bottom row, right}: Power of optimal LSS and top eigenvalue-based tests as a function of the position of the spike under the alternative.}
\label{fig:MC}
\end{figure}

To illustrate the finite-sample performance of our methods, we present the results of a Monte Carlo (MC) simulation (Fig. \ref{fig:MC}).  The eigenvalues of an autoregressive covariance matrix of order 1 (AR-1) with $\Sigma_{ij} = \rho^{|i-j|}$, and $\rho = 0.5$ make up the null $H$. The sample size is $n=500$ while $\gamma = 1/2$, so the dimension is $p=250$. For large $p$ it is well known that the largest eigenvalue of $\Sigma$ is approximately $(1+\rho)/(1-\rho)$, which equals three (3) in our case. The null spike $s^0=1$ is buried within the population bulk, while the alternative spike $s^1=3.5$ sticks out of it. The histograms of the null and alternative are in the top left plot of Fig. \ref{fig:MC}. The spike $s^1$ is clearly visible. 

We generate a random Gaussian matrix $X = Z \Sigma^{1/2}$ with this covariance matrix and aspect ratio. The histograms of the sample eigenvalues---for both null and alternative---are in the top right plot of Fig. \ref{fig:MC}. The top sample spike does not separate obviously from the sample bulk. This is reinforced by the scree plots of the top 10 eigenvalues under null and alternative, shown in the middle row left plot of Fig. \ref{fig:MC}. The two scree plots look nearly indistinguishable! 

Is it possible to distinguish the two distributions? Our approach is to use the optimal LSS, plotted in in the middle row, right plot of Fig. \ref{fig:MC}. This LSS puts a large weight on the top eigenvalues, while also putting a smaller weight on the middle eigenvalues; and it is extended as a constant outside the bulk. This can indeed distinguish between the two distributions---in the bottom left plot of Fig. \ref{fig:MC} we show the histogram of the LSS over 200 MC samples; we have used the empirical mean and standard error under the null to standardize both histograms. Under both null and alternative, the distributions look approximately normal. Under the alternative, the distribution has mean approximately equal to 2, which is highly encouraging. 

\subsubsection{Increasing the spike}

To examine the power more thoroughly, we perform a broader MC simulation, increasing the alternative spike $s^1$ from 1 to 5. We compare the test which rejects if the top eigenvalue is large to the test based on the optimal LSS---which rejects if the LSS is large enough. For both, we set the critical values based on the empirical distribution of the test statistics under the null, to ensure finite sample type I error control at level $\alpha=0.05$. We record 1000 MC iterates with sample size $n=2000$ and other parameters kept the same as before. 

The results---in the bottom right plot of Fig. \ref{fig:MC}---show that the LSS-based test has power even below the PT threshold, while the top eigenvalue test does not. The vertical line shows the location of the asymptotic PT. 

\begin{figure}
\centering
\begin{minipage}{.5\textwidth}
  \centering
  \includegraphics[scale=0.33]
  {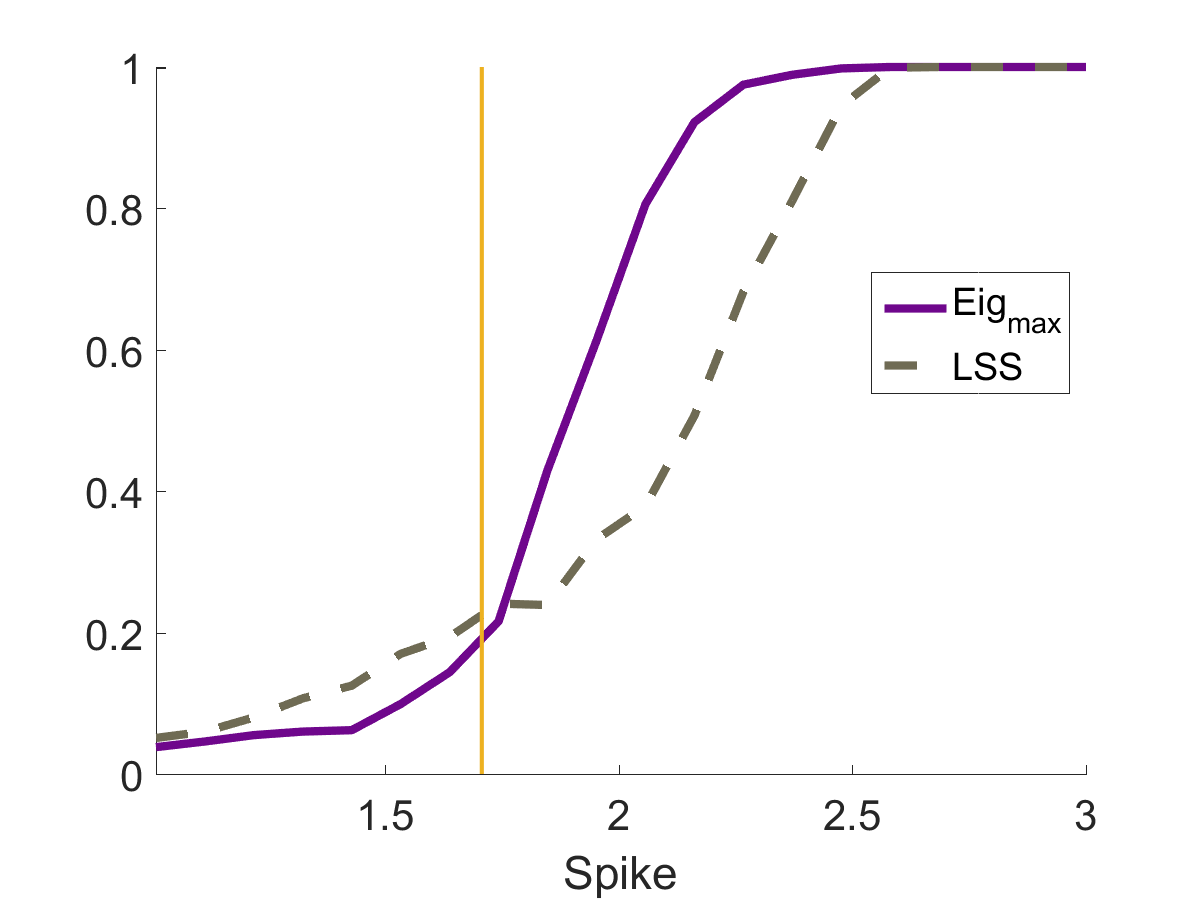}
\end{minipage}%
\begin{minipage}{.5\textwidth}
  \centering
  \includegraphics[scale=0.33]{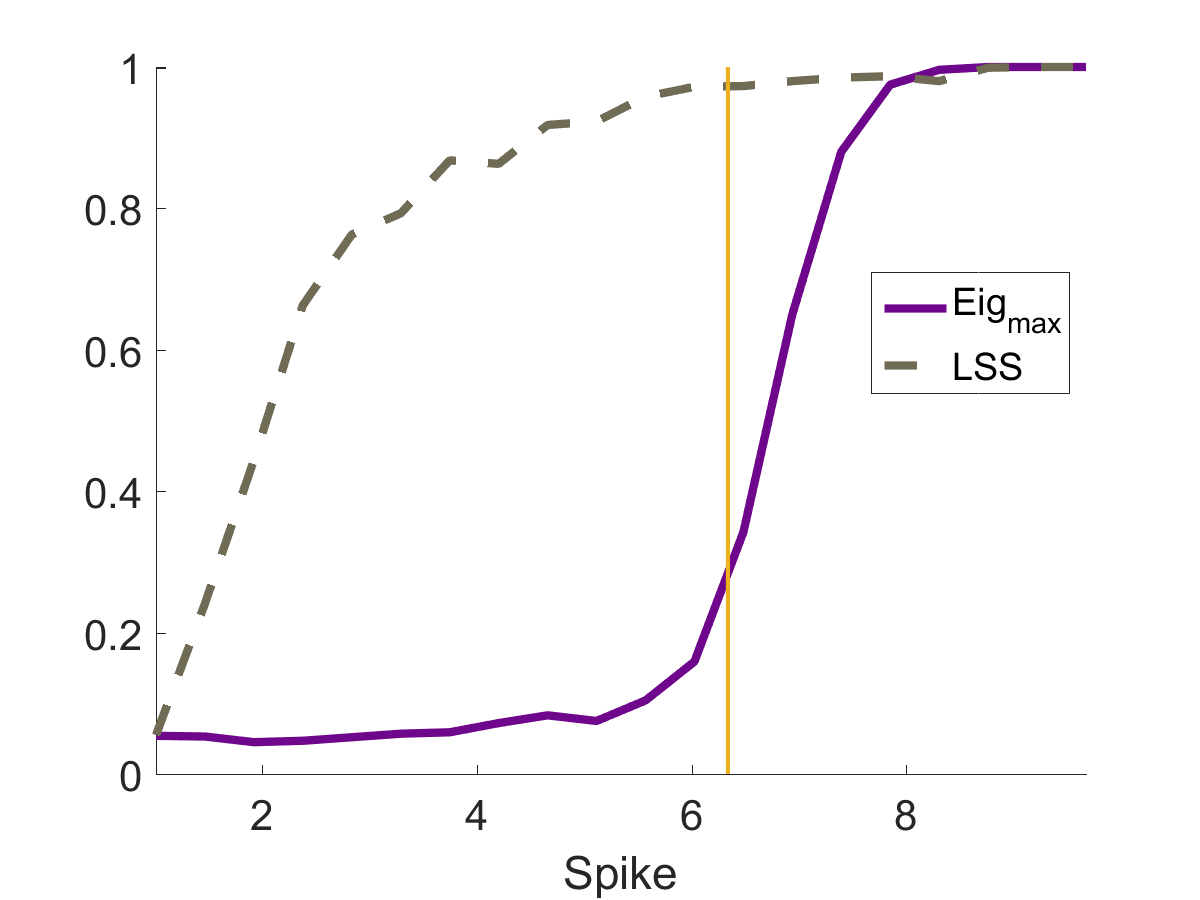}
\end{minipage}
\caption{Power of optimal LSS and top eigenvalue based tests for increasing alternative spike. Left: $\rho = 0$ (identity matrix). Right: $\rho = 0.7$.}
\label{increase_spike}
\end{figure}

To get a broader view of the achievable power in various scenarios, we repeat the last experiment for two additional values of $\rho$. We use $\rho=0$---corresponding to an identity covariance matrix---and $\rho=0.7$, which allows for higher correlations.  In Fig. \ref{increase_spike}, we show the results recorded over 1000 MC iterates with sample size $n=500$ and $\gamma=1/2$. 

For the identity case, the optimal LSS has weak finite sample power. The top eigenvalue test surpasses it above the PT. In contrast, for $\rho=0.7$, the LSS has a lot of power below the PT. The broad conclusion of these experiments is that for eigenvalue distributions that are ``widely spread'', one has indeed the power to detect spikes below the PT. 

\subsection{Properties of the optimal LSS}
\label{prop_lss}

\subsubsection{Full power above the phase transition}
\label{sec:full_pow}

We now continue to study testing in PCA, and derive some fundamental properties of the optimal LSS. 
 In the first section we show that the optimal LSS have full power when the spikes are above the known phase transition threshold from classical spiked models.  This relies on studying the weak derivative of the Marchenko-Pastur map.  For simplicity we will let $G_0=H$, corresponding to a null that is equal to $H$. In this case, $\dfg(H,G_0)=0$, so $\Delta=\dfg(H,G_1)$. With extra work, similar results can be derived for general $G_0$.  
 
We are interested to find the cases where the weak derivative has mass outside of the support $S = \text{Supp}(\fg(H))$. In such a case $\Delta(x)\neq 0$ must occur on a set of positive measure outside $S$. Since the kernel is supported on $S$,  the optimal LSS equation cannot have a solution. This argument will show that the asymptotic power is unity.

We say that a spike $s_j$ is \emph{above the phase transition} if $s_j \in -1/v(S^c)$, where $v$ is the companion Stieltjes transform of $\fg(H)$. This is consistent with the previous definitions for the ''generalized'' spiked model in \cite{benaych2011eigenvalues}, \cite{bai2012sample}; \cite[see also][Chapter 11]{yao2015large}. 
Our goal is to prove the following result: 

\begin{theorem}[Full power above phase transition]
\label{full_pow}
Suppose that in the nonparametric spiked model we have  $\smash{H = d^{-1}\sum_{i=1}^{d}\delta_{t_i}}$, $G_0=H$, and $\smash{G_1 = h^{-1}\sum_{i=1}^{h}\delta_{s_i}}$.  If there is any spike above the phase transition---so that $s_j \in -1/v(S^c)$ for some $j$---then the asymptotic power of the optimal LSS is unity.
\end{theorem}

\begin{proof}
If there is a spike $s_j$---with mass $u_j$ in $G_1$---above the phase transition, then $\dfg(H,G_1)$ has a point mass of weight $\gamma u_j>0$ for some $x \in S^c$ by Proposition \ref{prop_df} (to be proved next). Therefore, the distribution function $\Delta$ has a discontinuity at $x$, and in particular, it is nonzero on  a subset of $S^c$ with positive Lebesgue measure. Since the kernel $k$ is zero on $S^c$, $\Delta$ is not in the image of $K$. By Theorem \ref{thm:lss}, the asymptotic power is unity.
\end{proof}

It remains to prove the following key proposition, which establishes properties of the weak derivative $\dfg$. It will be convenient to define the \emph{spike forward map} $\psi(s)$, which for a population spike $s$ and bulk $H$, gives the location of the sample spike under the effect of the bulk $H$. This is defined through its functional inverse, which is expressed as $\psi^{-1}(x) = -1/v(x)$ \cite[see][Chapter 11]{yao2015large}; and one can verify that $\psi$ is well-defined outside of the support of $H$. The values $x \in S^c$ in the image of the spike forward map,  i.e., for which $x = \psi(s_j)$ for some $j$, will be called the \emph{sample spikes}. We study the weak derivative for arbitrary weighted mixtures of point masses.

\begin{proposition}[Properties of the weak derivative]
\label{prop_df}
Suppose the population bulk is $\smash{H = \sum_{i=1}^{k} w_i \delta_{t_i}}$, with $w_i >0$ such that $\sum_i w_i =1$. Suppose the spikes have distribution $\smash{G = \sum_{j=1}^{l} u_j \delta_{s_j}}$ with distinct $s_j>0$ and weights $u_j>0$ summing to one. 
Let the support of the forward map be $S = \textnormal{Supp}(\fg(H))$, and consider the weak derivative $\dfg(H,G)$. Then, 
\begin{enumerate}
\item $\dfg$ has a density at all in the interior of $S$, $x\in \textnormal{int}(S)$.
\item $\dfg$ has a point mass $\gamma u_j$ at sample spikes $x = \psi(s_j)$, i.e., for the values $x\in S^c$ such that $s_j = -1/v(x)$ for some $j$.

\item $\dfg$ has zero density at all $x$ outside  $\textnormal{int}(S)$ that are not sample spikes. 
\end{enumerate}
\end{proposition}

The proof is postponed to Section \ref{pf:prop_df}. This result sheds new light on phase transition phenomena in spiked models. It shows that the population spikes $s_j$ are ``above the phase transition'', if and only if they create an isolated point mass in the weak derivative. We find this explanation illuminating.

\subsubsection{Linear dependence on the alternative}

In this section we show that the optimal LSS depends \emph{linearly} on the alternative distribution. In this section we will fix $H$ and $G_0$, and will vary $G_1$. Following Theorem \ref{thm:lss_part2}, we will call $\varphi\in \mathcal{W(\I)}$ optimal for testing $H$, $G_0$ against $G_1$, with constant $\eta>0$, if it solves $K(\varphi') = -\eta\, \Delta_\gamma(H,G_1,G_0)$, where $\Delta_\gamma(H,G_1,G_0)=\Delta$ is the distribution function of the difference of weak derivatives. We will need to keep track of the constant $\eta$ in showing linearity. 

\begin{corollary}[Linearity of optimal LSS] Consider a fixed null hypothesis specified by $H$ and $G_0$. Suppose $\varphi_i$ are optimal for testing against the probability measures $G_i$ with constants $\eta_i>0$, for all $i=1,\ldots, M$. Then for any $a_i>0$, $\sum_{i=1}^{M}a_i \varphi_i$ is optimal for testing against $\smash{G = \sum_{i=1}^{M} a_i \eta_i G_i/}$ $\smash{(\sum_{i=1}^{M} a_i \eta_i)}$ (with constant $\sum_{i=1}^{M} a_i \eta_i$). 
\end{corollary}

\begin{proof}
This follows from the linearity of the weak derivative $\dfg(H,\cdot)$ in the second variable (Theorem \ref{weak_der}), and the linearity of the optimal LSS equation (Theorem \ref{thm:lss_part2}).
\end{proof}
This corollary implies that we can build up optimal LSS for complicated alternative hypotheses from simple ones. For instance, we saw numerically that the OMH LSS $f(x,t) = -\log(z(t)-x)$ with $z(t) = t[1 + \gamma/(t-1)]$ is optimal for $H=\delta_1$ against $\delta_t$ with subcritical $t$. It may be possible to use this to find analytically the optimal LSS against more complicated distributions. 

\subsection{Sphericity tests---PCA with unknown scale}
\label{sphericity}

Our entire framework can be extended to sphericity tests, which allow for an unknown scale parameter in PCA. Classically this corresponds to the composite null hypothesis $\Sigma_p = \sigma^2 I_p$, for some unknown $\sigma^2>0$. When studying PCA, the alternative hypothesis of interest is $\smash{\Sigma_p =}$ $\smash{ \sigma^2 (I_p+ \sum_{j=1}^{k} h_j v_j v_j^{\top})}$, for orthonormal $v_j$. We will study the natural generalization of the nonparametric spiked model where the $p$-th problem is
\begin{align}
\label{null_s}
H_{p,0}:&\, H_p = \sigma^2 [(1-hp^{-1})H+h\,p^{-1}G_0] \textnormal{ for some } \sigma^2>0,\\
\label{altve_s}
H_{p,1}:&\, H_p = \sigma^2[(1-hp^{-1})H+h\,p^{-1}G_1] \textnormal{ for some } \sigma^2>0.
\end{align}
Here $H,G_0$ and $G_1$ are probability measures and the integer $h>0$ is the local parameter, with same properties as in the previous sections. When $H = \delta_1$, $G_0=\delta_1$, $h=k$, and $\smash{G_1 = k^{-1}\sum_{j=1}^k \delta_{h_j+1}}$, this recovers the classical setup. 

The null and alternative are both invariant with respect to orthogonal rotations and scaling. It is reasonable to consider tests based on the set of standardized eigenvalues $\lambda_i/\hsigma^2$ of the sample covariance matrix, with $\hsigma^2 =\hsigma_{p}^2 = p^{-1}\tr\hSigma$. With Gaussian data, and when $H = \delta_1$, $G_0 = \delta_1$, they form a set of maximal invariants with respect to rotations and scaling. Moreover, the standardized eigenvalues are distribution-free---or pivotal---under the null. Therefore, we consider linear \emph{standardized} spectral statistics (LS$^3$), which we define as $\smash{S_p(\varphi) = \tr(\varphi(\hSigma/\hsigma^2))}$ = $\smash{\sum_{i=1}^{p} \varphi(\lambda_i/\hsigma^2)}$. This is a broad class of statistics, and many of the existing tests of sphericity are special cases (see Section \ref{lss_examples}). 

Our goal will be to find the optimal LS$^3$. We first establish their asymptotic distribution. We assume the same model as in Section \ref{main_res}. We consider smooth functions $ \varphi \in \mathcal{H}(\I/m_1)$, where $m_1 = \int x dH(x)$. This is because the eigenvalues $\lambda_i$ still belong to the compact interval $\I$ almost surely and---as we will see in the proofs---$\hsigma^2 \to m_1>0$ almost surely. We will use the notation $ \mathcal{F}_{\gamma_p}(g(x))=\int g(x) d\mathcal{F}_{\gamma_p}(x)$ for the integral of a function $g$ under $\mathcal{F}_{\gamma_p}(H)$.

\begin{lemma} [CLT for LS$^3$]
\label{ls3}
For $\varphi\in \mathcal{H}(\I/m_1)$, under the null and alternative \eqref{null_s}, \eqref{altve_s} the linear standardized spectral statistics $S_p(\varphi)$ are asymptotically normal. There is a sequence of constants $c_p$ such that under $H_{p,0}$, $S_p(\varphi)-c_p \Rightarrow \mathcal{N}(0,\sigma_{\varphi ,s}^2)$, while under $H_{p,1}$, $S_p(\varphi)-c_p \Rightarrow \mathcal{N}(\mu_{\varphi ,s},\sigma_{\varphi ,s}^2)$, for a mean shift $\mu_{\varphi ,s}$ and variance $\sigma_{\varphi ,s}^2$. The mean shift and variance are \emph{the same} as those in the asymptotic distribution of the LSS $T_p(j)$, where $j \in  \mathcal{H}(\I)$ is defined by
\begin{equation}
\label{equi_lss}
j(x) = \varphi\left(\frac{x}{m_1}\right)  -\frac{x}{m_1} \mathcal{F}_{\gamma}\left(\frac{x}{m_1} \varphi'\left( \frac{x}{m_1}\right)\right).
\end{equation}
\end{lemma}

The lemma, proved in Section \ref{pf:equi_lss}, states that the LS$^3$ for $\varphi$ and the LSS for $j$ are asymptotically equivalent. Hence we will find the optimal LS$^3$ by optimizing over LSS of the form \eqref{equi_lss}. By scale invariance, we can restrict to working with $\sigma^2=1$, which implies $m_1=1$. 

First we characterize the LSS that are of the required form $j(x) = \varphi(x) - x \fg(x\varphi'(x))$. We claim that 
a function $j\in\mathcal{H}(\I)$ is of this form if and only if $\fg(xj'(x))=0$. Indeed, if $j$ has this form, then $\fg(xj'(x)) = \fg[x\varphi'(x) - \fg(x\varphi'(x))] = 0$. On the other hand, if  $\fg(xj'(x)) =0$, then by taking $f=j$, clearly $j$ is of the required form, as the second term cancels. 

Therefore, we optimize the efficacy from \eqref{opt_lss_x} over the function class $\mathcal{H}_0(\I) = \{\varphi\in \mathcal{H}(\I): \fg(x\varphi'(x))=0\}$. The constraint $\fg(x\varphi'(x))=0$ is a linear equation $\la g,D\ra=0$ for the derivative $g=\varphi'$, with $D = x d\fg(x) \in L^2(\I)$. $D$ is an $L^2$ function, because $\fg$ has a continuous density except at 0, while the $x$ term is null at 0. From the previous sections, it follows that the efficacy optimization over a space $\mathcal{X}$ can be written in terms of $g = \varphi'$ as 
\begin{equation*}
\label{opt_ls3_problem}
\sup_{g\in\mathcal{X}} -h  \frac{\la g,\Delta\ra}{\la g,Kg\ra^{1/2}} \textnormal{ s.t. }\la g,D\ra=0 .
\end{equation*}
As in the previous section, at first we will optimize over $g \in L^2(\I)$, and then extend to analytic functions.  Consider the projection operator $P$ into the orthogonal complement of the one-dimensional space spanned by $D$: $Pg = g -  D\la g,D\ra/\|D\|^2$. Optimizing subject to the linear constraint is equivalent to optimizing over the set $g\in \textnormal{Im}(P)$---or with $g=Pl$ to solving the problem
\begin{equation*}
\sup_{l\in L^2(\I)} -h  \frac{\la Pl,\Delta\ra}{\la Pl,KPl\ra^{1/2}}.
\end{equation*}

Denoting $\Delta_1 = PL$ and $K_1 = PKP$, this reduces to the type of optimization problem solved previously (see \eqref{opt_lss_x}). 
Putting this together with Lemma \ref{ls3} and the analogue of Theorem \ref{thm:lss_part3} for LS$^3$---whose statement and proof is omitted due to its similarity to Theorem \ref{thm:lss_part3}---we obtain the power of LS$^3$.

We consider tests that reject the null if $\smash{S_p(\varphi)-c_p\notin [t_{\varphi}^-,t_{\varphi}^+]}$ for some function-dependent constants $\smash{t_{\varphi}^-,t_{\varphi}^+}$. By scale-invariance it is enough to consider $\sigma^2=1$. In this case we denote the $p$-th null and alternative distribution as $H_{p,0}$ and $H_{p,1}$, respectively. The maximal asymptotic power of LS$^3$ is
$$
\beta_s = \sup_{\varphi\in \mathcal{H}(\I),\,\, t_{\varphi}^-<t_{\varphi}^+} \,\,\, \lim_{p\to\infty} \mathbb{P}_{H_{p,1}}\left(S_p(\varphi) \notin [t_{\varphi}^-,t_{\varphi}^+])\right).
$$

\begin{theorem}[Asymptotic power of LS$^3$] 
Consider scale-invariant tests for detecting weak PCs based on linear standardized spectral statistics $S_p(\varphi)$. Suppose $\varphi\in \mathcal{H}(\I)$ and the tests have asymptotic level $\alpha\in(0,1)$. 
The maximal asymptotic power is
$$
\beta_s=
\left\{
	\begin{array}{ll}
		\Phi\left(z_{\alpha}+h \, \la K_1^+\Delta_1,\Delta_1\ra^{1/2}\right) & \mbox{\, if \, } \Delta_1 \in \textnormal{Im}(K_1), \\
		1 & \mbox{\, if \, } \Delta_1\notin\textnormal{Im}(K_1).\\
	\end{array}
\right.
$$
\end{theorem}

This theorem quantifies the loss of power due to restricting to scale-invariant LS$^3$ from LSS. If $\Delta,D \in  \textnormal{Im}(K)$, it can be checked that $\Delta_1 \in \textnormal{Im}(K_1)$. Moreover, the efficacy is $\theta_s = h[\la K^+L,\Delta\ra - \la K^+L,D\ra^2/\la K^+D,D\ra]^{1/2}$. This shows that the efficacy is reduced from $\theta = h\la K^+L,\Delta\ra^{1/2}$, and the power loss depends on the ``correlation'' between $\Delta$ and $D$ with respect to $K$.

\subsection{Related work}
\label{rel_work}

We now turn to discussing some related work. As described in the introduction, there are two main lines of inquiry on testing in high-dimensional PCA that are related to our results. The first line connects PCA to tests of sphericity against low-rank alternatives. Classical tests of sphericity, designed for general alternatives, are a topic of renewed interest in a high-dimensional context \citep{ledoit2002some,bai2009corrections, cai2013optimal,li2015testing}, reviewed also in \cite[][Sec. 5.4]{cai2014estimating}. Some works, such as \cite{wang2013sphericity, wang2014note,choi2015regularized} study classical tests under spiked alternatives. 

As already discussed, in the special case of Gaussian data in white noise, \cite{onatski2013asymptotic,onatski2014signal} have recently discovered that subcritical PCs can be detected with nontrivial power. Their LR tests are asymptotically equivalent to LSS. This has been extended to $F$-matrices and a few other explicitly solved examples \citep{dharmawansa2014local,johnstone2015testing}.
 Their interesting likelihood approach does not generalize obviously to our setting. Nonetheless, we showed that their surprising discovery is true much more generally. Even more strikingly, we numerically recovered their methods as a special case (see Section \ref{ex_plots}). Their approach has the advantage of leading to more explicit formulas. Our approach has the advantage of generality.  We view these results as complementary. 

The second line of work studies stronger principal components, and allows for correlations in the residuals. In the econometrics literature the presence of correlated background noise---or idiosyncratic noise---is a key concern in the related area of factor models. However, most of that work assumes strong factors \citep{bai2008large,onatski2009testing,ahn2013eigenvalue}. In that case the signal eigenvalues are asymptotically unbounded under the alternative, and thus easier to detect. A similar separation holds for weaker signals ``above the phase transition''. The signal eigenvalues asymptotically separate from the bulk, and can be detected with full power via a simple test \citep{paul2007asymptotics,nadakuditi2008sample, kritchman2009non,bianchi2011performance}. 

Finally, there are many other important results on PCA in high dimensions that we cannot review due to space limitations. The second-order asymptotics of eigenvalues have been studied starting with \cite{paul2007asymptotics}. The finite-sample, and low-noise, asymptotics have been described in \cite{nadler2008finite}. The \emph{estimation} of the number of PCs has also been analyzed above the phase transition \citep{bai2012estimation,onatski2012asymptotics}. There are many other interesting topics---sparse PCA, computation-statistics tradeoffs, kernel PCA, non-linear dimension reduction, etc.---that we do not have space to go into.

\section{Linearization of tests: A unifying framework}
\label{lss_examples}

\subsection{A general approach to non-linear spectral statistics}

Tests for the covariance matrix are a core topic in multivariate statistics, discussed in many textbooks \citep{anderson1958introduction, yao2015large}. In this section we will analyze many existing tests in a unified way. We show that these tests are asymptotically equivalent to LSS in our nonparametric model, going beyond sphericity. Therefore, the existing tests can be compared in the same general framework. 

Linearization techniques like the delta-method are of course well-known in asymptotic statistics. In the specific context of high-dimensional sphericity tests, they have been used by \cite{ledoit2002some} and several later works. However our results are much more general, because the asymptotic equivalence to LSS holds in the nonparametric spiked model with \emph{any} distribution of PC variances, not just under the standard null $H=\delta_1$, as in \cite{ledoit2002some} and later works such as \cite{bai2009corrections}. The recent papers of \citet{wang2013sphericity, wang2014note} still consider the standard null, but also compute power under some alternatives. 

Another key difference is that we are not interested in obtaining the limiting mean and variance parameters of the test statistics. A great deal of work---usually contour integral calculations---usually goes into finding these parameters \emph{explicitly}. In contrast, we simply reduce the test statistics to LSS. Then the mean and variance can be computed \emph{numerically} using the computational framework of \cite{dobriban2015efficient}. Accurate numerical methods may suffice in many applications.

The first component of our theory is the following lemma, proved later in Section \ref{pf:linearize} using the CLT for the LSS of \cite{bai2004clt}. This is essentially the delta method of classical statistics, as it arises here. 

\begin{lemma}[Linearization of Spectral Statistics] 
\label{linearize}
Let  $T_p(\varphi)$ and $T_p(\psi)$ be two linear spectral statistics with $\varphi,\psi \in \mathcal{H}(\I)$, and $y$ be a real-valued bivariate function continuously differentiable in a neighborhood of $a = (\fg(\varphi),\fg(\psi))$. Let $j = \partial_1y(a)\cdot \varphi + \partial_2 y(a)\cdot \psi \in \mathcal{H}(\I)$ and suppose the LSS $T_p(j)$ has a nonzero asymptotic variance $\smash{\sigma_j^2>0}$. Then the \emph{non-linear} spectral statistic $Y_p = y(p^{-1}T_p(\varphi), p^{-1} T_p(\psi))$ is asymptotically equivalent to the LSS $T_p(j)$. Specifically, there is a sequence of deterministic constants $d_p$ such that $p \cdot Y_p= T_p(j) + d_p + o_P(1)$ under the nonparametric spiked model \eqref{null_s}, \eqref{altve_s}.
\end{lemma}

The equivalent LSS provided by this lemma generally depends on $\sigma^2$, and thus may not be a \emph{bona fide} test statistic. Nonetheless, it can be a useful tool to compare different tests. The result also holds with the same proof for multivariate functions $y$. However, all examples of interest are at most bivariate.

\subsection{Examples}
\label{examples}

In this section we use Lemma \ref{linearize} to show that several popular tests of identity and sphericity are asymptotically equivalent to linear spectral statistics in the nonparametric spiked model. Moreover, this section also reviews related work, including tests that are LSS in the original form, and tests that are not equivalent to LSS. 

Whenever we use Lemma \ref{linearize}, we assume that the limiting variance of the equivalent LSS is positive: $\smash{\sigma_j^2>0}$. This assumption is not a significant limitation, and it can be checked directly for any example of interest.  We will use the notation $g_i(x) = x^i$ for the monomials. The proofs will be given in Section \ref{pf:examples}.

\begin{enumerate}
\item 
The normal log-LRT for testing identity, $\Sigma = I_p$ against any $\Sigma$, \cite[][Sec. 10.8]{anderson1958introduction}, equals $\tr \hSigma - \log\det\hSigma - p = T_p(\varphi_1)$  up to normalization, where $\varphi_1(x) = x - \log(x)-1$. This is a linear spectral statistic. \cite{bai2009corrections} proposed to ``correct'' this test under the null using the proper high-dimensional centering term from the CLT for LSS. 

\item  The normal log-LRT for testing sphericity, $\Sigma = \sigma^2 I_p$ with unknown $\sigma^2$ against any $\Sigma$, due to \cite{mauchly1940significance}, equals $\smash{p \log(p^{-1}\tr \hSigma) - \log\det\hSigma}$  up to normalization. Linearization by Lemma \ref{linearize} shows that it is equivalent to the LSS $T_p(\varphi_2)$, where $\varphi_2(x) = x/\fg(g_1) - \log(x/\fg(g_1))-1$.  

\citet{wang2013sphericity, wang2014note}  studied the power of this test under spiked alternatives to the null of identity, and under more general non-Gaussian moment conditions. \cite{li2015testing} introduced a quasi-LRT modification for the ultrahigh dimensional regime $p\gg n$. These interesting papers generally calculate the limiting moments of the relevant LSS \emph{explicitly}, and then use the delta method to get the final distribution. In contrast, our theory relies on a \emph{single} linearization lemma to reduce to LSS, followed by numerical computation of the moments. Accurate numerical moments may suffice in many applications. 

In general, the LSS depends on the unknown $\fg(g_1)$, which equals the unknown parameter $\sigma^2$ under the null, thus it is not a \emph{bona fide} test statistic. However, the LSS can still be used to study the power of the original statistic. For instance, under the null of identity, the LRT of sphericity is equivalent to the LRT of identity, discussed above. Indeed, if $\sigma^2=1$, then $\fg(g_1) = 1$, so $\varphi_2(x)=\varphi_1(x)$ and the two tests are asymptotically equivalent.  Tests of sphericity are designed to work under the composite null $\Sigma_p = \sigma^2 I_p$, with $\sigma^2>0$ unspecified. Due to this, they may have lower power than tests of identity. Therefore, it is perhaps remarkable that there is asymptotically \emph{no loss of efficiency} in using the LRT of sphericity.

\item  The locally optimal invariant test of identity of \cite{john1971optimal} is based on $\tr{\hSigma}$, and is clearly an LSS. The locally optimal invariant test of sphericity of \cite{john1971optimal} is based on $\smash{\tr{(a\hSigma-I_p)^2}}$, where $a = p/\tr{\hSigma}$. Lemma \ref{linearize} shows that it is equivalent to $T_p(\varphi_3)$, where $\varphi_3(x) = \fg(g_1)\cdot x^2 - 2\fg(g_2) \cdot x$.  

In particular, under the null of identity, it is known that $\fg(g_1) = 1$ and $\fg(g_2) = 1+\gamma$. John's test is then asymptotically equivalent to $\varphi_3(x) = x^2 - 2(1+\gamma)x$. In the special case of a standard null $H = \delta_1$, this result agrees with \cite{ledoit2002some}; but we emphasize that our asymptotic equivalence is valid under \emph{any} nonparametric null and any local alternative. 

More recently, using the same technique of explicit calculation described above, \cite{wang2013sphericity} studied this test statistic under general fourth moment conditions and under spiked alternatives for $H=\delta_1$. \cite{li2015testing} established its distribution in ultrahigh dimensions, and argued for its powerful dimension-proof property.

\cite{john1971optimal}'s locally optimal tests are similar in spirit to our approach. They were derived via local alternatives to the standard null in low dimensions, based on the explicit density of eigenvalues for Gaussian data. 

\item  The test of identity from \cite{nagao1973some} is based on $\tr{(\hSigma-I_p)^2}$ . This is an LSS with $\varphi_4(x) = (x-1)^2$. \cite{ledoit2002some} proposed the modification $\tr{(\hSigma-I_p)^2} - n^{-1} (\tr{\hSigma})^2$. This was one of the first proper high-dimensional tests of identity, proven to have correct level as $n,p\to \infty$ with $p/n\to\gamma>0$. Its asymptotic distribution under the null was originally derived using an earlier CLT for LSS valid for polynomial $\varphi$. 

In our more general setting of arbitrary null and alternative, using linearization, \cite{ledoit2002some}'s proposal is asymptotically equivalent to an LSS $T_p(\varphi_5)$, with $\varphi_5(x) = x^2 - 2(1 +\gamma \fg(g_1)) x $. In particular, under the null of identity, it is asymptotically equivalent to the sphericity test from \cite{john1971optimal}. Indeed, $\fg(g_1)=1$, so $\varphi_5(x) = \varphi_3(x)$. This again recovers as a special case the results of  \cite{ledoit2002some} who showed that the two have the same limit distribution under the identity null. In our asymptotic setting, \cite{srivastava2005some}'s proposed test of identity is also equivalent to the Ledoit-Wolf test.

\item \cite{fisher2010new} proposed a sphericity test based on $p \tr{\hSigma^4}/( \tr{\hSigma^2})^2$, similar to John's test. This is asymptotically equivalent to the LSS with $\varphi_6(x) = \fg(g_2)\cdot x^4 - 2\fg(g_4) \cdot x^2$. Under the null of identity, we have $\fg(g_2) = 1 + \gamma$ and $\fg(g_4) = (1+\gamma)(1+5\gamma+\gamma^2)$, using the moments of the Marchenko-Pastur distribution \citep{bai2009spectral}. Thus the test is equivalent to the LSS with $\varphi_6(x) = x^4 - 2(\gamma^2+5\gamma+1)x^2$. 

Later \cite{fisher2012testing} proposed two tests of identity based on unbiased estimators of all moments $\tr \Sigma^i$, for $i=1,\ldots,4$, obtained by linear combinations of products of various $\tr \hSigma^i$. These are clearly equivalent to LSS for certain polynomials $\varphi_6'(x) = \sum_{i=0}^{4}c_i x^i$, but the coefficients are too complicated to derive here.

\item Recently, \cite{choi2015regularized} proposed a regularized LRT which equals $\tr{ \hSigma} - \log\det (\hSigma + \lambda I_p)$ up to constants, for some $\lambda>0$. This is an LSS with $\varphi_7(x) = x - \log (x + \lambda)$. 

\end{enumerate}

These findings are summarized in Table \ref{tab:lss_equiv}. The tests as given here can differ from the original papers by normalization. We display in red the test statistics for which our equivalent LSS are new in the nonparametric spiked model. We also show the equivalent LSS for the LRT from \cite{onatski2013asymptotic}, which is valid under a Gaussian white noise null with spiked alternatives. 

{\renewcommand{\arraystretch}{1.5}
\begin{table}[]
\centering
\caption{Existing tests of identity and sphericity. We use the notation $m_i=\fg(g_i)$. ``I'' indicates that the test was devised as a test of identity, while ``S'' indicates it was a test of sphericity. Here $z(t) = t[1 + \gamma/[t-1]]$.}
\label{tab:lss_equiv}
\resizebox{\textwidth}{!}{%
\begin{tabular}{|l|l|l|l|}
\hline
Source                           & $H_0$    & Original Form                                     & Equivalent LSS                                          \\ \hline
Folklore  - LRT                       & I  & $\tr \hSigma - \log\det\hSigma - p$               & $x - \log(x)-1$                                \\ \hline
\cite{mauchly1940significance} - LRT & S & $p \log\tr \hSigma - \log\det\hSigma$     & $\red{ x/m_1 - \log(x/m_1)-1}$                       \\ \hline
\cite{john1971optimal}         & I  & $\tr{\hSigma}$                                    & $x$                                                     \\ \hline
\cite{john1971optimal}         & S & $\tr{(p\hSigma/\tr\hSigma-I_p)^2}$                & $\red{ x^2 m_1 - 2xm_2 }$    \\ \hline
\cite{nagao1973some}           & I  & $\tr{(\hSigma-I_p)^2}$                            & $(x-1)^2$                                               \\ \hline
\cite{ledoit2002some}          & I  & $\tr{(\hSigma-I_p)^2} - (\tr{\hSigma})^2/n$ & $ \red{ x^2 - 2(1 +\gamma m_1) x}$ \\ \hline
\cite{fisher2010new}           & S & $p\tr{\hSigma^4}/(\tr{\hSigma^2})^2$ & $\red{ x^4m_2 - 2x^2m_4}$  \\ \hline
\cite{fisher2012testing}    & I & complicated & $\sum_{i=0}^{4}c_i x^i$  \\ \hline
\cite{onatski2013asymptotic}           & I & LRT & $-\log(z(t)-x)$  \\ \hline
\cite{onatski2013asymptotic}           & S & LRT & $-\log(z(t)-x) - (t-1)/\gamma x$  \\ \hline
\cite{choi2015regularized}           & I & $\tr{ \hSigma} - \log\det (\hSigma +\lambda I_p)$ & $ x - \log (x + \lambda)$  \\ \hline
\end{tabular}
}
\end{table}

\subsection{Other tests}

Not all tests of identity or sphericity are asymptotically equivalent to linear spectral statistics. Here we give some representative examples. In the 1950's Roy proposed tests based on the extreme eigenvalues of the sample covariance matrix. After work by Tracy and Widom, \cite{johnstone2001distribution} showed that the largest eigenvalue of $\smash{\hSigma}$ has an asymptotic Tracy-Widom distribution for real Gaussian white noise under high-dimensional asymptotics. Various authors have later proposed ways to use this distribution in practice, see \cite{onatski2013asymptotic} for a review. These tests are not equivalent to LSS.

\cite{chen2010tests} proposed tests of identity and sphericity inspired by \cite{ledoit2002some}, based on more accurate estimators of $\tr{\Sigma}$ and $\tr{\Sigma^2}$. Their test statistic of identity is equivalent to 
$$
\frac{\sum^{*}(x_i^\top x_j)^2+2\sum^{*}x_i^\top x_j}{P_n^2}
- \frac{2\sum^{*}x_i^\top x_j x_j^\top x_k}{P_n^3}
+ \frac{\sum^{*}x_i^\top x_j x_k^\top x_l}{P_n^4}
-\frac{2\sum x_i^\top x_i}{n} ,
$$
where $x_i$, $i=1,\ldots,n$ are the samples, $P_n^k = n!/(n-k)!$, and the summations with the $*$ symbol are over all distinct indices. In the same spirit, \cite{cai2013optimal} proposed the U-statistic $U = 2\sum_{i<j}h(x_i,x_j)/[n(n-1)]$, with $h(x,y) = (x^\top y)^2 - \|x\|^2 - \|y\|^2 +p$ for testing identity.  \cite{onatski2014signal} show that the results of \cite{cai2013optimal} imply that $U$ has equivalent power to the Ledoit-Wolf LSS for Gaussian white noise and alternatives. However, it is not clear if this equivalence holds more generally for non-Gaussian models, and whether there exists an explicit LSS such that $U-T_p(\varphi)-d_p=o_P(1)$, which is the claim we are interested in. 

\cite{li2014hypothesis} developed a test of identity based on a measure of distance of the sample ESD and the null Marchenko-Pastur law. Their test statistic is $\sum_{i=1}^{k}|1/\hv(z_i)-1/v_{\gamma_p}(z_i)|^2$ for suitable $z_i$, where $\hv(z)$ is the Stieltjes transform of the companion matrix $n^{-1}XX^\top$, and $v_{\gamma_p}$ is the companion Stieltjes transform of $\smash{\mathcal{F}_{\gamma_p}(\delta_1)}$. Due to the squared norm, this has a $\chi^2$ limit distrbution and is not equivalent to an LSS. 

Finally, there are many tests for covariance matrices based on maximum entrywise deviations; see \cite{cai2014estimating} for a review. These are generally not equivalent to LSS.

\section{Proofs}
\label{Proofs}

\subsection{Main steps of the proofs}
\label{mainproofs}
\subsubsection{Weak derivative of the Marchenko-Pastur map}
\label{sec:weak_der}

We start by explaining the main steps in proving Theorems \ref{thm:lss} (asymptotically normal equivalent) and \ref{thm:lss_part2} (optimal LSS equation). These lead to the proof of Theorem \ref{pow_lss} (asymptotic power).
Starting with Theorem \ref{thm:lss}, our assumptions imply that the Bai-Silverstein CLT for linear spectral statistics \citep[][Thm 1.1]{bai2004clt} applies both under the sequences of null and alternative hypotheses. Denoting---perhaps with a slight abuse of notation---by $H_{p,i}$ the spectral distributions under null ($i=0$) and alternative ($i=1$), this shows that 
\begin{align*}
\textnormal{under} \,\, H_0:&\, T_p(\varphi) - p\int_\I \varphi(x) d \mathcal{F}_{\gamma_p}(H_{p,0}) \Rightarrow \mathcal{N}(m_\varphi,\sigma_{\varphi}^2), \textnormal{ while}\\
\textnormal{under} \,\, H_1:&\, T_p(\varphi) - p\int_\I  \varphi(x) d \mathcal{F}_{\gamma_p}(H_{p,1}) \Rightarrow \mathcal{N}(m_\varphi,\sigma_{\varphi}^2).
\end{align*}

Here $m_\varphi, \sigma_{\varphi}^2$ are certain constants that are the same under the null and the alternative. Indeed, in Theorem 1.1 of \cite{bai2004clt}, these limiting parameters are given by certain contour integrals that \emph{only depend on the weak limit of the PSD}, and in our case these weak limits are both equal to $H$. The explicit form of these constants will only matter later. The important part is the \emph{difference} in the centering terms, i.e., the change from the argument of $\mathcal{F}_{\gamma_p}$ from $H_{p,0}$ to $H_{p,1}$. Therefore, the mean shift between the two hypotheses ought to equal
$$
\mu_{\varphi} = \lim_{p\to\infty}p\int_\I  \varphi(x) d \left[\mathcal{F}_{\gamma_p}(H_{p,1})-\mathcal{F}_{\gamma_p}(H_{p,0})\right],
$$

provided this limit is well defined. 
It is natural to conjecture that the signed measures $D_p=p\left[ \mathcal{F}_{\gamma_p}(H_{p,1})-\mathcal{F}_{\gamma_p}(H_{p,0}) \right]$ have a weak limit---and we will in fact prove this. We can write 
\begin{align*}
D_p=&p\left[ \mathcal{F}_{\gamma_p}(H_{p,1})-\mathcal{F}_{\gamma}(H_{p,1}) \right]
-p\left[ \mathcal{F}_{\gamma_p}(H_{p,0})-\mathcal{F}_{\gamma}(H_{p,0}) \right]  \\
+ &p\left[ \mathcal{F}_{\gamma}(H_{p,1})-\mathcal{F}_{\gamma}(H) \right] 
- p\left[ \mathcal{F}_{\gamma}(H_{p,0})-\mathcal{F}_{\gamma}(H) \right]
\end{align*}
Since $\gamma_p=\gamma$, the first two terms are 0; if we relaxed the assumptions to $\gamma_p\to\gamma$, these limits would need to be evaluated. Therefore, by the definition of the weak derivative of the Marchenko-Pastur map \eqref{dfg}, and by the definition of $H_{p,i} = (1-hp^{-1})H+hp^{-1}G_i$ the limit of $D_p$ will be $h\cdot[\delta \mathcal{F}_{\gamma}(H,G_1)-\delta \mathcal{F}_{\gamma}(H,G_0)]$. Further, $\varphi$ is continuous and bounded on $\I $, since by assumption $\varphi'$ exists on $\I$. Therefore, by the definition of weak convergence of signed measures \citep[see e.g.,][Ch. 8]{bogachev2007measure}, the mean shift will be 
\begin{equation}
\label{mu_f_prelim}
\mu_{\varphi} = h \int_\I  \varphi(x) d[\delta \mathcal{F}_{\gamma}(H,G_1)-\delta \mathcal{F}_{\gamma}(H,G_0)](x).
\end{equation}

We are therefore naturally lead to the study of the weak derivative. We will study it in a slightly more general setting than above, allowing for \emph{arbitrary compactly supported} probability distributions $H$ and $G$.

For any signed measure $\mu$, it will be convenient to define the \emph{companion measure} $\underline \mu = (1-\gamma) \mu + \gamma \delta_0$. The companion Stieltjes transform of a measure $\mu$ is then the Stieltjes transform of its companion $\underline \mu$. This terminology is consistent with the companion ESD, which we already used. Let $\mathcal{P}_c$ be the set of compactly supported probability measures on $([0,\infty),\mathcal{B}([0,\infty)))$. It is known that for $H\in\mathcal{P}_c$, one has $\fg(H)\in\mathcal{P}_c$ \citep{bai2009spectral}.  Our main result on the derivative of the Marchenko-Pastur map is the following:

\begin{theorem}[Weak derivative of the Marchenko-Pastur map]
\label{weak_der}
Let $\mathcal{F}_{\gamma}:\mathcal{P}_c \to \mathcal{P}_c$ be the forward Machenko-Pastur map, which takes the population spectral distribution $H$ to the limit empirical spectral distribution $\mathcal{F}_{\gamma}(H)$. Then $\fg$ has a well-defined weak derivative $\dfg(\cdot,\cdot)$, i.e., for any $H,G \in \mathcal{P}_c$, the following weak limit exists as $\ep\to 0$: 
\begin{equation*}
\frac{\fg((1-\ep)H+\ep G)-\fg(H)}{\ep} \Rightarrow\dfg(H,G). 
\end{equation*}
The limit $\dfg$ is a compactly supported signed measure with finite total variation, and has zero total mass:  $\dfg(\mathbb{R})=0$. Furthermore, 
\begin{enumerate}
\item The companion Stieltjes transform $s(z)$ of the weak derivative can be expressed as
\begin{equation}
\label{weak_st}
s(z) = -\gamma \,v'(z)\int \frac{t}{1+tv(z)} d\nu(t),
\end{equation}
where $\nu = G-H$, and v(z) is the companion Stieltjes transform of the limit empirical spectral distribution $\fg(H)$.
\item Therefore, the weak derivative is linear in the second argument: $\dfg(H,aP+bQ)$ $= a\,\dfg(H,P) + b\,\dfg(H,Q)$ for all $P,Q \in \mathcal{P}_c$, and $a,b>0$ with $a+b = 1$.
\item The distribution function of the weak derivative belongs to $L^2(\I)$. 
\end{enumerate}

\end{theorem}

 The proof of this result is given later in Section \ref{thm:weak_der}. We choose to parametrize the derivative by $G \in \mathcal{P}_c$ instead of $G-H$, because this has a more direct connection to the testing problem.
 
Weak differentiability---in contrast to the stronger Fr\'echet differentiability---does not imply the linearity in the second argument; this must be established separately. It is possible that the Marchenko-Pastur map is Fr\'echet differentiable, but we do not need that.
Further, the equation \eqref{weak_st} is important, because it enables the efficient computation of the weak derivative.


By studying further properties of the weak derivative, we show that the power to detect PCs is unity above the phase transition (see Section \ref{prop_lss}).

\subsubsection{Finishing the proof}

We now have the tools to finish the proof of Theorems \ref{thm:lss} and \ref{thm:lss_part2}. 

\begin{proof}[Proof of  Theorem \ref{thm:lss} (continued)] 
So far, in Section \ref{sec:weak_der} we established that under the null $T_p(\varphi) -c_p \Rightarrow \mathcal{N}(0,\sigma_{\varphi}^2)$, while under the alternative  $T_p(\varphi) -c_p \Rightarrow \mathcal{N}(\mu_{\varphi},\sigma_{\varphi}^2)$, for a sequence of constants $c_p$. It follows from Eq. (1.17) on p. 564 of \cite{bai2004clt} that the variance $\sigma_{\varphi}^2$ has the form stated in Theorem \ref{thm:lss} (see \eqref{sigma_f}), while we showed that $\mu_{\varphi}$ has the form in \eqref{mu_f_prelim}. 

Recall that the distribution function of the weak derivative was defined by $\Delta(x) = D((-\infty,x])$, where $D=\dfg(H,G_1)-\dfg(H,G_0)$. Since $H$ and $G_i$ are compactly supported, from Theorem \ref{weak_der} it follows that the $\dfg$---and $D$--- are also compactly supported. The compact interval $\I=[a,b]$ is such that it includes this support. Since $D$ has zero total  mass, $\Delta(x)=0$ for $x\le a$ and for $x\ge b$. Using the integration by parts formula for the Lebesgue-Stieltjes integral, which is valid since $\varphi$ is absolutely continuous, and $D$ is a bounded Borel measure on $\mathcal{I}=[a,b]$ with $\Delta(a)=\Delta(b)=0$,  \cite[see e.g.][Ex. 5.8.112]{bogachev2007measure}, we thus have 
\begin{equation*}
\mu_{\varphi} = h \int_\I \varphi(x) d[\dfg(H,G_1)-\dfg(H,G_0)](x) = -h \int_\I \varphi'(x) \Delta(x) dx.
\end{equation*}
This shows the asymptotic equivalence to the normal problem stated in Theorem \ref{thm:lss}, and finishes its proof.
\end{proof}

We will now proceed to prove Theorem \ref{thm:lss_part2}.

\begin{proof}[Proof of  Theorem \ref{thm:lss_part2}]
To optimize over $\varphi$, we will use properties of the Hilbert space $L^2(\I)$ and its inner product $\la g,j\ra = \int_\I g(x) j(x) dx$.  Let us write $g = \varphi' \in L^2(\I)$. We are optimizing over $\varphi \in \mathcal{W}(\I)$, which by the definition of $\mathcal{W}(\I)$ is equivalent to optimizing over $\varphi' = g \in L^2(\I)$. The mean and variance are $\mu = -h \la g,\Delta\ra $, and $\sigma^2 = \la g,Kg\ra $. The expression $\mu = -h \la g,\Delta\ra $ is valid because $\Delta \in L^2(\I)$ by Theorem \ref{weak_der}. 

Therefore the efficacy optimization is equivalent to the problem of maximizing $\theta(g) = -h  \la g,\Delta\ra /$ $\la g,Kg\ra^{1/2}$ over $g \in L^2(\I)$. The following  lemma, proved in Section \ref{pf:ell2_opt}, gives the desired answer. 

\begin{lemma}
\label{ell2_opt} Consider maximizing $\theta(g)$ over $g\in L^2(\I)$. If $\Delta\notin \textnormal{Im}(K)$, the supremum is $+\infty$. It is achieved for $g$ such that $Kg=0$ and $\la g,\Delta\ra<0$. If $\Delta\in \textnormal{Im}(K)$, the maximum is  $h \la \Delta, K^+\Delta\ra^{1/2}$, and is achieved for $g$ such that $Kg = -\eta \Delta$, for some $\eta>0$. 
\end{lemma}
The conclusion of Theorem \ref{thm:lss_part2} follows immediately from the above lemma, and finishes the proof. 
\end{proof}

Finally, we can prove Theorem \ref{pow_lss}. 

\begin{proof}
Consider first the choice of the critical values $t_{\varphi}^-,t_{\varphi}^+$ for a fixed $\varphi$. From Theorem \ref{thm:lss}, under the null $\smash{T_p(\varphi) - c_p(\varphi) \Rightarrow \mathcal{N}(0,\sigma_{\varphi}^2)}$, while under the alternative $\smash{T_p(\varphi) - c_p(\varphi)  \Rightarrow\mathcal{N}(\mu_{\varphi},\sigma_{\varphi}^2)}$. If the effect size of $\varphi$ is 0, $\mu_{\varphi}=0$, then using $\varphi$ leads to trivial power, so we will examine $\mu_{\varphi}<0$ in the remainder; the case $\mu_{\varphi}>0$ is analogous.

If $\sigma_{\varphi}>0$, the asymptotically optimal choices are $\smash{t_{\varphi}^- =m_\varphi + \sigma_{\varphi} z_{\alpha}}$ and $t_{\varphi}^+ = +\infty$; while the asymptotic power equals $\Phi\left(z_{\alpha}+|\mu_{\varphi}|/\sigma_{\varphi}\right)$. If $\sigma_{\varphi}=0$, then we can take $t_\varphi^{\pm} = m_\varphi \pm \ep$ for any $\ep>0$, and still have asymptotic level $\alpha$. Moreover, the test statistic converges in probability two dfferent values---0 and $\mu_{\varphi}$---under the null and the alternative. Therefore, the power of such a test is asymptotically unity for small $\ep$. 
We conclude that the maximal power over analytic functions is $\beta = \sup_{\varphi\in\mathcal{H}(\I)} \Phi(z_{\alpha}+|\mu_{\varphi}|/\sigma_{\varphi}) = \Phi(z_{\alpha}+\theta^*(\mathcal{H}(\I)))$. 
Here we used the convention of defining $\mu_{\varphi}/\sigma_{\varphi}$ as 0 or $+\infty$ in corner cases. 

We now show that the efficacy over the set of analytic functions $\mathcal{H}(\I)$ equals the efficacy over $\mathcal{W}(\I)$, because the optimal LSS can be approximated arbitrarily well---in an $L^2$ sense---by analytic functions. 

\begin{lemma}[Optimal Linear Spectral Statistics over $\mathcal{H}(\I)$]
\label{thm:lss_part3} The efficacy over the set of analytic functions $\mathcal{H}(\I)$ equals that over $\mathcal{W}(\I)$: $\theta^*(\mathcal{H}(\I))=\theta^*(\mathcal{W}(\I))$. There is a sequence $\varphi_n\in\mathcal{H}(\I)$ such that $\theta(\varphi_n) \uparrow  \theta^*(\mathcal{W}(\I))$. 
\end{lemma}

The proof is in Section \ref{pf:lss_part3}. From Lemma \ref{thm:lss_part3} we conclude that $\beta= \Phi(z_{\alpha}+\theta^*(\mathcal{W}(\I)))$. Now Theorem \ref{thm:lss_part2} shows that $\theta^*(\mathcal{W}(\I))$ has the desired form, finishing the proof. 
\end{proof}

\subsection{Proof of Theorem \ref{weak_der}}
\label{thm:weak_der}

\begin{proof}[Proof of Theorem \ref{weak_der}]
Let us denote the finite difference $D_\ep = [\fg((1-\ep)H+\ep G)-\fg(H)]/\ep$. The following lemma shows that the Stieltjes transform of $D_\ep$ converges as $\ep \to 0$.

\begin{proposition}
\label{der_st}
The Stieltjes transform $s_\ep(z)$ of $D_\ep$ converges to $s(z)$ from \eqref{weak_st}, as $\ep \to 0$, for all $z\in \mathbb{C}^+$.
\end{proposition}

The proof is given later (in Section \ref{pf:der_st}). Recall now that a sequence of signed measures $\mu_n$ on $\mathbb{R}$ endowed with the Borel sigma-algebra $\mathcal{B}$ \emph{converges vaguely} to the signed measure $\mu$, denoted $\mu_n \Rightarrow_v \mu$, if $\mu_n(\varphi) \to \mu(\varphi)$, for all bounded continuous $\varphi$ of compact support. Due to Proposition \ref{der_st}, it follows that there is a unique signed measure $\dfg(H,G)$ such that $D_\ep \Rightarrow_v\dfg(H,G)$. Indeed, by Theorem B.9 from \cite{bai2009spectral}, we only need to notice that $D_\ep$ is a finite signed measure (since it is a difference of two positive finite measures), and $D_\ep(-\infty)=0$ (since this is true for $\fg$). By the result cited above, it follows that the vague limit exists as a signed measure $\dfg$ with finite total variation, and that $s(z)$ is its Stieltjes transform. 

Next, if $H$ is compactly supported within an interval $[0,M]$, it follows that $\fg(H)$ is compactly supported in the interval $[0,(1+\sqrt{\gamma})^2M]$ \citep{bai2009spectral}. Therefore, if $H$ and $G$ are compactly supported, then $\fg(H)$, $\fg((1-\ep)H+\ep G)$ and $D_\ep$ are uniformly compactly supported for all $\ep>0$. Hence $\dfg$ is compactly supported. Furthermore, the vague convergence $D_\ep \Rightarrow \dfg$ is equivalent to weak convergence, as required. 

Clearly $D_\ep$ has zero total measure for all $\ep$, hence $\dfg$ has zero total measure. This establishes the claims about convergence---including Claim 1---stated in Theorem \ref{weak_der}. 

For the Claim 2 in Theorem \ref{weak_der}, the explicit formula for $s(z)$ shows that the Stieltjes transform is linear with respect to the second argument, i.e., denoting by $s_\mu$ the Stieltjes transform of the signed measure $\mu$ (and omitting $\gamma$ from $\dfg$):
$$
s_{\df(H,aP+bQ)}(z) = a\, s_{\df(H,P)}(z) + b \, s_{\df(H,Q)}(z).
$$
By the uniqueness of Stieltjes transforms of signed measures, Theorem B.8 of \cite{bai2009spectral}, it follows that the weak limit is itself linear in the second argument.

For Claim 3 in Theorem \ref{weak_der}---which states that the cdf $\Delta \in L^2(\I)$---we argue as follows: since $\dfg$ is a signed measure with finite total variation, it can be written as the difference of its positive and negative parts, $\mu=\mu_+-\mu_-$, by the Jordan decomposition theorem (see e.g., \cite{bogachev2007measure} Vol. I. p. 176). Therefore, $\Delta$ can be written as the difference of their distribution functions, $\Delta(x) = L_+(x)-L_-(x)$. The d.f.-s of the positive finite measures $\mu_\pm$ are nondecreasing, and $|\Delta(x)| \le L_+(x)+L_-(x) \le  L_+(b)+L_-(b) = \|\dfg\|_{TV} <\infty$. Therefore, $\Delta$ is bounded on the compact interval $\I$, and hence square integrable. 
\end{proof}

\subsection{Proof of Proposition \ref{der_st}}
\label{pf:der_st}

\begin{proof}
Let us denote by $v_\mu$ the companion Stieltjes transform of a measure $\mu$. By linearity, the Stieltjes transform of $D_\ep$ equals
$$
s_\ep = \frac{v_{\mathcal{F}(H+\ep \Delta)}-v_{\mathcal{F}(H)}}{\ep}.
$$
The Silverstein equation \citep{marchenko1967distribution,silverstein1995analysis} for $H$ and $H + \ep \nu$ shows that  for $z\in\mathbb{C}^+$ (omitting the argument $z$ from the Stieltjes transforms)
\begin{align*}
-\frac{1}{v_{\mathcal{F}(H)}} &= z - \gamma \int \frac{t \, dH(t)}{1 + tv_{\mathcal{F}(H)}},\\
-\frac{1}{v_{\mathcal{F}(H + \ep \nu)}} &= z - \gamma \int \frac{t \, d[H+ \ep \nu](t)}{1 + tv_{\mathcal{F}(H + \ep \nu)}}.
\end{align*}
Let us denote $v = v_{\mathcal{F}(H)}$ and $v_\ep = v_{\mathcal{F}(H + \ep \nu)}$. Subtracting the first equation from the second one, and rearranging, we find
\begin{equation}
\frac{v_\ep - v}{\ep}
\left\{\frac{1}{v_\ep v} - \gamma \int \frac{t^2 \, dH(t)}{[1 + tv_\ep][1 + tv]}\right\} = - \gamma \int \frac{t \, d\nu(t)}{1 + tv_\ep}.
\label{raw_st_eq}
\end{equation}

To take the limit as $\ep\to0$, we use the following Lemma, proved in Section \ref{pf:cont_st}. 

\begin{lemma}[Continuity of the Marchenko-Pastur map]
\label{cont_st} As $\ep\to0$, $v_{\mathcal{F}(H + \ep \Delta)}(z) \to v_{\mathcal{F}(H)}(z)$ for all $z \in \mathbb{C}^+$.
\end{lemma}

Furthermore, $v(z)\neq0$ for all $z\in\mathbb{C}^+$ follows directly from the Silverstein equation. By the bounded convergence theorem, it follows that as $\ep\to0$, we have the limits
\begin{align*}
\int \frac{t \, d\nu(t)}{1 + tv_\ep} &\to \int \frac{t \, d\nu(t)}{1 + tv}, \\
 \int \frac{t^2 \, dH(t)}{[1 + tv_\ep][1 + tv]} &\to \int \frac{t^2 \, dH(t)}{[1 + tv]^2}.
\end{align*}
Indeed, for the first claim, since $\nu = G-H$, by linearity it is enough to show the convergence for bounded probability measures. For the Stieltjes transform $m_F$ of any probability measure $F$, we have the inequality (see Corollary 3.1 in \cite{couillet2011random})

$$\left|\frac{1}{1+tm_F(z)}\right| \le \frac{|z|}{\Im{z}}, \,\, z\in \mathbb{C}^+.$$

Thus, $|t/(1 + tv_\ep)| \le t |z|/\Im{z}$. This shows that the integrand is uniformly bounded for compactly supported probability measures; and so the bounded convergence theorem applies to show the required convergence. A similar argument works for the second convergence claim. 

Therefore, as $\ep\to 0$, the term in curly braces in \eqref{raw_st_eq} tends to 
$$
\frac{1}{v^2} -\gamma\int \frac{t^2\,dH(t)}{[1 + tv]^2}.
$$
However, by directly differentiating the Silverstein equation in $z$, we see that this equals $1/v'(z) = (dv/dz)^{-1}$. The derivative exists because $v$ is analytic on $\mathbb{C}^+$. Therefore, in the limit $\ep\to0$ \eqref{raw_st_eq} becomes
$$
s(z) = \lim_{\ep\to 0}\frac{v_\ep - v}{\ep} = -\gamma \,v'(z)\int \frac{t}{1+tv(z)} d\nu(t),
$$
as required. This finishes the proof.  
\end{proof}

\subsection{Proof of Proposition \ref{prop_df}}
\label{pf:prop_df}

{\renewcommand{\arraystretch}{1.1}
\begin{proof}

We will study the behavior of the companion Stieltjes transform of $\dfg$. From Theorem \ref{weak_der}, this satisfies the equation
\begin{equation*}
s(z) = -\gamma \,v'(z)\left(\sum_{j=1}^{l} \frac{u_js_j}{1+s_jv(z)} -  \sum_{i=1}^{k} \frac{w_it_i}{1+t_iv(z)} \right).
\end{equation*}

The Stieltjes transform fully characterizes the distribution function (d.f.) $\Delta$ of $\dfg$. Indeed, first we note that $\Delta$ is almost everywhere (a.e.) continuous. This follows because every signed measure $\mu$ with finite total variation on $(\mathbb{R},\mathcal{B})$ can be written as the difference of its positive and negative parts, $\mu=\mu_+-\mu_-$ by the Jordan decomposition theorem (see e.g., \cite{bogachev2007measure} Vol. I. p. 176). The d.f.-s of the positive finite measures $\mu_\pm$ are continuous a.e., hence the d.f. of $\mu $ is also continuous a.e.

Next, by the inversion formula for signed measures with finite total variation, (see Theorem B.8 in \cite{bai2009spectral}), for all points of continuity $a<b$ of $\Delta$,
\begin{equation}
\label{inversion_st}
\Delta(b) -\Delta(a)= \lim_{\ep\downarrow 0}\frac{1}{\pi} \int_{a}^{b} \Im(s(y+i\ep))dy.
\end{equation}

Therefore $\Delta$ is determined on intervals $(a,b]$ with $a,b$ belonging to a set of full Lebesgue measure; and hence is fully determined.

Further, if $j(x) = \pi^{-1}\lim_{\ep\downarrow 0}s(x+i\ep)$ exists, then $\Delta$ is differentiable at $x$ with derivative $j(x)$ (see Theorem B.10 in \cite{bai2009spectral}).

Hence we study the limit behavior of $s(z)$ near the real line, as $z\downarrow x \in \mathbb{R}$. This can be deduced from the formula for $s(z)$, and the behavior of the companion Stieltjes transform $v(z)$ of $\fg$, which is well understood \citep{silverstein1995analysis}.

\begin{table}[]
\centering
\caption{Behavior of the Stieltjes transform of $\dfg$ near the real axis, and properties of its density.}
\label{t2}
\begin{tabular}{|c|c|c|c|}
\hline
\multirow{2}{*}{$x$} & \multirow{2}{*}{$\in \text{int}(S)$}  &\multicolumn{2}{|c|}{$\in S^c$}      \\ \cline{3-4} 
                     &                                      &                           $v(x)\neq-1/s_i$ & $v(x)=-1/s_i$    \\ \hline
$v'(x)$              & $\in \mathbb{C}$                     &               $\in \mathbb{R}$ & $\in \mathbb{R}$ \\ \hline
$s_i/(1+s_iv(x))$    & $\in \mathbb{C}$                     &             $\in \mathbb{R}$ & diverges         \\ \hline
$s(x)$               & $\in \mathbb{C}$                     &                         $\in \mathbb{R}$ & diverges         \\ \hline
$\dfg$               & $(\exists)$ density                     &                    0 density & point mass         \\ \hline
\end{tabular}
\end{table}

It is helpful to do the analysis separately depending on whether or not $x$ belongs to the interior of $S$. The different cases are examined below, and summarized in Table \ref{t2}. We remark that the edges of $S$ are more delicate to analyze, because $v'$ has a singularity at the edges, and $1+s_j v=0$ may happen; this is not required in the present proof and is left for future work. 
\begin{itemize}

\item If $x$ belongs to the interior of the support $S$, then the limit $\ld v(x+i\ep) = v(x)$ exists with $\Im(v(x))>0$. The limit $\ld v'(x+i\ep) = v'(x)$ also exists  \citep[see ][for both claims]{silverstein1995analysis}. Hence the limit $\ld s(x+i\ep) = s(x)$ exists. This shows that $\dfg$ has a density at all $x\in\textnormal{int}(S)$.

\item If $x$ belongs to the complement of the support, $S^c$, then the limit $\ld v(x+i\ep) = v(x)$ exists with $\Im(v(x))=0$. The limit $\ld v'(x+i\ep) = v'(x) \in \mathbb{R}$ also exists \citep[again, see ][for these claims]{silverstein1995analysis}. Therefore, if $x$ is such that  $v(x)\neq -1/t_i$ and $v(x)\neq -1/s_j$, the limit $s(x) \in \mathbb{R}$ also exists. 

Now $v(x)$ does not take the values $-1/t_i$. This follows because,  by continuity,  $v(x)$ is a  solution to the Silverstein equation, which by inspection cannot have that root. Therefore, $v(x)\neq -1/s_j$ guarantees that $s(x)\in \mathbb{R}$  exists. Then $\dfg$ has a density equal to 0 at $x$.

\item If $x\in S^c$, but $v(x)=-1/s_j$ for some $j$, then we will show $\dfg$ has a point mass at $x$. This will be based on a lemma, proved later in Section \ref{pf_lemma:pointmass}.

\begin{lemma}
\label{lemma:pointmass}
Let $\mu$ be a signed measure with finite total variation on $(\mathbb{R},\mathcal{B})$, with Stieltjes transform denoted $s_\mu$. Suppose $s_\mu$ is complex analytic in the neighborhood of a point $x\in\mathbb{R}$, but has a residue $-w$ at $x$. Then $\mu$ has a point mass $w$ at $x$, i.e., $\mu(\{x\})=w$, while $\mu((x-\delta,x+\delta))=w$ for small $\delta>0$. 
\end{lemma}

To use this lemma, we evaluate the residue as $\ld i \ep \cdot s(x + i\ep)$. Since the $s_k$ are distinct, all terms $1+s_k v(z)$ have finite limits, except $1+s_j v(z)$, which tends to 0 by assumption. Similarly, all terms $1+t_k v(z)$ have finite limits. Therefore, in the limit as $\ep\to0$, $1/[1+s_jv(x + i\ep)]$ diverges. However, by definition and by continuity of $v'$, $i\ep/[v(x + i\ep)+1/s_j] \to 1/v'(x)$. Therefore

$$\ld i \ep \cdot s(x + i\ep) = - v'(x) \gamma \ld i \ep \frac{u_js_j}{1+s_jv(x + i\ep)}  = -\gamma u_j.$$
Therefore, by Lemma \ref{lemma:pointmass}, $\dfg$ has point mass $\gamma u_j$ at $s_j$.

\end{itemize}
This finishes the proof. 
\end{proof}

\section*{Acknowledgments}

We are grateful to David Donoho for his encouragement, inspiring guidance and feedback on the manuscript; and to Iain Johnstone for his enthusiastic interest and helpful suggestions.

{\small
\setlength{\bibsep}{0.2pt plus 0.3ex}
\bibliographystyle{plainnat-abbrev}
\bibliography{references}
}

\section{Supplement}

These sections contain proof details (Sec. \ref{Proof details}), implementation details (Sec. \ref{Implementation}), and empirical motivation  (Sec. \ref{empirical_evidence}). 

\section{Proof details}
\label{Proof details}

\subsection{Proof of Lemma \ref{cont_st}}
\label{pf:cont_st}

\begin{proof}
Note that all $v_\ep$ and $v$ are analytic on $\mathbb{C}\setminus\mathbb{R}^+$, since $\mathcal{F}(H)$ is supported on $\mathbb{R}^+$. Therefore it is enough to show the convergence for real $z = -\lambda$ in a set $(-\infty,M)$, for some $M<0$. The convergence on $\mathbb{C}\setminus\mathbb{R}^+$ follows by Vitali's Theorem (see Lemma 2.14 in \cite{bai2009spectral}). Now, by slightly rewriting Eq \eqref{raw_st_eq}, we find
\begin{equation}
\label{raw_st_eq_2}
\left(\frac{v_\ep}{v} - 1\right)
\left\{1 - \gamma \int \frac{t^2 v_\ep v  \, dH(t)}{[1 + tv_\ep][1 + tv]}\right\} = - \ep\gamma \int \frac{t v_\ep\, d\nu(t)}{1 + tv_\ep}.
\end{equation}
We will show that the term in the curly braces in \eqref{raw_st_eq_2} is bounded, while the right hand side tends to 0 as $\ep\to0$. First, note that $v(-\lambda) = \EE{(X+\lambda)^{-1}}\le1/M$. Next, because $u \to (tu)/(1+tu)$ is increasing for $t\ge0$, 
$$ 0 \le \int \frac{t^2 v_\ep v  \, dH(t)}{[1 + tv_\ep][1 + tv]} \le  \int \frac{(t/\lambda)^2  \, dH(t)}{[1 + t/\lambda]^2}=\int \frac{t^2  \, dH(t)}{[\lambda + t]^2} \to 0,$$
as $\lambda \to \infty$, by the dominated convergence theorem. Therefore, for $\lambda >M$ large enough, the term in the curly braces in \eqref{raw_st_eq_2} is contained in a bounded interval $(a,1)$ for some $a>0$. 
Similarly, 
$$
A(\lambda) = \left|\int \frac{t v_\ep\, d\nu(t)}{1 + tv_\ep}\right| \le \int \left| \frac{t v_\ep\, }{1 + tv_\ep}\right| d|\nu|(t) \le \int \left| \frac{t/\lambda }{1 + t/\lambda}\right|  d|\nu|(t) \to 0,
$$
as $\lambda \to \infty$, by the dominated convergence theorem, where we have denoted the measure $|\nu| = H+G$. For $M$ large, we have $A(\lambda)<1$ for all $\lambda>M$. Therefore the right hand side of \eqref{raw_st_eq_2} tends to 0 as $\ep\to0$. It follows that $v_\ep/v-1\to0$. Since $v\neq0$, this shows that $v_\ep \to v$, as required.
\end{proof}

\subsection{Proof of Lemma \ref{ell2_opt}}
\label{pf:ell2_opt}

\begin{proof}
Let $\overline{\textnormal{Im}}(K)$ be the closure of the image of $K$. We treat three cases. 

\subsubsection{$\Delta\notin \overline{\textnormal{Im}}(K)$}
If $\Delta\notin \overline{\textnormal{Im}}(K)$, then there is a function $g\in L^2(\I)$ such that $\la g,Kg\ra=0$, and $\la g,\Delta \ra>0$. Indeed, let $\Pr$ be the orthogonal projection operator onto $\overline{\textnormal{Im}}(K)$, well-defined since it is a closed subspace, and consider $g = \Delta  - \Pr(\Delta ) \neq 0$.  $g$ is orthogonal to $\overline{\textnormal{Im}}(K)$, so $\la g,Kg\ra=0$. Further $\la g,\Delta \ra = \la g,g \ra>0$. Choosing $\tilde g = -g$, we have $\theta(\tilde g) = +\infty$. This shows that the efficacy is $+\infty$ in this case. 

Further, a $g$ with the above properties can exist only in this case. Indeed, suppose that $\la g, Kg \ra=0$ and $\Delta\in \overline{\textnormal{Im}}(K)$. Then there is a sequence $\Delta_n = Kg_n$ with $\Delta_n\to \Delta$, implying that $|\la g, \Delta \ra| $ $\le \lim\sup |\la g, \Delta_n \ra| $ $\le \lim\sup |\la g, Kg \ra|^{1/2} $ $|\la g_n, Kg_n \ra|^{1/2} =0$. Hence $\la g, \Delta \ra=0$. Therefore, $\Delta\notin \overline{\textnormal{Im}}(K)$ is the only case when there is a $g$ such that $\theta(g)=+\infty$. In the remaining cases we can restrict to $g$ such that $|\la g, Kg \ra|>0$ without decreasing the objective. 

To finish this case, it remains to find the optimal LSS. If the supremum is $+\infty$, the derivative $g$ of an optimal LSS must obey $\la g,Kg\ra=0$ and $\la g,\Delta \ra<0$. Defining $K^{1/2}$ by its spectral decomposition, which exists since $K$ is a compact self-adjoint operator, the first equality is equivalent to $\|K^{1/2}g\|=0$, i.e., $K^{1/2}g=0$. Clearly the last equation is also equivalent to $Kg=0$, which proves the desired claim---$Kg=0$, $\la g,\Delta \ra<0$---for the optimal LSS.

\subsubsection{$\Delta\in \textnormal{Im}(K)$}
In the remaining case, suppose first that $\Delta\in \textnormal{Im}(K)$, so $Kl=-\Delta$ for some $l\in L^2(\I)$. Then, we have using the Cauchy-Schwarz inequality that
$$\mu = h \la Kl,g\ra \le h \|K^{1/2}l\| \|K^{1/2}g\| = h\sigma \|K^{1/2}l\|.$$
If $\sigma^2>0$, this shows that $\mu/\sigma \le h \|K^{1/2}l\|$. Since this bound is true for all $l$, we will choose $l$ to make the bound tight. Let $l_0 = - K^+\Delta$, where $K^+$ is the generalized inverse of $K$. Since $\Delta \in \textnormal{Im}(K)$, $K^+\Delta$ is well-defined \citep[see e.g.,][p. 115]{groetsch1977generalized}, and is the minimum norm solution to the equation $Kl=-\Delta$. In terms of $l_0$, we can write the upper bound as $h\la Kl_0,l_0\ra^{1/2} = h \la \Delta, K^+\Delta\ra^{1/2}$. 
 
Hence, for any test statistic, the efficacy $\mu/\sigma$ is at most $\theta^* = h  \la \Delta, K^+\Delta\ra^{1/2}$. The optimum is achieved when equality occurs in Cauchy-Schwarz, i.e., $\eta K^{1/2}l_0 = K^{1/2}g$ for some $\eta>0$. Hence $g = - \eta K^{+}\Delta + u$ for some $u$ such that $K^{1/2}u=0$. Since $K^{1/2}u=0$ if and only if $Ku=0$, this implies that the optimal set is described by $g$ such that $Kg = -\eta \Delta$, for $\eta>0$. 

The case $\sigma^2=0$, which was not treated above, occurs when $K^{1/2}g=0$, which implies $Kg=L=0$. In this case, clearly the objective value equals 0 identically, and the maximum is 0. Any test statistic has zero efficacy. The formula $\theta^* = h  \la \Delta, K^+\Delta\ra^{1/2}=0$ is still valid.

\subsubsection{$\Delta\in \overline{\textnormal{Im}}(K)$, but $\Delta\notin \textnormal{Im}(K)$}

The last case is when $\Delta\in \overline{\textnormal{Im}}(K)$, but $\Delta\notin \textnormal{Im}(K)$. In this case, we find it simplest to use the spectral decomposition of $K$ explicitly. In order to make the proof as geometric as possible, so far we avoided its use; however the previous properties can be also be derived from the spectral decomposition. Let then $k_1 \ge k_2 \ge \ldots \ge 0$ be the eigenvalues of $K$, which obey $k_i\to0$ since $K$ is compact. 

Rotating to the eigenbasis of $K$, we can write the objective as $\theta(g)=-h \sum_{i=1}^{\infty} g_i l_i / ( \sum_{i=1}^{\infty} g_i^2 k_i)^{1/2}$, where $g_i,l_i$ are the coefficients of $g$ and $\Delta$ in the eigenbasis. With $T = \{i: k_i>0\}$, clearly $\Delta\in \overline{\textnormal{Im}}(K)$ if and only if $l_i=0$ for $i \notin T$. Furthermore, since $\Delta\notin \textnormal{Im}(K)$, we must have $\sum_{i\in T} l_i^2/k_i = +\infty$. In particular, taking $T_M = T\cap \{1,\ldots,M\}$, defining $g^M$ by $g^M_i = - l_i/k_i$ for $i\in T_M$, and $0$ otherwise, we have $\theta(g^M)=h\cdot(\sum_{i\in T_M} l_i^2/k_i)^{1/2}\to+\infty$ as $M\to\infty$. Therefore, the objective is unbounded in this case. This finishes the proof.
\end{proof}

\subsection{Proof of Lemma \ref{thm:lss_part3}}
\label{pf:lss_part3}
\begin{proof}
Since $\mathcal{H}(\I) \subset \mathcal{W}(\I)$, clearly $\theta^*(\mathcal{H}(\I)) \le \theta^*(\mathcal{W}(\I))$. To show equality, it is enough to exhibit a sequence of functions $\varphi_n\in\mathcal{H}(\I)$ such that $\theta(\varphi_n) \uparrow  \theta^*(\mathcal{W}(\I))$. For this it is enough to show that for any $\varphi\in \mathcal{W}(\I)$, there is a sequence of functions $\varphi_n\in\mathcal{H}(\I)$ such that $\theta(\varphi_n) \to  \theta(\varphi)$. Note, specifically, that this is still enough even in the corner case when $\theta^*(\mathcal{W}(\I)) = +\infty$.

It is known that the set of analytic functions $ \mathcal{H}(\I)$ is dense in $L^2(\I)$ in the topology induced by the $L^2$ norm.  Therefore for any $\varphi\in\mathcal{W}(\I)$, there is a sequence $g_n \in \mathcal{H}(\I)$, such that $\|g_n-\varphi'\|_2\to0$. Let $G_n \in \mathcal{H}(\I)$ be indefinite integrals of $g_n$. Since the numerator $\mu_{\varphi}$ of $\theta(\varphi)$ is a linear function of $\varphi'$, which is in particular $L^2$-continuous, it follows that $\mu_{G_n} \to \mu_{\varphi}$. Since the operator $K$ is compact, the denominator $\sigma_{\varphi} = \la \varphi',K\varphi'\ra^{1/2}$ is also continuous, therefore $\sigma_{G_n} \to \sigma_{\varphi}$. If $\sigma_{\varphi}>0$, it follows that $\theta(G_n)\to\theta(\varphi)$, which proves the desired claim. If $\sigma_{\varphi}=0$, then the conclusion follows from the definition of $\theta(\varphi)$: indeed, if $\mu_{\varphi}\le0$, then we can take the sequence $\varphi_n(x)=0$, while if $\mu_{\varphi}>0$, then the above sequence $G_n$ will satisfy $\theta(G_n) \to +\infty$, as required. This finishes the proof. 
\end{proof}

\subsection{Proof of Lemma \ref{lemma:pointmass}}
\label{pf_lemma:pointmass}
\begin{proof}
Take a small rectangular contour $C$ around $x$ in the following way: let $a<x<b$, and let the contour move clockwise along the edges of the rectangle with vertices $(a,\pm\ep)$, $(b,\pm\ep)$. Take $a,b$ close enough to $x$ that $s_\mu$ is analytic at all points but $x$ in the rectangle. 
Using that $s_\mu(\bar z) = \bar s_\mu(z)$, we can express the contour integral of $s_\mu$ as
\begin{align*}
\oint_C s_\mu(z) dz &= \int_a^b s_\mu(y+i\ep) dy + \int_b^a s_\mu(y-i\ep) dy +O(\ep) \\
&= 2i \int_a^b \Im(s_\mu(y+i\ep)) dy +O(\ep).
\end{align*}
Combining the above equation with the Cauchy residue formula and with the inversion formula for Stieltjes transforms \eqref{inversion_st}, we obtain as $\ep\to0$ that $\Delta(b)-\Delta(a) = w$ (where $\Delta$ is the distribution function). Since this holds for all $a,b$ in a neighborhood of $x$ such that $a<x<b$, it follows that $\mu$ has a point mass $w$ at $x$. 
\end{proof}

\subsection{Proof of Lemma \ref{ls3}}
\label{pf:equi_lss}

\begin{proof}
Note first that  $\hsigma^2 - m_1 = p^{-1}(\tr\hSigma - p \mathcal{F}_{\gamma_p}(x)) + \mathcal{F}_{\gamma_p}(x) - \mathcal{F}_{\gamma}(x) = O_P(p^{-1})$. Indeed, for the first term, the rate follows using the Bai-Silverstein CLT under both the null and alternative. For the second term, note that $\mathcal{F}_{\gamma_p}(x) = \int x d\mathcal{F}_{\gamma_p}(x) = \int x dH_p(x)$, where $H_p$ is the $p$-th population spectrum. 
However, $\mathcal{F}_{\gamma_p}(x) = \int x dH(x)+ p^{-1} h \int x d\Delta_i(x)$, where $\Delta_i = G_i-H$, with $i=0$ under the null and $i=1$ under the alternative, so that $\mathcal{F}_{\gamma_p}(x)-\mathcal{F}_{\gamma}(x) = O(p^{-1})$. Since $m_1>0$, we also have $1/\hsigma^2 - 1/m_1 = O_P(p^{-1})$.

The following analysis applies \emph{both} under the null and the alternative; in particular ``almost surely'' means ``almost surely'' under both the null and the alternative. 
The sample eigenvalues $\lambda_i$ all belong to the compact interval $\I$, almost surely \citep{bai2009spectral}. Moreover, $\hsigma^2 \to m_1$ almost surely. This is analogous to the error rate above: First, $p^{-1}(\tr\hSigma - p \mathcal{F}_{\gamma_p}(x)) \to 0$ almost surely, by the strong law of large numbers. Combining this with the convergence $\smash{\mathcal{F}_{\gamma_p}(x)-\mathcal{F}_{\gamma}(x) = O(p^{-1})}$ that we saw above, we obtain the desired claim that $\hsigma^2 \to m_1$ almost surely. 

Therefore, $\lambda_i/\hsigma^2$ belong to the compact interval $\I/m_1$ almost surely. By assumption $\varphi$ is analytic on $\I/m_1$, and in particular it is twice differentiable with uniformly bounded second derivative. Therefore, 
$$
\varphi\left(\frac{\lambda_i}{\hsigma^2}\right) - \varphi\left(\frac{\lambda_i}{m_1}\right) = \left(\frac{\lambda_i}{\hsigma^2} - \frac{\lambda_i}{m_1}\right) \varphi'\left(\frac{\lambda_i}{m_1}\right) + O\left(\left(\frac{\lambda_i}{\hsigma^2} - \frac{\lambda_i}{m_1}\right)^2\right). 
$$
 From the above discussion, the error term is of order $O_P(\lambda_i^2/p^2)$. Summing over $i=1,\ldots,p$, 
\begin{align*}
R_p = S_p(\varphi) - \sum_i \varphi\left(\frac{\lambda_i}{m_1}\right) & = 
\left(\frac{1}{\hsigma^2} - \frac{1}{m_1}\right) \sum_i\lambda_i\varphi'\left(\frac{\lambda_i}{m_1}\right) 
+ O_P\left(p^{-2}\sum_i \lambda_i^2\right). \\
&
= \left(\frac{1}{\hsigma^2} - \frac{1}{m_1}\right)  \sum_i\lambda_i\varphi'\left(\frac{\lambda_i}{m_1}\right) 
+ O_P(p^{-1}). 
\end{align*}
Studying the first term on the right hand side, we rewrite it by centering as
\begin{align*}
&\left(\frac{m_1}{\hsigma^2} - 1\right) \left[ \sum_i\frac{\lambda_i}{m_1}\varphi'\left(\frac{\lambda_i}{m_1}\right) - p \mathcal{F}_{\gamma_p} \left(  \frac{x}{m_1} \varphi'\left(\frac{x}{m_1}\right)\right) \right] - \\
& \frac{\tr\hSigma-p\,m_1}{\hsigma^2} \mathcal{F}_{\gamma_p} \left(  \frac{x}{m_1} \varphi'\left(\frac{x}{m_1}\right)\right) . 
\end{align*}
The first term is $O_P(p^{-1})$, because $|m_1/\hsigma^2 - 1|=O_P(p^{-1})$, and the term multiplying it is a properly centered linear spectral statistic, which is $O_P(1)$ by the Bai-Silverstein CLT. 

For the second term, write 
\begin{align*}
\frac{\tr\hSigma-p\,m_1}{\hsigma^2} & = 
\frac{\tr\hSigma-p\,m_{\gamma_p}}{\hsigma^2} + 
\frac{p(m_{\gamma_p}-m_1)}{\hsigma^2} \\
& = 
\frac{\tr\hSigma-p\,m_{\gamma_p}}{m_1} + 
\frac{p(m_{\gamma_p}-m_1)}{m_1} + O_p(p^{-1}).
\end{align*}
The variability of this expression comes from $\tr(\hSigma)$, and is asymptotically same as that of the LSS $T_p(x/m_1)$. Putting this together with the expression for $R_p$, we see that $S_p(\varphi)$ has asymptotically the same variance as the LSS $T_p(j)$, with $j(x) = \varphi(x/m_1)-x/m_1\fg[x/m_1\varphi'(x/m_1)]$, as stated in theorem, under both the null and the alternative. 

Finally, we need to compute the asymptotic mean shift of $S_p(\varphi)$, i.e., the asymptotic difference between the centering terms under the alternative and the null.  From the above argument we see that the mean shift of $R_p$ equals the difference between the following two expressions evaluated for $\Delta_i=G_i-H$, with $i=1$ (alternative), and $i=0$ (null):
\begin{align*}
\lim_{p\to\infty} -\frac{p(m_{\gamma_p}-m_1)}{m_1} \mathcal{F}_{\gamma_p} \left(  \frac{x}{m_1} \varphi'\left(\frac{x}{m_1}\right)\right) = 
-\frac{h  \int x \, d\Delta_i(x)}{m_1} \mathcal{F}_{\gamma} \left(  \frac{x}{m_1} \varphi'\left(\frac{x}{m_1}\right)\right).
\end{align*}
This equals the asymptotic mean shift of the statistic $T_p(j_0)$, where 
$$j_0(x) = -\frac{x}{m_1}\mathcal{F}_{\gamma} \left(  \frac{x}{m_1} \varphi'\left(\frac{x}{m_1}\right)\right).$$
Putting this together with the fact that the mean shift of $S_p(\varphi)$ equals the sum of the mean shifts of $R_p$ and $T_p( \varphi(x/m_1))$, we obtain exactly that the mean shift of $S_p(\varphi)$ asymptotically equals that of $T_p(j)$. This finishes the proof. 
\end{proof}

\subsection{Proof of Lemma \ref{linearize}}
\label{pf:linearize}

\begin{proof} The Bai-Silverstein CLT \citep[][ Thm. 1.1]{bai2004clt} states that the statistics $X_p(\varphi) = T_p(\varphi) - p\mathcal{F}_{\gamma_p}(\varphi)$ are asymptotically normal, and in particular, $O_P(1)$. Therefore, by a Taylor series expansion around $a_p = (\mathcal{F}_{\gamma_p}(\varphi), \mathcal{F}_{\gamma_p}(\psi))$, we see 
\begin{align*}
y\left(p^{-1}T_p(\varphi), p^{-1} T_p(\psi)\right) &= y\left(\mathcal{F}_{\gamma_p}(\varphi) + p^{-1}X_p(\varphi), \mathcal{F}_{\gamma_p}(\psi) + p^{-1}X_p(\psi)\right) \\
&= y\left(a_p \right) + \partial_1y(a_p)  p^{-1}X_p(\varphi) + \partial_1y(a_p)  p^{-1}X_p(\psi) \\
 &+ o\left(\max\{ |p^{-1}X_p(\varphi)|,|p^{-1}X_p(\psi)|\}\right). 
\end{align*}

Now by assumption $\partial_1y(a) X_p(\varphi) + \partial_2y(a) X_p(\psi)$ has a non-trivial limit distribution, which follows from the Bai-Silverstein CLT and the assumption that $\sigma_h^2>0$.  Since $c_p = p\cdot y(a_p)$ is a constant and $ o(\max\{ |X_p(\varphi)|,|X_p(\psi)|\}) \to 0$ in probability, we see that 
\begin{align*}
p \cdot y\left(p^{-1}T_p(\varphi), p^{-1} T_p(\psi)\right) &= \partial_1y(a_p) X_p(\varphi) + \partial_2y(a_p) X_p(\psi) + c_p + o_P(1) \\
&= \partial_1y(a) X_p(\varphi) + \partial_2y(a) X_p(\psi) + c_p + o_P(1),
\end{align*}

because $(\partial_1y(a_p) - \partial_1y(a)) X_p(\varphi) \to 0$ in probability by Slutsky's theorem. Indeed, the constants $\mathcal{F}_{\gamma_p}(\varphi) \to \fg(\varphi)$, because the Marchenko-Pastur map is weakly continuous as a function of $\gamma$ \citep{silverstein1995analysis}. Therefore $a_p\to a$. Since $y$ is continuously differentiable at $a$, it then follows that $(\partial_1y(a_p) - \partial_1y(a))\to0$. Now, since $X_p(\varphi)  = O_P(1)$, we conclude that $(\partial_1y(a_p) - \partial_1y(a)) X_p(\varphi) \to 0$ in probability, as claimed. 

After proper centering, the above calculations imply that $p \cdot Y_p= T_p(j) + d_p + o_P(1)$, for a sequence of constants $d_p$, as desired.
\end{proof}

\subsection{Proof of LSS equivalence for examples in Section \ref{examples}}
\label{pf:examples}

\begin{proof}

We use the linearization lemma \ref{linearize} to show the equivalence of classical tests with LSS. We denote $m_i=\fg(g_i)$.
\begin{itemize}

\item For the log-LRT of sphericity, $T = p \log(p^{-1}\tr \hSigma) - \log\det\hSigma$, we take $y(r,s) = \log(r)-s$, and the two LSS $\varphi(x) = x$ and $\psi(x) =  \log(x)$. Then $T = p\cdot y(p^{-1}T_p(\varphi),p^{-1}T_p(\psi))$. Further $\partial_1y(r,s) = 1/r$, and $\partial_2y(r,s) = -1$, while $a = (\fg(\varphi),\fg(\psi)) = (m_1,\fg(\log(x))$, so that the equivalent LSS is $j(x) =\partial_1 y(a) \cdot \varphi(x) + \partial_2 y(a)\cdot \psi(x) = x/m_1-\log(x)$. By definition, this is also equivalent to $\varphi_2(x) = x/m_1 - \log(x/m_1)-1$. 

This result can also be deduced from Lemma \ref{ls3}, by taking $\varphi(x) = \log(x)$. Then the associated LS$^3$ is $j(x) = \varphi(x/m_1) - x/m_1 \cdot \fg(x/m_1\cdot m_1/x) = \log(x/m_1) - x/m_1$.

The remaining examples are similar. 
\item \cite{john1971optimal}'s test of sphericity can be shown to be equivalent to $\tr \hSigma^2/ (\tr \hSigma)^2$ after some algebra. We then take $y(r,s) = r/s^2$ and the two LSS $\varphi(x) = x^2$ and $\psi(x) = x$. Clearly $\partial_1y(r,s) = 1/s^2$, and $\partial_2y(r,s) = -2r/s^3$, while $a = (\fg(\varphi),\fg(\psi)) = (\fg(g_2),\fg(g_1)) = (m_2,m_1)$, so that the equivalent LSS is $j(x) =\partial_1 y(a) \cdot \varphi(x) + \partial_2 y(a)\cdot \psi(x) = x^2/m_1^2-2x \cdot m_2/m_1^3$. Multiplying this by $m_1^2$ leads to an equivalent LSS, and gives the claimed result. 

Again, this can be also deduced from Lemma \ref{ls3}, by taking $\varphi(x) = x^2$. Then the associated LS$^3$ is $j(x) = \varphi(x/m_1) - x/m_1 \fg(x/m_1 \cdot 2 x/m_1) = (x/m_1)^2 - 2x \cdot m_2/m_1^3$, as required.

\item \cite{ledoit2002some}'s test of identity is based on $\tr{(\hSigma-I_p)^2} - n^{-1} (\tr{\hSigma})^2$. The first term is a LSS corresponding to $(x-1)^2$, while the second is a univariate function $\gamma_p y(\varphi)$, where $\gamma_p = p/n$, $y(r) = r^2$, and the LSS  is $\varphi(x) = x$. Now by a simpler univariate version of Lemma \ref{linearize}, and denoting asymptotic equivalence by $\equiv$, we have $y(\varphi)\equiv 2 m_1 T_p( g_1)$. Therefore, by Slutsky's theorem, $\gamma_p y(\varphi)\equiv  2 \gamma m_1T_p( g_1)$. Finally, by additivity, the whole test is equivalent to the LSS with $(x-1)^2 - 2 \gamma m_1 x$. 

\item The test from \cite{fisher2010new} based on $p \tr{\hSigma^4}/( \tr{\hSigma^2})^2$ can be handled similarly to John's test of sphericity. 
\end{itemize}
\end{proof}

\section{Implementation}
\label{Implementation}

\subsection{Computation}
\label{comput}

We now explain the computational details of our method. A MATLAB implementation, along with the code to reproduce our computational experiments, is available at \url{github.com/dobriban/eigenedge}. 

The computational problem is the following: Given a null distribution $\smash{H = d^{-1}}$ $\smash{\sum_{i=1}^{d} \delta_{t_i}}$, spikes $G_i = $ $h^{-1}\smash{\sum_{j=1}^{h} \delta_{s^i_j}}$, $i=0,1$, and an aspect ratio $\gamma$, compute the optimal LSS. We also include the sample size $n$ as an optional input, that is only needed to adjust the finite sample performance of the optimal LSS above the PT (see Sec. \ref{sec:above_pt}). For simplicity we take all spikes in $G_0$ subcritical---which is the only case we need in simulations---but the general case is similar. We will outline the needed steps and collect them in Algorithms \ref{optimal_LSS_main_alg}-\ref{optlss_alg}, giving the key parameter choices in Table \ref{params}.

\begin{table}[]
\centering
\caption{Parameter choices.}
\label{params}
\scalebox{0.85}{
\begin{tabular}{|l|l|l|}
\hline
Parameter & Meaning  & Choice  \\ \hline
$s_+$  & critical value for switching above PT & $0.75\cdot(1+\sqrt{\gamma})\cdot \tilde a_{PT}$ \\ \hline
$s_-$   & substitute spike  & $0.99\cdot \tilde a_{PT}$  \\ \hline
$\varepsilon$ & {\sc Spectrode} accuracy  &$5\cdot10^{-6}$ \\ \hline
$\ep_1$ & accuracy in collocation & $\max(10^{-8},c_0\cdot\ep)$ \\ \hline
$c_0$ & accuracy multiplier & $10^{-2}$ \\ \hline
$c_1$ & diagonal multiplier for $k$ & $1.5$ \\ \hline
$r$ & regularization of kernel & $10^{-4}\cdot\tr(K_{I0})/I$ \\ \hline
$n_{SD}$ & number of SDs in Epanechnikov & 3 \\ \hline
\end{tabular}
}
\end{table}

\begin{algorithm}
\caption{\textsc{Optimal LSS}: main wrapper for computation of the optimal Linear Spectral Statistic}
\begin{algorithmic}[1]
\Procedure{Optimal LSS}{}
\BState \textbf{input}
\State bulk $H = d^{-1}\sum_{i=1}^{d} \delta_{t_i}$ 
\State spikes $G_i = h^{-1}\sum_{j=1}^{h} \delta_{s^i_j}$, $i=0,1$
\State aspect ratio $\gamma$
\State sample size $n$ (optional)
\BState \textbf{begin}
\State Call \textsc{Spectrode} to compute:
\State $(x_m,\tilde v(x_m)), m=1\ldots,M  \gets$ approximate companion ST of $\fg(H)$
\State $\tilde S = \cup_j [\tilde l_j,  \tilde u_j]$, $j=1,\ldots,J \gets$ approximate support of $\fg(H)$
\State Define spike forward map $\psi(s) = s[1+\gamma d^{-1}\sum_{i=1}^dt_i/(s-t_i)]$
\State Define sample spikes $\psi_j=\psi(s_j^1)$
\If{there is $\psi_j  \notin \tilde S$}  
\State solve LSS above the phase transition by Alg. \ref{optimal_LSS_above_PT}
\Else 
\State solve LSS below the transition by Alg. \ref{optimal_LSS_below_PT}
\EndIf
\BState \textbf{return}  $(x_m,\tilde\varphi(x_m)), m=1\ldots,M  \gets$ approximate optimal LSS
\EndProcedure
\end{algorithmic}
\label{optimal_LSS_main_alg}
\end{algorithm}

\begin{algorithm}
\caption{\textsc{Optimal LSS}:  above the PT}
\begin{algorithmic}[1]
\Procedure{Optimal LSS above PT}{}
\BState \textbf{begin}
\State initialize LSS $\tilde\varphi(x_m)=0$
\If{the number of spikes is $h=1$}  
\If{the top spike is so small that $s_1^1< s^+$}  
\State $s_1^1\gets s^-$
\State solve LSS below the transition by Alg. \ref{optimal_LSS_below_PT}
\EndIf
\Else 
\State define Epanechnikov kernel $K_e(x) =\max(0,1-x^2)$
\State let $n=(d+h)/\gamma$ if it is not already defined
\For{each sample spike $\psi_j \notin \tilde S$}
\State set asy variance of spike $\hsigma_j^2 = 2[s^1_j]^2\psi'(s^1_j)$ 
\State define interval width $w = n_{SD} \cdot n^{-1/2}\hsigma_j$
\State set LSS for all $x_m$ in the neighborhood of the sample spike:
\State $\tilde\varphi(x_m) = K_e((x_m-\psi_j)/w)$
\EndFor
\State if there are extremal spikes, extend LSS to a constant away from $S$
\EndIf
\BState \textbf{return}  $(x_m,\tilde\varphi(x_m)), m=1\ldots,M  \gets$ approximate optimal LSS
\EndProcedure
\end{algorithmic}
\label{optimal_LSS_above_PT}
\end{algorithm}

\begin{algorithm}
\caption{\textsc{Optimal LSS}:  below the PT}
\begin{algorithmic}[1]
\Procedure{Optimal LSS below PT}{}
\BState \textbf{begin}
\State compute weak derivatives $\Delta_i = \dfg(H,G_0)$, $i=0,1$ using Alg. \ref{wd_alg}
\State solve optimal LSS equation $K(g) = - \Delta$ using Alg. \ref{optlss_alg}. ($\Delta:=\Delta_1-\Delta_0$)
\State numerically integrate $g$ to get LSS $\tilde\varphi(x_m)$
\BState \textbf{return}  $(x_m,\tilde\varphi(x_m)), m=1\ldots,M  \gets$ approximate optimal LSS
\EndProcedure
\end{algorithmic}
\label{optimal_LSS_below_PT}
\end{algorithm}

\begin{algorithm}
\caption{\textsc{Weak Derivative}:  computation of weak derivative of MP map}
\begin{algorithmic}[1]
\Procedure{Weak Derivative}{}
\BState \textbf{input}
\State bulk $H = d^{-1}\sum_{i=1}^{d} \delta_{t_i}$ 
\State spikes $G = h^{-1}\sum_{j=1}^{h} \delta_{s_j}$ below PT
\State aspect ratio $\gamma$
\State $(x_m,\tilde v_m), m=1\ldots,M  \gets$ approx companion ST of $\fg(H)$, $\tilde v_m = \tilde v(x_m)$ (optional)
\BState \textbf{begin}
\If{$(x_m,\tilde v_m)$ is not provided}  
\State call \textsc{Spectrode} to compute  $(x_m,\tilde v_m), m=1\ldots,M$
\EndIf
\State define the derivative map 
$$d(v) = \left[\frac{1}{v^2} - \frac{\gamma}{d}\sum_{i=1}^{d} \frac{t_i^2}{(1+t_iv)^2}\right]^{-1}$$
\State compute approximate ST of weak derivative using \eqref{weak_st}: 
$$\tilde s(x_m) = -\gamma \, d(\tilde v_m ) \left[\frac{1}{h}\sum_{j=1}^{h} \frac{s_j}{1+s_j\tilde v_m} -  \frac{1}{d}\sum_{i=1}^{d} \frac{t_i}{1+t_i\tilde v_m}\right]$$
\State compute approximate density $\tilde f(x_m) = \pi^{-1} \Im(\tilde s(x_m))$, and integrate numerically to approximate $\Delta_G$
\BState  \textbf{return}  pointwise CDF $(x_m,\Delta_G(x_m))$ of $\dfg(H,G)$
\EndProcedure
\end{algorithmic}
\label{wd_alg}
\end{algorithm}

\begin{algorithm}
\caption{\textsc{Optimal LSS}:  solve integral equation}
\begin{algorithmic}[1]
\Procedure{Optimal LSS Equation}{}
\BState \textbf{begin}
\State restrict to grid $x_t$ within the support $\tilde S$
\If{using collocation method} 
\State use {\sc Spectrode} with accuracy $\ep_1$ to find $\tilde v(y_a)$ on a dense grid $\{y_a\}$
\State let $\{g_l\}$ be the Lagrange basis for linear interpolation on the grid $x_t$
\State compute $\tilde K_{it} = (Kg_i)(x_t) = \int k(x,x_t) g_i(x) dx$ for all $i,t$
\State - do this by discretizing the integral to the dense grid $y_a$
\State - replace any $\infty$ terms in $k$ by $c_1\cdot \max_b k(y_b,x_t)$ for $y_b\neq x_t$
\State solve linear equations $\sum_{i=1}^{I} \gamma_i \tilde K_{it}= -\Delta(x_t)$, $t=1,\ldots,I$ for $\gamma_i$. 
\State set $\tilde g(x_m) = \sum_{i=1}^{I}\gamma_i g_i(x_m)$
\EndIf
\If{using diagonal regularization method}
\State define a matrix $K_{I0}$ as follows:
\State -let $k_{ij} = k(x_i,x_j)$ for $i\neq j$ 
\State -let $k_{ii} = c_1\cdot k(x_i,x_{|i-1|})$, if $i>1$; while $K_{I0}(x_1,x_1) = c_1\cdot k(x_1,x_2)$.
\State let $K_I = K_{I0}+rI_{I}$, where $r$ is a parameter in Table \ref{params}.
 
\State solve the linear system $K_I \tilde g = -\Delta_I$; $\Delta_I =\{\Delta(x_m)\}_m$
\EndIf
\State numerically integrate $\tilde g(x_m)$ to get LSS $\tilde \varphi(x_m)$
\State interpolate LSS linearly between bulk intervals; extrapolate it as a constant outside 
\BState \textbf{return}  $(x_m,\tilde\varphi(x_m)), m=1\ldots,M  \gets$ approximate optimal LSS
\EndProcedure
\end{algorithmic}
\label{optlss_alg}
\end{algorithm}

\subsubsection{Computing $v$ and the support}
First we compute the companion Stieltjes transform $v(x)$ of the ESD $\fg(H)$ on a dense grid $\{x_m\}$ on the real line (see Alg. \ref{optimal_LSS_main_alg}). We use the \textsc{Spectrode} method \citep{dobriban2015efficient}, which produces an approximation $\tilde v(x)$ that depends on a user-specified accuracy parameter $\ep>0$, and converges to $v(x)$ as $\ep\to0$.  In \cite{dobriban2015efficient}, we showed that $\pi^{-1}\Im(\tilde v(x))$ converges to the density $\pi^{-1}\Im(v(x))$ of the ESD. An analogous argument shows that $\tilde v(x)$ converges to $v(x)$. \textsc{Spectrode} also produces a converging approximation to the support $S$ of $\fg(H)$ as a union of closed intervals $\tilde S = \cup_j [\tilde l_j,  \tilde u_j]$, $j=1,\ldots,J$, sorted in increasing order.

There are two cases---below and above the PT---which depend on whether or not $\Delta\in\textnormal{Im}(K)$. As a proxy to this abstract statement, we check if the sample spikes corresponding to $s_j^1$ belong to the support, as in Section \ref{sec:full_pow}. We have shown that $\Delta\notin\textnormal{Im}(K)$ if some sample spikes are outside the ESD. This is the first case that we handle (Alg. \ref{optimal_LSS_above_PT}). Second, if all sample spikes are in the support, we directly attempt to solve a discretized version of the optimal LSS equation (Alg. \ref{optimal_LSS_below_PT}). We were able to solve it with good accuracy in all examples we have seen; but there may be cases where it does not have a solution. This unlikely case is discussed after the two main cases. 

\subsubsection{Above the PT}
\label{sec:above_pt}
From $\tilde v(x)$ and the support, we check if there are any spikes above the phase transition by verifying if any \emph{sample spike} $\psi(s^1_j)$ falls outside the support: $\smash{\psi(s^1_j) \notin \tilde S}$ for any $j$.  Recall here that $\psi$ is the spike forward map from Section \ref{sec:full_pow}, and equals $\smash{\psi(s) = s[1+\gamma d^{-1}\sum_{i=1}^dt_i/(s-t_i)]}$, see \cite[Ch. 11, Eq. 11.15]{yao2015large}. If there are spikes above the PT, then we follow the steps in Alg. \ref{optimal_LSS_above_PT}. By Thm \ref{thm:lss_part2}, the asymptotic power is unity, and any smooth function $\varphi$ that equals unity in a small neighborhood of the sample spike, and zero on $\tilde S$, is an approximate optimal LSS. 

For good finite sample performance, for the LSS we take a small Epanechnikov kernel centered at each sample spike $\psi(s_j^1)$, and zero elsewhere. Since the fluctuations of the spikes are asymptotically normal above the phase transition, we choose the width of the kernel as $n_{SD}\cdot n^{-1/2}\hsigma_j$. Here $n_{SD}$ is a constant given in Table \ref{params}, $n$ is the sample size, and $\hsigma_j$ is the asymptotic standard deviation of the sample spike, $\hsigma_j^2 = 2[s^1_j]^2\psi'(s^1_j)$; see \cite{yao2015large} Thm 11.11, and also \cite{onatski2012asymptotics} Thm 2 for closely related earlier results. Moreover, we extend the LSS as a constant equal to unity in the direction pointing away from the support $S$, for any \emph{extremal} spikes that fall above $\max(S)$, or below $\min(S)$. If the optional input $n$ is not provided, we set $n =  (d+h)/\gamma$, which is equivalent to assuming that $p=d+h$. 

We noticed a drop in the power of this method right above the PT threshold. This finite sample effect is due to the overestimation of the variability of the top eigenvalue. We use a formula predicting an order $n^{-1/2}$ fluctuation, however, below the PT the fluctuation is of order $n^{-2/3}$ (e.g., \cite{baik2005phase} show a special case of this), which is an order of magnitude smaller. We are not aware of better approximations to the variability in the spike near the PT. 

To overcome this challenge, we heuristically use the optimal LSS corresponding to a spike $s_{-}$ right below the PT threshold $a_{PT}$, even when testing for a spike $s_j^1$ slightly above the PT, in a certain interval $[\tilde a_{PT},s_{+}]$ (see Table \ref{params}). This is how we performed our MC experiments, and the results were satisfactory. In principle, all edges could be adjusted similarly.

We approximate the uppermost PT threshold by $\tilde a_{PT} = -1/\tilde v(\tilde u_J)$. The true PT threshold is at $a_{PT} = -1/v(u_J)$. Theorem 2.7 of \cite{benaych2011eigenvalues} and its proof is an equivalent statement; \cite{nadler2008finite} in Sections 5.2-5.3 also basically establishes the same; and finally Theorem 11.3 of \cite{yao2015large} (quoting \cite{bai2012sample}) also shows the same. 

\subsubsection{Below the PT}

If there are no spikes above the PT, we proceed to solve the optimal LSS equation (see Alg. \ref{optimal_LSS_below_PT}). The LSS is well-defined only within the support $S$ of the bulk $\fg(H)$, so we restrict to that subset of the grid. First, the kernel $k$ is evaluated pointwise using $\tilde v$. 

Next, we compute the difference $\Delta$ of the weak derivatives (Alg. \ref{wd_alg}). As explained in \cite{dobriban2015efficient}, $v'(z)$ can be expressed in closed form as a function of $v(z)$. Hence, using Eq.~\eqref{weak_st} we can approximate the Stieltjes transforms of $\dfg(H,G_i)$.  We find their density from the inversion formula for Stieltjes transforms, and their distribution by integrating the density numerically. 

Finally, we need to solve the optimal LSS equation $Kg = -\eta \Delta$ (where we set the constant $\eta$ to 1 without loss of generality), see Alg. \ref{optlss_alg}. This is a Fredholm integral equation of the first kind with a logarithmically weakly singular kernel, and there are many methods for solving such equations numerically \citep[see][]{kress2013linear}.  We implement two methods: A fast heuristic diagonal regularization method, and a slower but potentially more accurate collocation method. 

First, our \emph{collocation} method---see \citep[Ch. 13]{kress2013linear} for a reference---expresses the potential solution as a linear combination $\smash{g = \sum_{i=1}^{I} \gamma_i g_i}$, for given basis functions $g_i$, and unknowns $\gamma_i$, and solves the system of linear equations resulting from the pointwise evaluation of the integral equation on a grid $x_t$, i.e., $\smash{\sum_{i=1}^{I} \gamma_i (Kg_i)(x_t) = -\Delta(x_t)}$, $t=1,\ldots,I$.  The grid $x_t$ is taken from the output of \textsc{Spectrode}. We choose $g_i$ as the Lagrange basis for linear interpolation on the grid $x_t$ \citep[Ch. 11]{kress2013linear}, because this reduces the length of intervals where we need to integrate.  

To evaluate these integrals, we use {\sc Spectrode} again---now with a higher accuracy---to approximate $v$ on a denser grid $y_a$. We then discretize each integral $(Kg_i)(x_t)= \int k(x,x_t) g_i(x) dx$ into the grid $y_a$, using the explicit formulas for $k$ in terms of $v$, and the explicit form of the Lagrange basis. If there are any elements of $y_a$ that coincide with $x_t$, then the kernel has a singularity, $k(y_a,x_t)=\infty$. We resolve this by replacing $k(y_a,x_t)$ by $c_1\cdot\max_b k(y_{b},x_t)$, for $y_b\neq x_t$, where $c_1$ is a parameter specified in Table \ref{params}.


This algorithm is empirically stable, and leads to accurate solutions in a few minutes on a desktop computer---see the experiments in the next section. There are theoretical convergence proofs for closely related versions of the collocation method  \citep[Ch. 13]{kress2013linear}. However, verifying their conditions requires work that would take us too far from our current goals. 

Second, in the \emph{diagonal regularization} method we discretize the optimal LSS equation by pointwise evaluation on $x_t$, replacing the singularities $k(x_t,x_t)=+\infty$ heuristically (see Table \ref{params}). First we compute an initial matrix $K_{I0}$, with $K_{I0}(x_i,x_j) = k(x_i,x_j)$ if $i\neq j$, and $K_{I0}(x_i,x_i) = c_1\cdot k(x_i,x_{|i-1|})$, if $i>1$; while $K_{I0}(x_1,x_1) = c_1\cdot k(x_1,x_2)$. Then we regularize $K_I = K_{I0}+rI_{I}$, where $r$ is a function of the trace of $K_{I0}$, as explained in Table \ref{params}. Finally we solve the pointwise equation $K_Ig=-\Delta_I$. This method is faster, while maintaining good empirical accuracy compared to the collocation method (see the next section). 
However, there are fewer numerical convergence guarantees  for such discretization methods.

Our theory only specifies the LSS within $S$, and there is some latitude in extending it outside. We interpolate linearly between the bulk components, and extend it by continuity to a constant in the two outermost regions. Smoother extrapolations may be possible, especially as the LSS can have sharp asymptotes at the edges. We leave such improvements to future research.

\subsubsection{Potential pathological cases}

Our results from Section \ref{olss}---such as Thm \ref{full_pow}---do not exclude that in some pathological cases the spikes $s_j^1$ are below the PT, but $\Delta\notin\textnormal{Im}(K)$. The optimal LSS equation would not solvable in such a case. 

However, we find this possibility unlikely; and we have not seen evidence for it. It would mean that the asymptotic power of the optimal LSS is unity even though the spikes do not separate from the bulk.  Instead, it is more likely that the one-to-one correspondence between the two characterizations of PT---in terms of $K$ and spikes $s_j^1$---will be proved in the future. Therefore we do not devise a special method for this case. 

\subsubsection{Unit tests}

\begin{figure}
\centering
\begin{minipage}{.5\textwidth}
  \centering
  \includegraphics[scale=0.32]
  {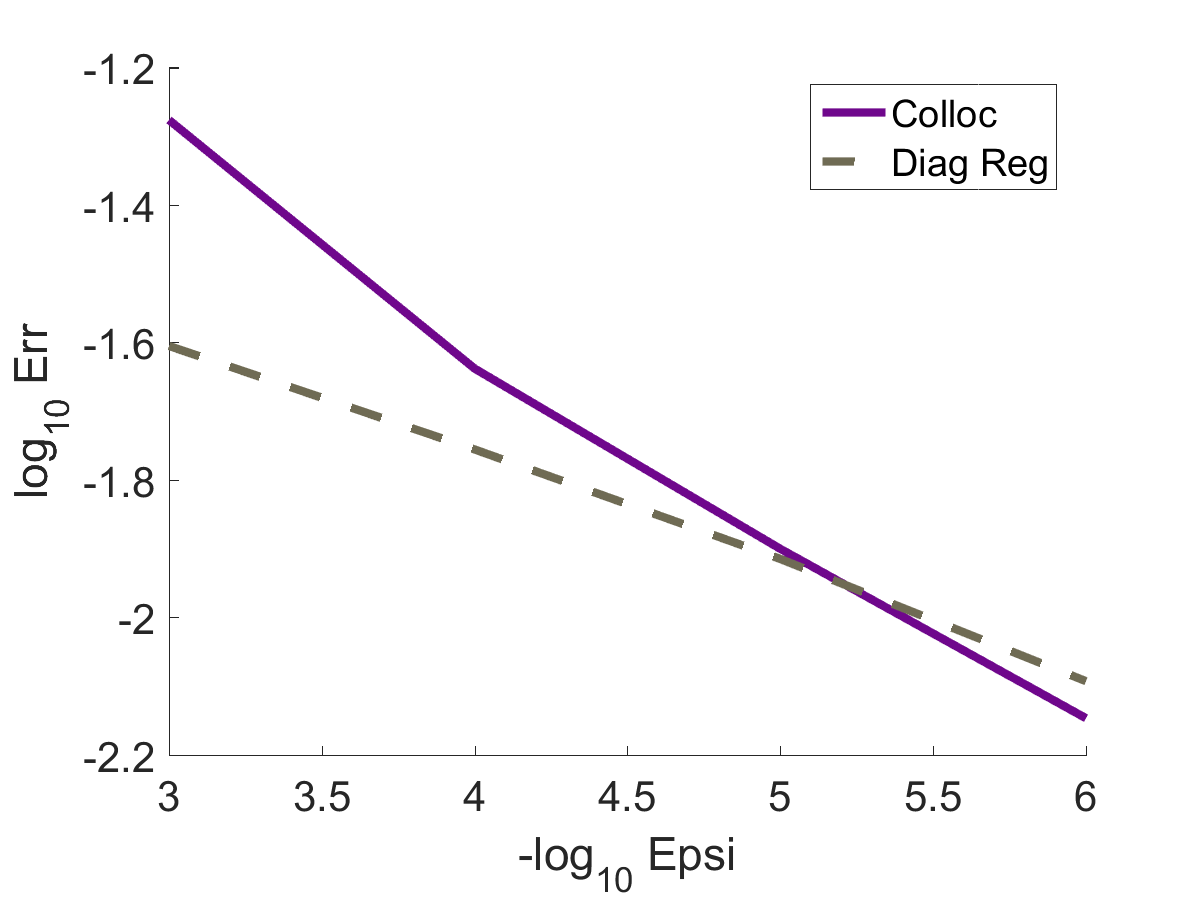}
\end{minipage}%
\begin{minipage}{.5\textwidth}
  \centering
  \includegraphics[scale=0.32]{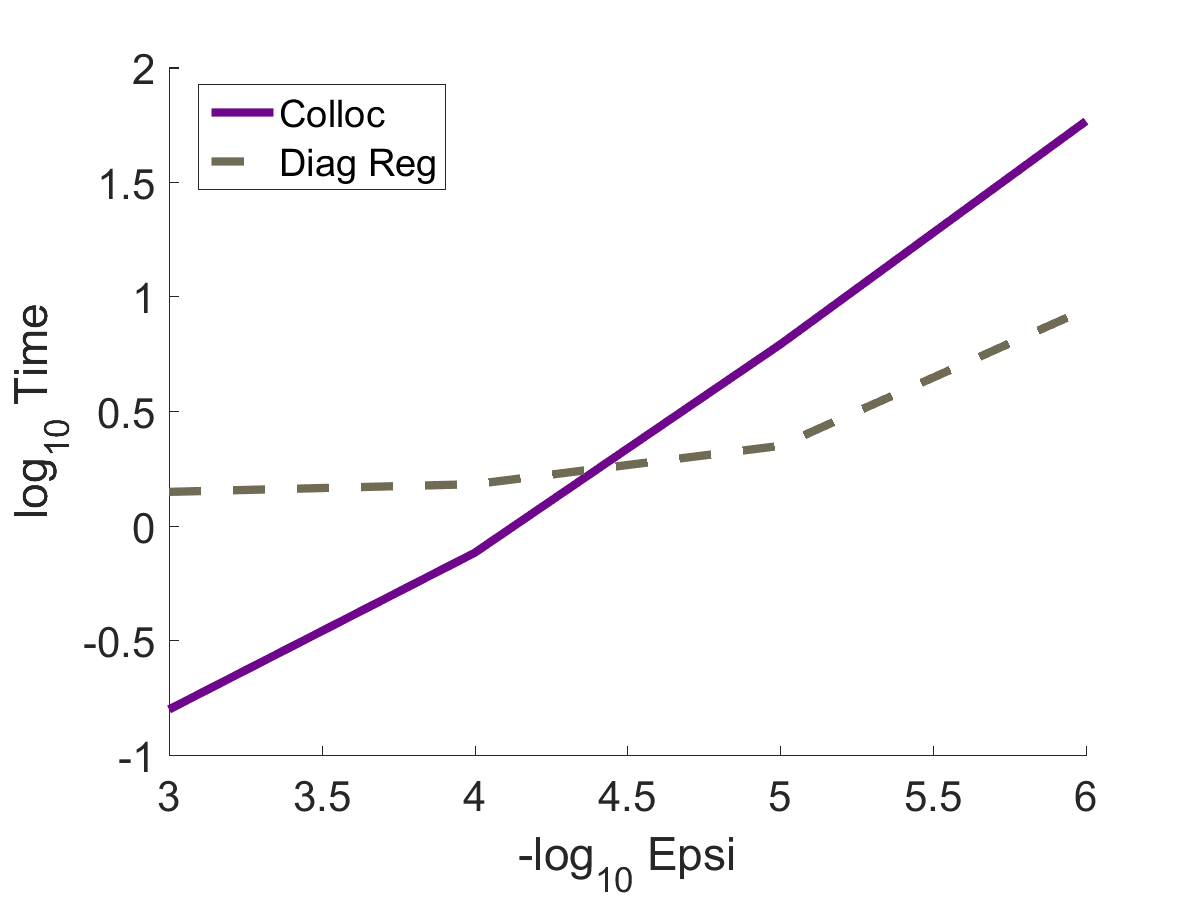}
\end{minipage}
\caption{Performance of collocation and diagonal regularization method, as a function of precision control parameter $\varepsilon$. Left: Numerical precision ($\log_{10}$ Mean Absolute Deviation from OMH LSS).  Right: Running time ($\log_{10}$ seconds). }
\label{test_err_time}
\end{figure}

\begin{figure}
\centering
\begin{minipage}{.5\textwidth}
  \centering
  \includegraphics[scale=0.32]
  {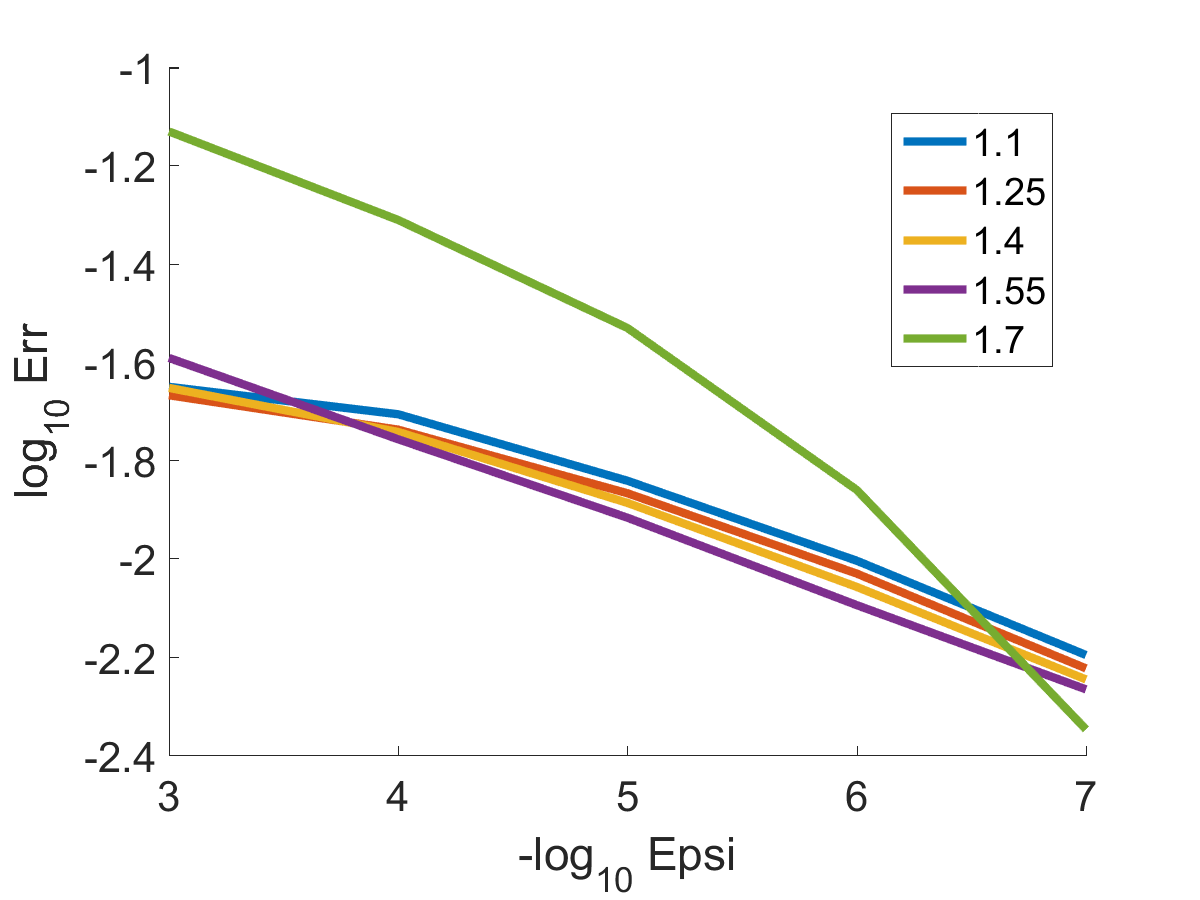}
\end{minipage}%
\begin{minipage}{.5\textwidth}
  \centering
  \includegraphics[scale=0.32]{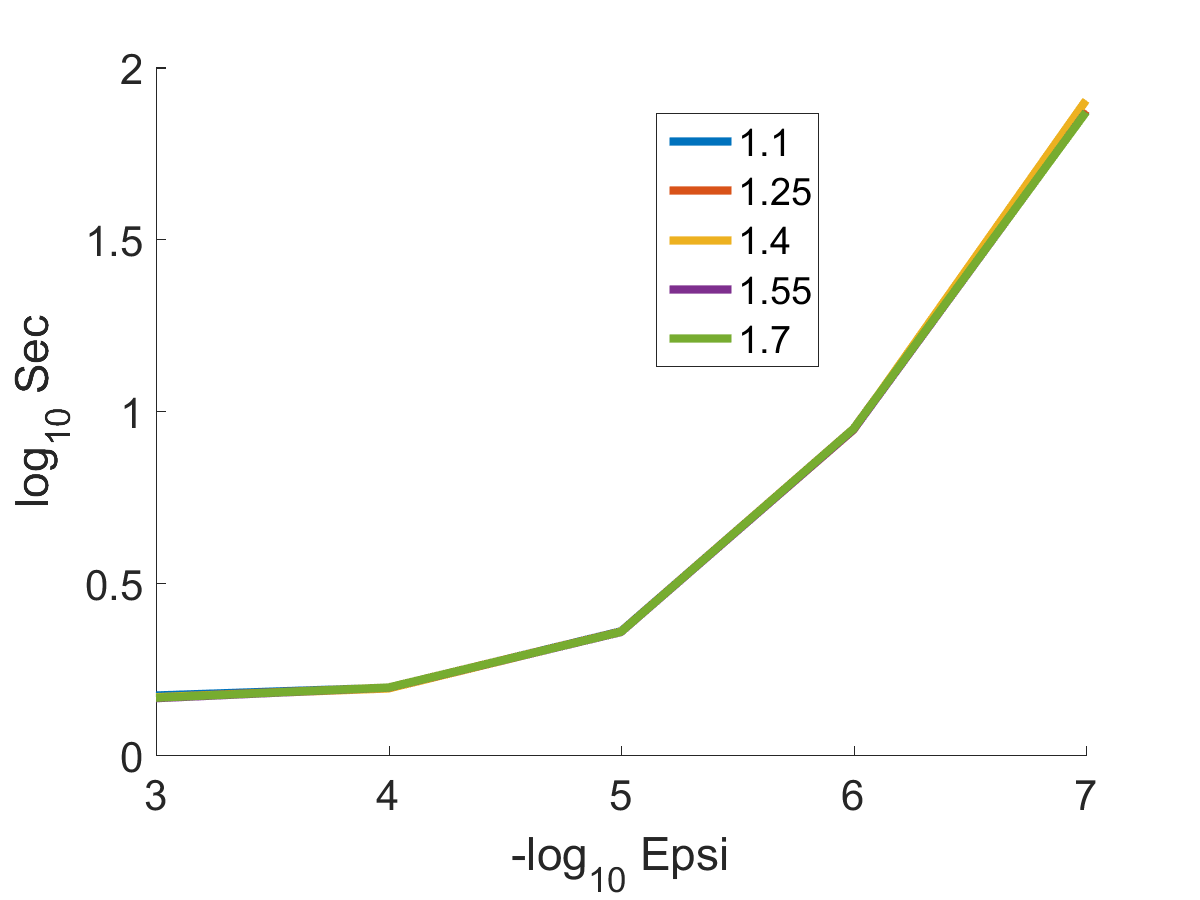}
\end{minipage}
\caption{Performance of diagonal regularization method, as a function of precision control parameter $\varepsilon$, for different spikes. Left: Numerical precision ($\log_{10}$ Mean Absolute Deviation from OMH LSS).  Right: Running time ($\log_{10}$ seconds). }
\label{test_dreg_spikes}
\end{figure}

To show the performance of our methods, we report the results of unit tests and timing experiments. In each test we compute the optimal LSS using the methods described in the previous sections. We use the test problem where $H=G_0=\delta_1$, $G_t=\delta_t$ for some $t$, and $\gamma=1/2$. We compare our results against the gold standard OMH LSS.

First, on Fig. \ref{test_err_time}, we compare the performance of the collocation and diagonal regularization method as a function of the precision control parameter $\varepsilon$. We take the spike $t=1.2$, varying $\ep$ on a grid such that $-\log_{10}\ep = 3,\ldots,6$, and record the precision and running time of the methods. For precision, we use the Mean Absolute Deviation (MAD) from the OMH LSS on the grid returned by the methods. 
Specifically if $(x_m,\tilde \varphi(x_m))$, $m=1,\ldots,M$ are the grid points and LSS values returned by a method for precision $\ep$, then $err(\ep)  = -\log_{10}(MAD(\ep))$, where $MAD(\ep) = mean(|\varphi(x_m)-f(x_m,t)|)$ taken over grid points $x_m$ within the support of the MP law, and $f(x_m,t)$ is the OMH LSS. For the time, we simply record the seconds to completion, tested in MATLAB 2015b on a desktop computer with 8Gb RAM and and 64-bit Intel 3.2Ghz processor. 

On Fig.~\ref{test_err_time} we see that the performance of both methods improves as $\ep$ decreases. The two methods have comparable accuracy. For $\ep = 10^{-6}$, we get approximately 2 significant digits. In \cite{dobriban2015efficient}, we observed that the output of {\sc Spectrode} has approximately as many significant digits of accuracy as its control parameter $\ep$. Therefore, the methods here have significantly lower accuracy. This is expected, however, because there are many processing steps which potentially decrease accuracy. Moreover, the linear integral equation is an ill-posed problem and is expected to decrease accuracy. Therefore the methods have satisfactory performance, but there may be room for improvement.

We also see on Fig. \ref{test_err_time} that for the highest accuracy, diagonal regularization is faster by an order of magnitude than collocation, while achieving comparable accuracy. For $\ep=10^{-6}$, computation takes cca 10 sec. For this reason, we use diagonal regularization as the default method. 

To gain a better understanding of the performance of diagonal regularization, we repeat this experiment varying the spike $t$. The accuracy is comparable across all values of the spike away from the PT, but it is lower near the PT $1+\sqrt{\gamma} =1.714\ldots$ (Fig. \ref{test_dreg_spikes}, left). Meanwhile, the running times are nearly the same (Fig. \ref{test_dreg_spikes}, right). 

We conclude that the two methods are fast and accurate, but diagonal regularization is somewhat more efficient for high accuracy computations. Notably, it has lower accuracy near the PT.

\subsection{Extension to unknown null}
\label{unknown_null}

In many cases, the question of scientific interest may be to test for principal components in the data without knowing the null distributions of the PC variances. This is more difficult than our problem, because the null must be estimated from the data. While a complete treatment is beyond our scope, we outline a possible approach below.

We suggest sample-splitting: one can estimate the noise structure from the first random subset of the data, leading to an estimate $\hat{H}_p$ for the spectrum. There are consistent methods for estimating the spectrum, e.g., \cite{karoui2008spectrum}, see \cite{yao2015large}, Ch.~10 for a review. Our methods can then be used to test for PCs in the data, by using $\hat{H}_p$ as a null. Further work is needed to evaluate or improve this informal proposal.

\section{Empirical motivation}
\label{empirical_evidence}

We review empirical evidence suggesting the need for methods that can detect weak PCs in the presence of complex residual noise. This empirical evidence is a main motivation for our methods.  Due to space limitations, we keep the references to a minimum.

\subsection{Genomics}

In genomics, PCA is commonly used to infer population structure from data on densely typed genetic markers. This has a wide range of applications, including correction for confounding in genome-wide association studies (see e.g., \cite{patterson2006population}, which guides our presentation). A standard setup is that $X$ is an $n \times p$ matrix with $X_{ij}$ equal to the number of minor alleles (0,1 or 2) of the $j$-th genetic variant---often a single nucleotide polymorphism (SNP)---in the genome of the $i$-th individual.

It is often a question of interest to detect the presence of multiple distinct subpopulations. Under the null hypothesis of no population substructure, the $n \times n$ population covariance matrix of \emph{indivduals} equals identity; while under the alternative of a small number of distinct populations, it equals a low-rank perturbation of a near-identity matrix, under certain assumptions \citep{patterson2006population}. A potential model for such situations is $X = (I_n + AA^\top)^{1/2}Z$, where $Z$ is a matrix whose columns have identity covariance, and $A$ is an $n \times k$ low-rank matrix. In this case, the population covariance of individuals is $\Gamma = I_n + AA^\top$, which equals $I_n$ if there is no structure. Based on the proposal of \citep{patterson2006population}, it is common to use the empirical eigenvalues of the sample counterpart $\hGamma = p^{-1} X X^\top$ to test for the existence of population structure using the standard Tracy-Widom test.

However, it is well known that genetic variants close to each other on the chromosome are correlated in the population  due to linkage disequilibrium (LD); this is acknowledged in \cite{patterson2006population}.  Therefore the population covariance matrix of \emph{variants} (SNPs) is non-identity even without population structure. Continuing with our model, one may write $\smash{X = (I_n + AA^\top)^{1/2}}$ $\smash{Z\Sigma_p^{1/2}}$, where $\smash{\Sigma_p}$ is the covariance matrix of the SNPs. The correlations due to $\Sigma_p$ may show up in the spectrum of $\smash{\hGamma}$, leading to a weaker approximation by the standard Marchenko-Pastur null.

This is not the only possible source of non-identity covariance. For instance, departures from the standard Marchenko-Pastur distribution have been observed empirically by \cite{bryc2013separation}, by analyzing data from the International HapMap Project.  After removing what appeared to be significant axes of variation, they observed an empirical bulk that had a long right tail, and was possibly multimodal; unlike the standard Marchenko-Pastur distribution (see their Fig. 3). They attributed these departures to complex substructure and relationships among individuals. 

Similarly, \cite{Kumar2015limitations} showed that the eigenvalues of the Framingham Heart Study dataset (49,214 SNPs in 2,698 unrelated individuals) are highly skewed, with many small eigenvalues (see their Fig. 3). The condition number of the data matrix is cca $10^{10}$. They attribute this to genetic stratification in the sample, and show its importance for estimating heritability.

As an approach for dealing with problem,  \cite{patterson2006population} proposed to correct for LD by either ``LD pruning'', i.e., removing SNPs from pairs above a correlation threshold; or by local regression of SNPs on their neighbors. However, these steps may induce additional variability and arbitrariness in the data analysis. For instance, local regression may remove correlations among the SNPs, but it may also reduce the correlations with the outcome, leading to an undesired loss of power. While such steps may sometimes work well, we are not aware of any \emph{general} correctness guarantees. 

Correlations among the SNPs, as well as complex population substructure, lead to departures from the standard Marchenko-Pastur null.  This motivates us to develop methods that detect PCs beyond the null of identity. The ability to test hypotheses that allow for correlations could lead to better methods for inference of population structure in the presence of LD or stratification. While this clearly requires more methodology development, we think that our work is a necessary step in that direction.

\subsection{Finance}

In finance, the sample covariance matrix is of interest in several problems, such as Markowitz portfolio optimization and factor analysis (see e.g., \cite{bouchaud2009financial}, for a recent review). For many financial data sets, it has been observed that the bulk of the eigenvalue distribution of the sample covariance differs from the standard Marhenko-Pastur distribution. 

For instance \cite{bouchaud2009financial} analyzed U.S. stock market data from the top 500 most liquid stocks in 1000 day periods from 1993 to 2008. They observed that the empirical eigenvalue distribution of the correlation matrix of the stocks has a long right-hand tail, and found that a power-law model for the spectrum gives a good fit (see Sec. 5A and Fig. 2 in \cite{bouchaud2009financial}). Specifically, they found that the Marchenko-Pastur map (Section \ref{olss}) of a power law density for the population eigenvalues $\rho(\lambda) \propto (\lambda-\lambda_0)^{-1-\mu}I(\lambda>\lambda_1)$, with $\mu,\lambda_i>0$, leads to a good empirical fit. They interpreted this as a model for the coexistence of larger and smaller sectors of activity. Similarly, \cite{zumbach2011empirical} analyzed three financial and economic data sets, and found that empirical spectral densities of the form $\rho(\lambda) \propto \lambda^{-1}$ were a good fit for covariance matrices (see his Sec 7. and Fig. 6).

This implies that we need signal detection methods that can account for complex noise structure in the bulk of the spectrum. Methods that assume white noise may be inefficient when the noise structure is non-white, and may lead to incorrect inferences. Our work is a step toward developing such a methodology. 

\subsection{Data Example}

\begin{figure}
\centering
\begin{minipage}{.5\textwidth}
  \centering
  \includegraphics[scale=0.32]
  {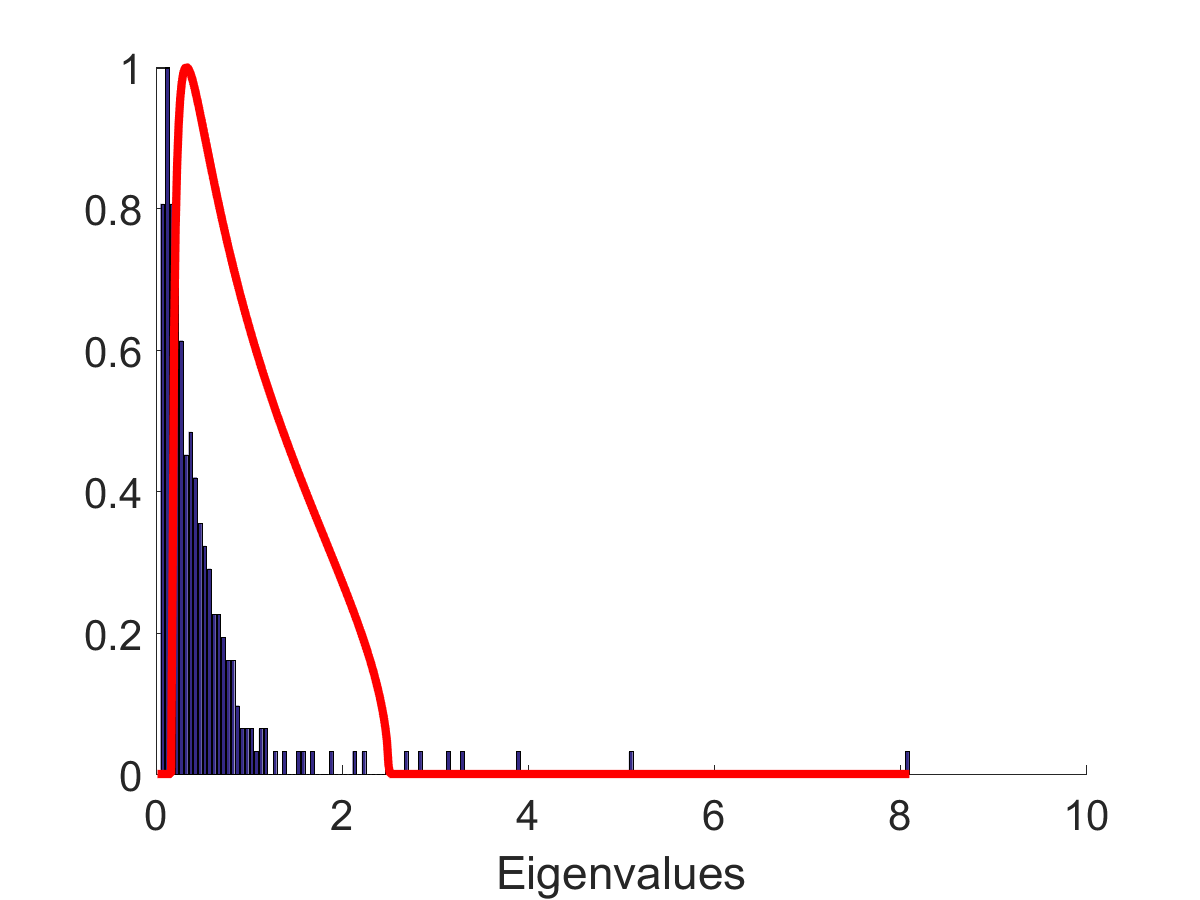}
\end{minipage}%
\begin{minipage}{.5\textwidth}
  \centering
  \includegraphics[scale=0.32]{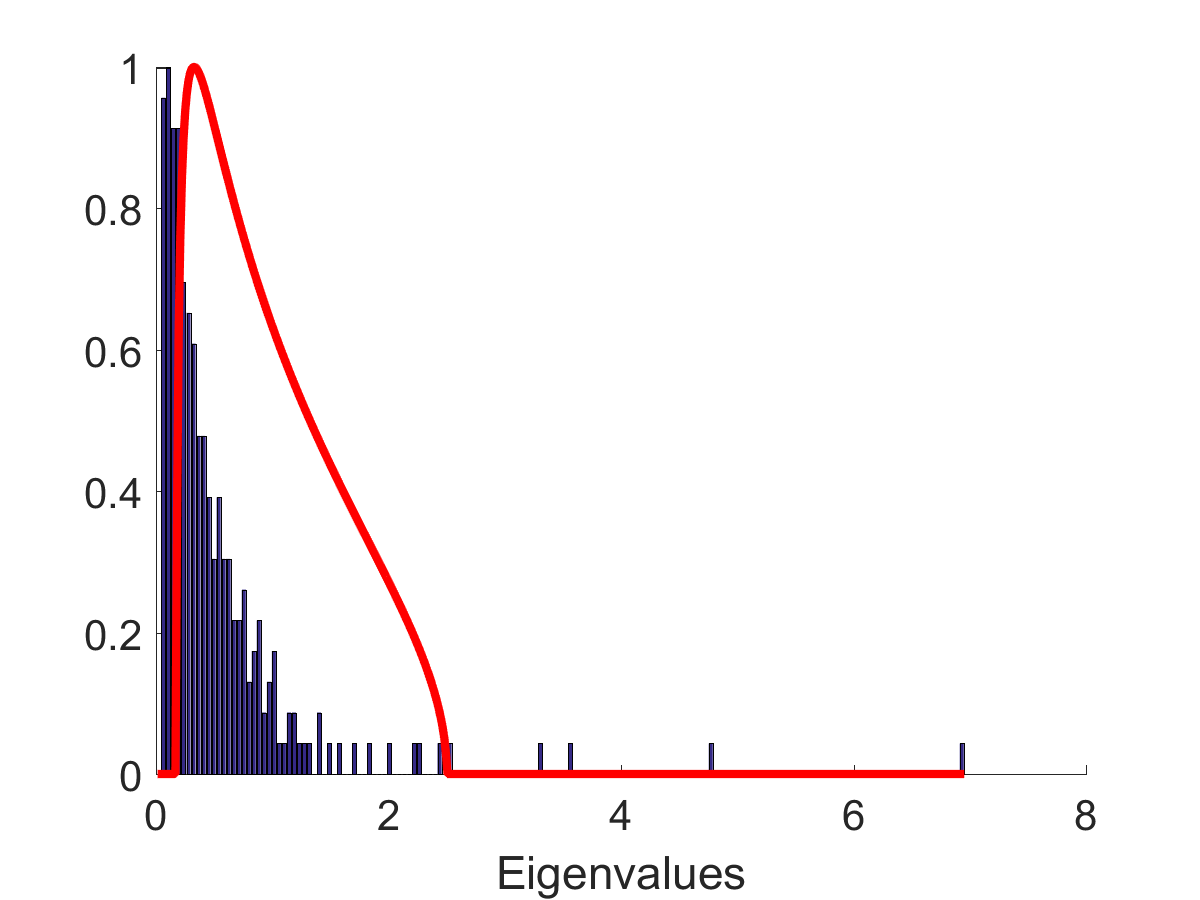}
\end{minipage}
\caption{Histogram of eigenvalues smaller than 10 of covariance matrix (left) and correlation matrix (right) of phoneme data. Superimposed is the Marchenko-Pastur density.}
\label{data}
\end{figure}

As a further motivating example, we show the eigenvalues of a dataset on phonemes, which was previously analyzed by \cite{johnstone2001distribution}, and originally presented by \cite{hastie1995penalized}. We choose this somewhat old dataset because it is a standard example in the field, and in addition to \cite{johnstone2001distribution} it is also used as an illustration in \cite{baik2005phase,yao2015large}. 

The dataset consists of log-periodograms of length $p=256$ of $n=757$ instances of the spoken phoneme ``dcl'' (as in ``dark''). A subset of $n=162$ observations were presented in \cite{hastie1995penalized}, however the full data set available at \href{statweb.stanford.edu/~tibs/ElemStatLearn/datasets/phoneme.data}{statweb.stanford.edu/~tibs/ElemStatLearn/datasets} is larger.  \cite{johnstone2001distribution} analyzed the smaller dataset and observed that the standard Marchenko-Pastur null density $f_\gamma(x) = \sigma\sqrt{(\gamma_+-x/\sigma)(x/\sigma-\gamma_-)}/(2\pi\gamma x)$ for $x\in[\gamma_-,\gamma_+]$, $\gamma_\pm = (1\pm\sqrt{\gamma})^2$, $\gamma = p/n$, provides a good fit to the bulk of the sample covariance matrix. In that analysis the largest 12 eigenvalues are significant according to a Tracy-Widom test. The noise level $\sigma$ is estimated as the mean of the eigenvalues. 

We show the histogram of eigenvalues of the sample covariance matrix on the left plot of Figure \ref{data}. We normalize the eigenvalues to have unit mean, and we also plot the Marchenko-Pastur density with $\sigma=1$ and $\gamma=p/n$. For display purposes, we omit 2 eigenvalues larger than 10. We see that the Marchenko-Pastur density is not a good fit to the bulk. Rescaling the sample covariance matrix does not seem give a better fit. At least 20\% of the eigenvalues are usually outside the bulk (See Section \ref{sec:scale}). The eigenvalues of the correlation matrix (right plot) do not seem to fit the Marchenko-Pastur law either. 

Moreover, in this example the Marchenko-Pastur bulk is a good model for a small subset of the data, but it is not so good for the whole dataset. These examples reinforce the need to have models going beyond the Marchenko-Pastur bulk, and provide further motivation for our theory. 

\subsection{Scaling the covariance matrix}
\label{sec:scale}
On Fig. \ref{fig:scale} we show histograms of the eigenvalues smaller than 10 of the covariance matrix of the phoneme data, scaled by various $\sigma$. We first normalize the eigenvalues to have unit mean, and then multiply them by $\sigma$ on a uniform grid on $[0.5,5]$. These are displayed moving from the top left image to the right, and continuing in the lower rows. Superimposed is the Marchenko-Pastur (MP) density. 

We observe that for most scaling parameters the MP density does not fit well. The best fit seems to be for the figure on the right in the second row, for which $\sigma = 3.2$. However, in this case there are 36 eigenvalues outside of the support of the MP density, even after enlarging the support conservatively, to take into account the fluctuations of order $n^{-1/3}$ of the largest eigenvalue \citep[the results of][imply that this is the right order of fluctuation in the Gaussian case under the null]{johnstone2001distribution}. This number seems too large to be practical, because there are only $p=256$ dimensions. Furthermore, there is no clear gap between ``signal'' and ``noise'' eigenvalues for this $\sigma$, and thus it would be hard to justify its choice. It is reasonable to think of more general models for the bulk, motivating the approach of this paper.

\begin{figure}[p]
\centering
\begin{tabular}{ccc}

\includegraphics[width=\FW, trim = \TRA mm \TRB mm \TRC mm \TRD mm, clip = TRUE]{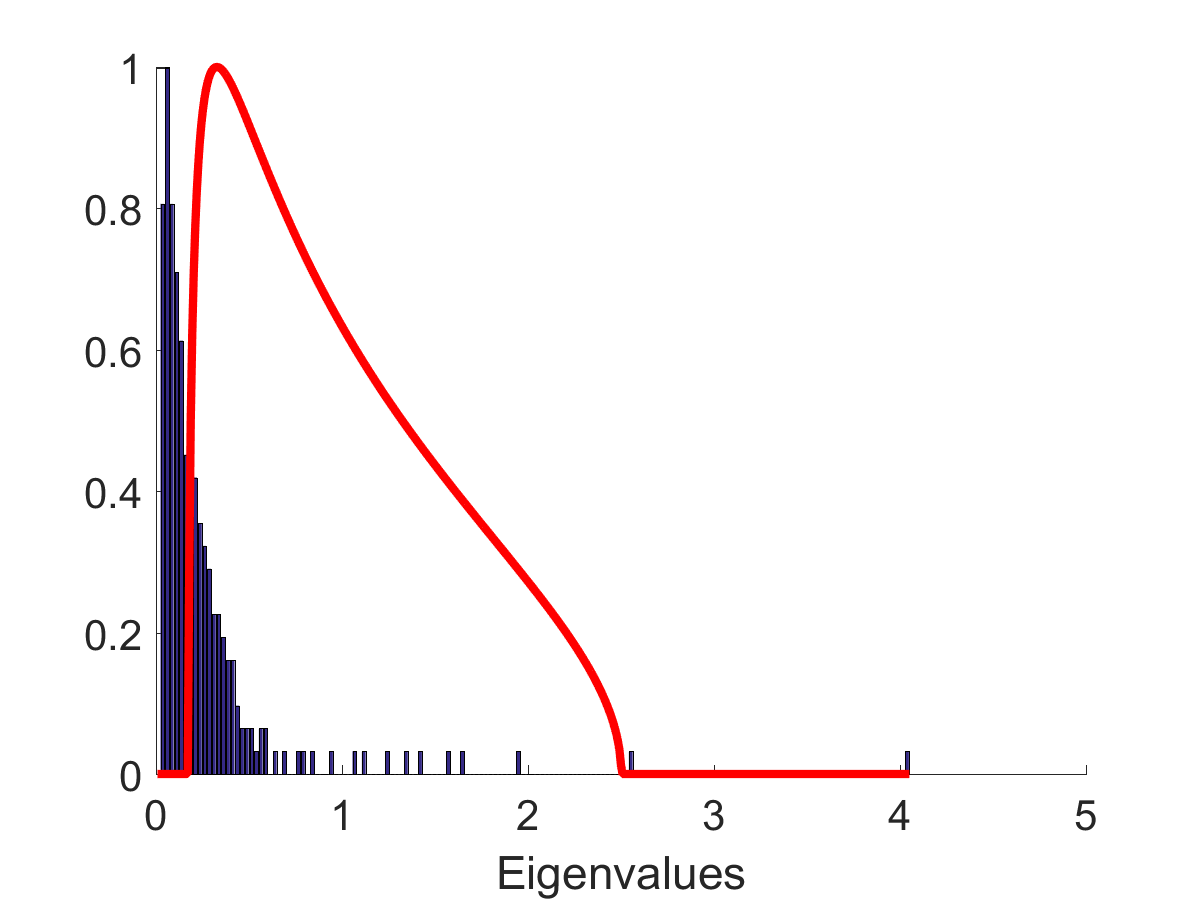} &
\includegraphics[width=\FW, trim = \TRA mm \TRB mm \TRC mm \TRD mm, clip = TRUE]{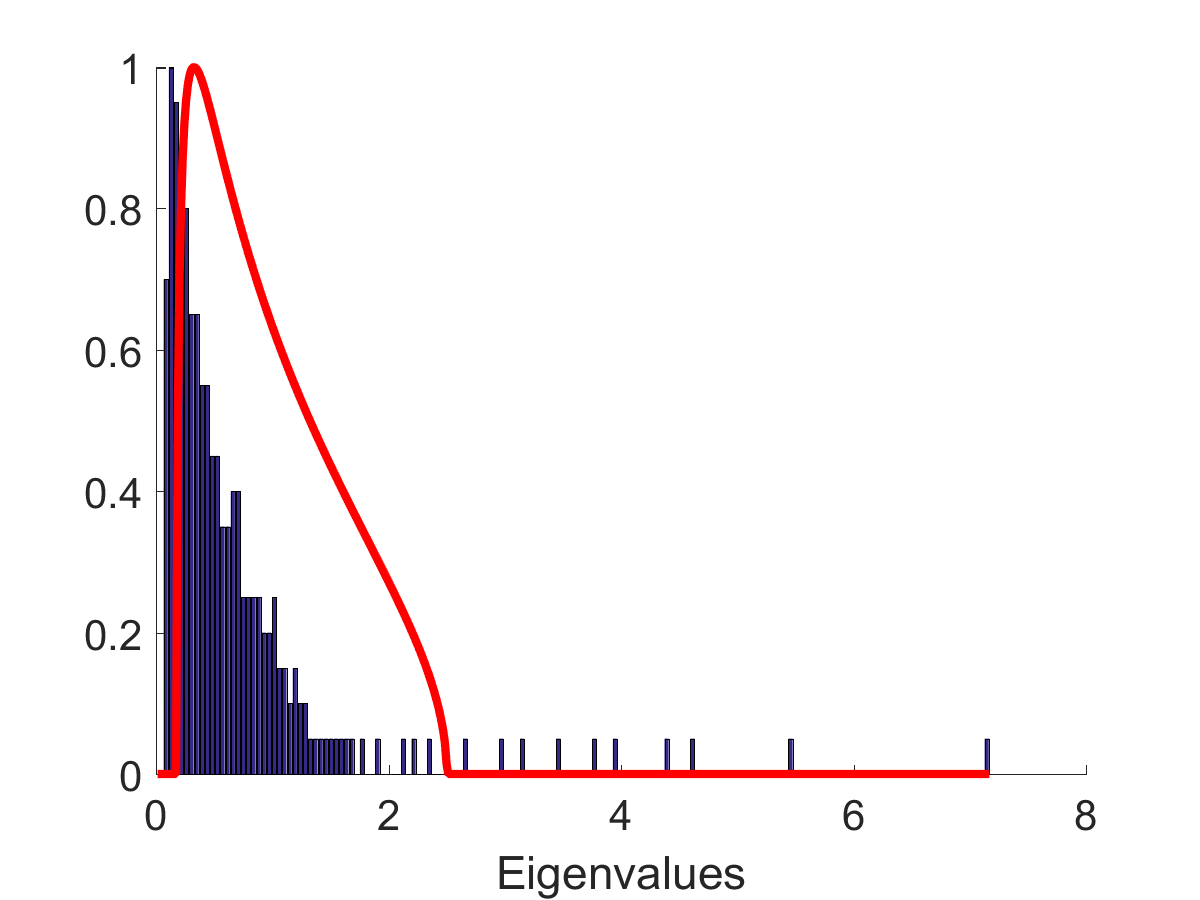} \\
\includegraphics[width=\FW, trim = \TRA mm \TRB mm \TRC mm \TRD mm, clip = TRUE]{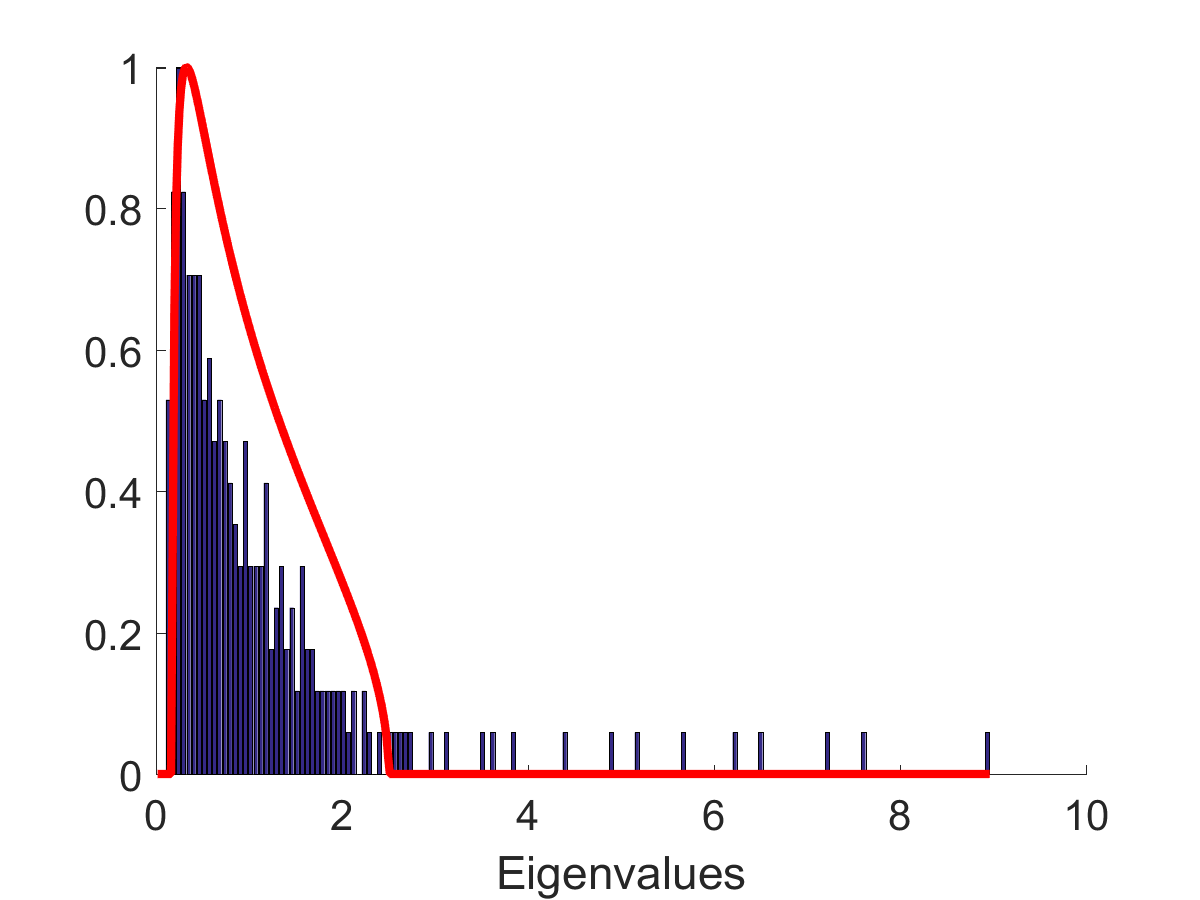}  &
\includegraphics[width=\FW, trim = \TRA mm \TRB mm \TRC mm \TRD mm, clip = TRUE]{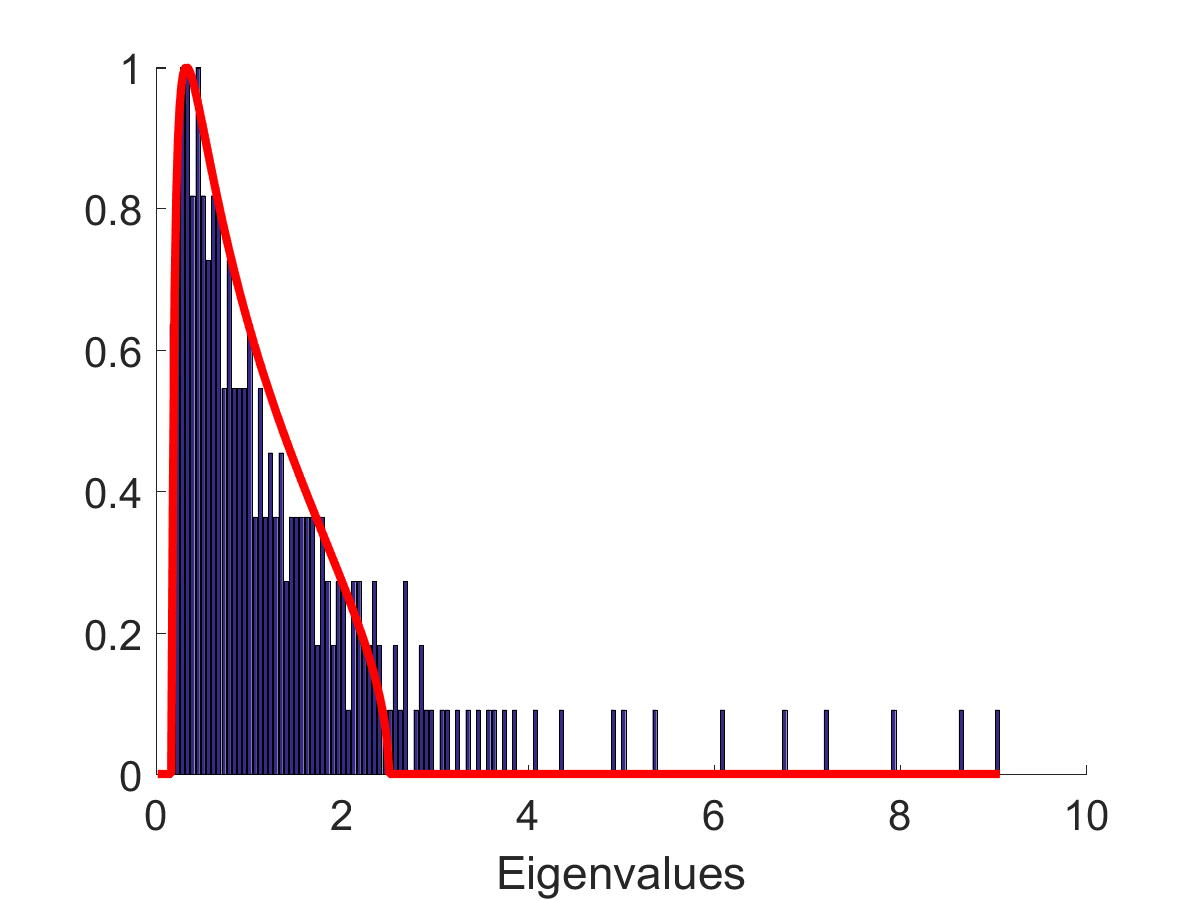}   \\

\includegraphics[width=\FW, trim = \TRA mm \TRB mm \TRC mm \TRD mm, clip = TRUE]{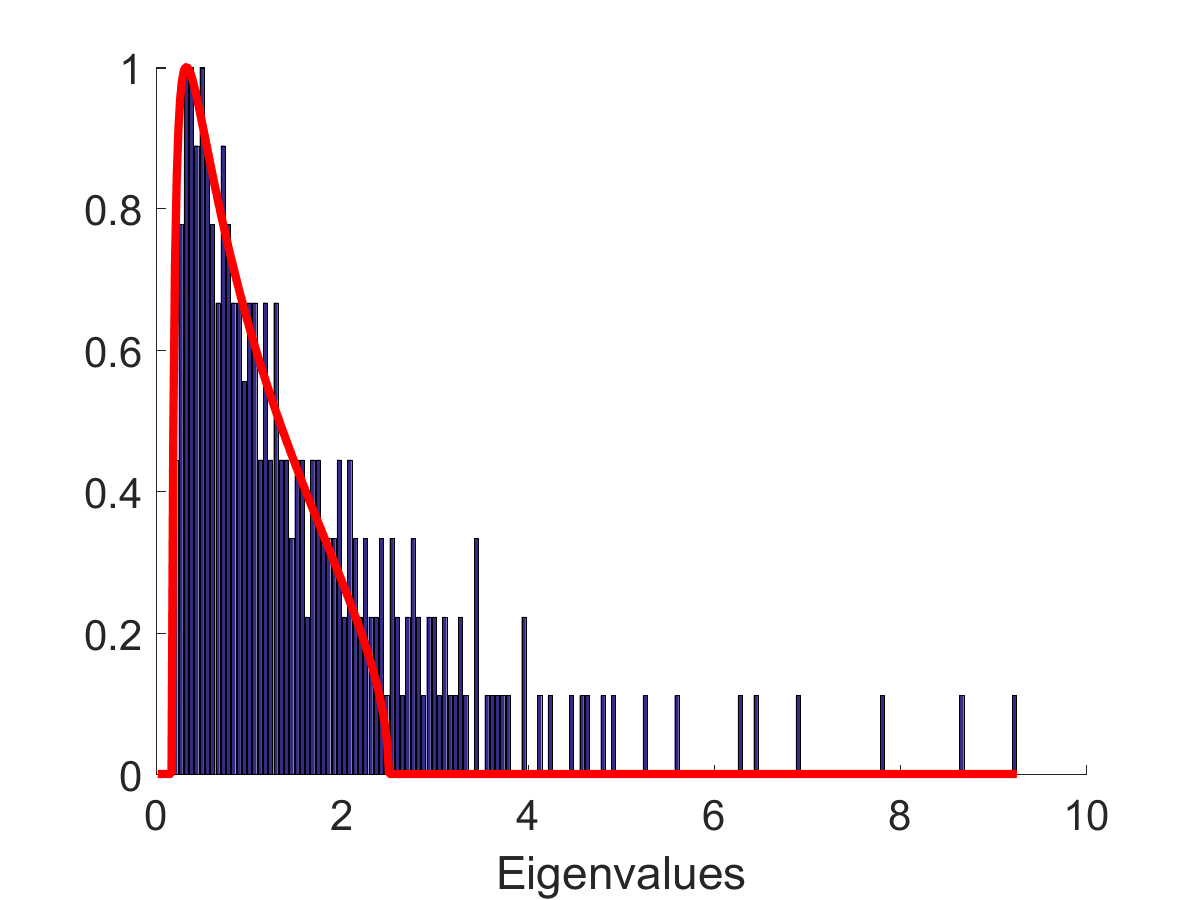} &
\includegraphics[width=\FW, trim = \TRA mm \TRB mm \TRC mm \TRD mm, clip = TRUE]{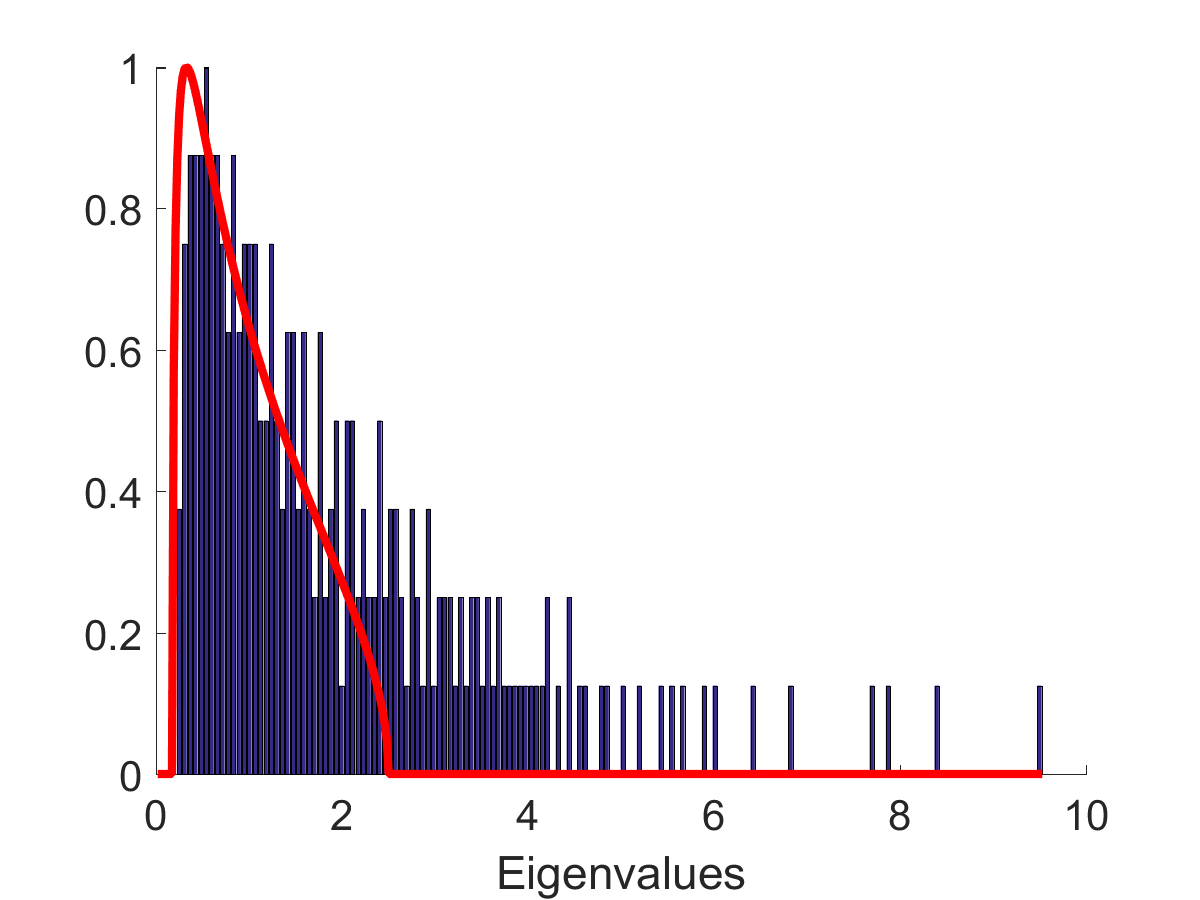}  \\
\end{tabular}
\caption{Histogram of eigenvalues smaller than 10 of the covariance matrix of the phoneme data, scaled by various $\sigma$. We first normalize the eigenvalues to have unit mean, and then multiply them by $\sigma$ on a uniform grid on $[0.5,5]$. The results are displayed moving from the top left image to the right, and continuing in the lower rows. Superimposed is the Marchenko-Pastur density.}
\label{fig:scale}
\end{figure}

\end{document}